%% file: 0main.tex
\documentclass{amsart}
\usepackage{amssymb,amsmath,graphicx,color}
\usepackage[all]{xy}
\usepackage{enumerate}
\usepackage{mathrsfs}

\theoremstyle{plain}
\newtheorem{theorem}{Theorem}[section]
\newtheorem{lemma}[theorem]{Lemma}
\newtheorem{corollary}[theorem]{Corollary}
\newtheorem{proposition}[theorem]{Proposition}

\theoremstyle{definition}
\newtheorem{definition}[theorem]{Definition}

\theoremstyle{remark}
\newtheorem{remark}[theorem]{Remark}

\newcommand\Aut{\operatorname{Aut}}
\newcommand\Inn{\operatorname{Inn}}
\newcommand\Out{\operatorname{Out}}
\newcommand\Hom{\operatorname{Hom}}
\newcommand\End{\operatorname{End}}
\newcommand\Mat{\operatorname{Mat}}
\newcommand\GL{\operatorname{GL}}
\newcommand\IA{\operatorname{IA}}

\newcommand\Sym{\operatorname{Sym}}
\newcommand\sym{\operatorname{sym}}
\newcommand\alt{\operatorname{alt}}
\newcommand\Ext{\operatorname{Ext}}

\newcommand\id{\operatorname{id}}

\newcommand\opegr{\operatorname{gr}}
\newcommand\Span{\operatorname{Span}}

\newcommand\gr{\mathrm{gr}}

\newcommand\op{\mathrm{op}}
\newcommand\ab{\mathrm{ab}}
\newcommand\sgn{\mathrm{sgn}}

\newcommand\Z{\mathbb{Z}}
\newcommand\N{\mathbb{N}}

\newcommand\K{\Bbbk}
\newcommand\R{\mathbb{R}}
\newcommand\repS{\mathbb{S}} 

\newcommand\A{\mathbf{A}} 
\newcommand\F{\mathbf{F}} 
\newcommand\FAb{\mathbf{FAb}} 
\newcommand\Grp{\mathbf{Grp}}
\newcommand\Ab{\mathbf{Ab}}
\newcommand\fVect{\mathbf{fVect}}
\newcommand\gVect{\mathbf{gVect}}
\newcommand\Vect{\mathbf{Vect}}

\newcommand\catC{\mathcal{C}} 
\newcommand\catD{\mathcal{D}}

\newcommand\jfA{\mathcal{A}}
\newcommand\jfE{\mathcal{E}}
\newcommand\g{\mathfrak{g}}

\newcommand\flA{\mathscr{A}}
\newcommand\flB{\mathscr{B}}
\newcommand\gpS{\mathfrak{S}} 

\newcommand\ti{\tilde}
\newcommand\wti{\widetilde}

\newcommand\centre[1]{\begin{array}{c} #1 \end{array}}

\title[Actions of automorphism groups of free groups on Jacobi diagrams]{Actions of automorphism groups of free groups on spaces of Jacobi diagrams. I}
\author{Mai Katada}
\address{Department of Mathematics, Kyoto University, Kyoto 606-8502, Japan}
\email{katada.mai.36s@st.kyoto-u.ac.jp}
\date{September 9, 2021 (First version: February 5, 2021)}

\begin{document}

\begin{abstract}
  We consider an action of the automorphism group $\Aut(F_n)$ of the free group $F_n$ of rank $n$ on the filtered vector space $A_d(n)$ of Jacobi diagrams of degree $d$ on $n$ oriented arcs.
  This action induces on the associated graded vector space of $A_d(n)$, which is identified with the space $B_d(n)$ of open Jacobi diagrams, an action of the general linear group $\GL(n,\Z)$ and an action of the graded Lie algebra of the IA-automorphism group of $F_n$ associated with its lower central series.
  We use these actions on $B_d(n)$ to study the $\Aut(F_n)$-module structure of $A_d(n)$.
  In particular, we consider the case where $d=2$ in detail and give an indecomposable decomposition of $A_2(n)$.
  We also construct a polynomial functor $A_d$ of degree $2d$ from the opposite category of the category of finitely generated free groups to the category of filtered vector spaces, which includes the $\Aut(F_n)$-module structure of $A_d(n)$ for all $n\ge 0$.
\end{abstract}

\maketitle
\setcounter{tocdepth}{1}
\tableofcontents

\section{Introduction} \label{intro}
 The \emph{Kontsevich integral} is a universal finite type invariant for links \cite{Kontsevich, BN_v}, which unifies all quantum invariants of links.
 The Kontsevich integral takes values in the space of \emph{Jacobi diagrams}, which are uni-trivalent graphs encoding the algebraic structures of Lie algebras and their representations.
 
 \emph{String links} and \emph{bottom tangles} are special kinds of tangles in a cube consisting of finitely many arc components.
 Since any links can be obtained by closing string links or bottom tangles, it is natural to consider the Kontsevich integral for string links \cite{HL,BN_sl} and bottom tangles \cite{Habiro_b}.

 The Kontsevich integral for $n$-component bottom tangles takes values in a vector space $A(n)$ of Jacobi diagrams on $n$ oriented arcs over a field $\K$ of characteristic $0$. 
 The degree $d$ part of $A(n)$, denoted by $A_d(n)$, has a filtration whose associated graded vector space is isomorphic to a $\K$-vector space $B_d(n)$ of \emph{open Jacobi diagrams} of degree $d$ colored by an element of an $n$-dimensional $\K$-vector space.
 
 Bar-Natan \cite{BN} studied an action of the symmetric group on the space of open Jacobi diagrams colored by distinct integers and computed irreducible decompositions for small degrees.
 
 Habiro and Massuyeau \cite{HM_k} extended the Kontsevich integral to construct a functor from the category of bottom tangles in handlebodies to the category of Jacobi diagrams in handlebodies. By using the restriction of this functor to the degree $0$ part, we construct a functor $A_d$ from the opposite category of the category of finitely generated free groups to the category of filtered vector spaces. 
 By restricting the functor $A_d$ to the automorphism group $\Aut(F_n)$ of the free group $F_n$ of rank $n$, we obtain an action of $\Aut(F_n)$ on the space $A_d(n)$.

 The $\Aut(F_n)$-action on $A_d(n)$ induces an action of the general linear group $\GL(n;\Z)$ on the space $B_d(n)$, which is an extension of the action of the symmetric group considered by Bar-Natan.
 Geometrically, the $\Aut(F_n)$-action on $A_d(n)$ can be interpreted as a restriction of an action of the handlebody group of genus $n$ on the set of $n$-component bottom tangles.

 The aim of the present paper is to give a way of studying the $\Aut(F_n)$-module structure of $A_d(n)$ and the functor $A_d$ and to study the case where $d=2$.
 The action of $\Aut(F_n)$ on $A_d(n)$ induces an action on $B_d(n)$ of the graded Lie algebra $\gr(\IA(n))$ of the \emph{IA-automorphism group} $\IA(n)$ of $F_n$ associated with its lower central series.
 We use an irreducible decomposition of the $\GL(n;\Z)$-module $B_d(n)$ and the $\gr(\IA(n))$-action on $B_d(n)$ to study the $\Aut(F_n)$-module structure of $A_d(n)$.
 In particular, we give an indecomposable decomposition of the $\Aut(F_n)$-module $A_2(n)$ and of the functor $A_2$.
 
 \subsection{The space $A_d(n)$ of Jacobi diagrams}\label{ss11}
  We work over a fixed field $\K$ of characteristic $0$. We consider here the $\K$-vector space $A_d(n)$ of Jacobi diagrams on oriented arcs, which is the main object of the present paper.

  For $n\geq 0$, let $X_n=\centre{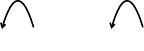}$ denote the oriented $1$-manifold consisting of $n$ arc components.
  The $\K$-vector space $A_d(n)$ is spanned by Jacobi diagrams on $X_n$ of degree $d$ modulo the STU relations. Here the degree of Jacobi diagrams is defined to be half the number of vertices as usual.
  (See Section \ref{ssJacobi} for further details.)

  We consider a filtration for $A_d(n)$
  $$
    A_d(n)=A_{d,0}(n)\supset A_{d,1}(n)\supset\cdots\supset A_{d,2d-2}(n)\supset A_{d,2d-1}(n)=0,
  $$
  such that $A_{d,k}(n)\subset A_d(n)$ is the subspace spanned by Jacobi diagrams with at least $k$ trivalent vertices. Hence, $A_d(n)$ is a filtered vector space.
  For example,
  $$\centre{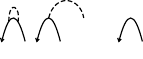}\:\in A_{2,0}(n),\:\centre{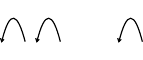}\:\in A_{2,1}(n),\:\centre{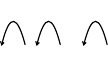}\:\in A_{2,2}(n).$$

 \subsection{A functor $A_d$ and an $\Aut(F_n)$-action on $A_d(n)$}\label{ss12}
  We construct a functor
  $$A_d:\F^{\op}\rightarrow \fVect$$
  from the opposite category $\F^{\op}$ of the category $\F$ of finitely generated free groups to the category $\fVect$ of filtered vector spaces over $\K$, which maps $F_n$ to the filtered vector space $A_d(n)$.

  \begin{proposition}[see Proposition \ref{poly}]\label{polynomial}
  The functor $A_d$ is a polynomial functor of degree $2d$.
  \end{proposition}
  Polynomial degrees measure complexity of functors. 
  Eilenberg and Mac Lane \cite{EM} introduced polynomial functors on finitely generated free modules over a ring in the study of the homology of the Eilenberg--MacLane spaces.
  Hartl--Pirashvili--Vespa \cite{HPV} considered polynomial functors on finitely generated free groups.
  Proposition \ref{polynomial} gives a new family of polynomial functors on finitely generated free groups.
  
  The functor $A_d$ gives a map
  $$\Hom(F_m,F_n)\times A_d(n)\rightarrow A_d(m)$$
  for $m,n\geq 0$. (See Section \ref{ss32} for the definition.)
  For example, for an element $f\in \Hom(F_2,F_3)$ defined by $f(x_1)=x_1x_2$, $f(x_2)=x_2x_3$, we have
  $$
    f \cdot \centre{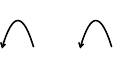}=\centre{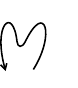}
    =\centre{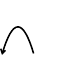}+\centre{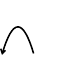}+\centre{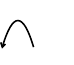}+\centre{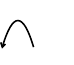}.
  $$

  When $n=m$, we have an action of the opposite $\End(F_n)^{\op}$ of the endomorphism monoid $\End(F_n)$ on $A_d(n)$
  $$\End(F_n)^{\op}\times A_d(n)\rightarrow A_d(n).$$
  By restricting this action to the opposite $\Aut(F_n)^{\op}$ of the automorphism group $\Aut(F_n)$, we obtain a right action of $\Aut(F_n)$ on $A_d(n)$.
  In fact, we have the following.
  \begin{theorem}[see Theorem \ref{out}]
   The $\Aut(F_n)$-action on $A_d(n)$ induces an action on $A_d(n)$ of the outer automorphism group $\Out(F_n)$ of $F_n$. 
  \end{theorem}
  
 \subsection{A functor $B_d$ and a $\GL(n;\Z)$-action on $B_d(n)$}\label{ss13}
  The associated graded vector space $\gr(A_d(n))$ of the filtered vector space $A_d(n)$ can be identified via the PBW map \cite{BN_v, BN_sl} with the vector space $B_d(n)$ of colored open Jacobi diagrams, which we explain below. The $\Aut(F_n)$-action on $A_d(n)$ induces an action of $\GL(n;\Z)$ on $B_d(n)$.
  
  For $n\geq 0$, let $V_n=\bigoplus_{i=1}^n\K v_i$ be an $n$-dimensional $\K$-vector space.
  The $\K$-vector space $B_d(n)$ is spanned by $V_n$-colored open Jacobi diagrams of degree $d$ modulo the AS, IHX and multilinearity relations, where ``$V_n$-colored'' means that each univalent vertex is colored by an element of $V_n$.
  (See Section \ref{ss33} for further details.)
  We consider a grading for $B_d(n)$ such that the degree $k$ part $B_{d,k}(n)\subset B_d(n)$ is spanned by open Jacobi diagrams with exactly $k$ trivalent vertices.
  For example,
  $$\centre{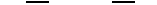}\;\in B_{2,0}(n),\;\centre{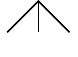}\;\in B_{2,1}(n),\;\centre{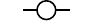}\;\in B_{2,2}(n).$$
  
  We construct a functor
  $$B_d:\FAb^{\op}\rightarrow \gVect$$
  from the opposite category $\FAb^{\op}$ of the category $\FAb$ of finitely generated free abelian groups to the category $\gVect$ of graded vector spaces over $\K$, which maps each object $\Z^n$ of $\FAb^{\op}$ to the graded vector space $B_d(n)$.
  Here, we have a map
  $$
    \Hom(\Z^m,\Z^n)\times B_d(n)\rightarrow B_d(m)
  $$
  for $m,n\geq 0$, which is given by matrix multiplication on each coloring.
  (See Section \ref{ss33} for the definition.)

  By restricting this map to the opposite group $\GL(n;\Z)^{\op}$ of $\GL(n;\Z)$, we obtain a right action of $\GL(n;\Z)$ on $B_d(n)$.
  Moreover, the $\GL(n;\Z)$-action on $B_d(n)$ extends to an action of $\GL(n;\K)\cong \GL(V_n)$ on $B_d(n)$.
  
  Let $D^c_{d,k}$ be the $\K$-vector space spanned by \emph{connected} open Jacobi diagrams of degree $d$ with $2d-k$ univalent vertices each of which is colored by a different element of $\{1,\cdots,2d-k\}$.
  The symmetric group $\gpS_{2d-k}$ acts on the space $D^c_{d,k}$. Bar-Natan \cite{BN} computed the dimensions and irreducible decompositions of the $\gpS_{2d-k}$-modules $D^c_{d,k}$ for $d\le 7$.
  By using irreducible decompositions of $D^c_{d,k}$, we obtain an irreducible decomposition of the $\GL(n;\Z)$-module $B_d(n)$, which we use to study the $\Aut(F_n)$-module structure of $A_d(n)$.

  Recall that for any partition $\lambda$ of $N\geq 0$ with at most $n$ rows, the Schur functor $\repS_{\lambda}$ gives a simple $\GL(V_n)$-module $\repS_{\lambda}V_n=V_n^{\otimes N}\cdot c_{\lambda}$, where $c_{\lambda}\in \K\gpS_N$ is the Young symmetrizer corresponding to $\lambda$.
  \begin{proposition}[see Propositions \ref{B1} and \ref{prop531}]
   We have irreducible decompositions of $\GL(V_n)$-modules
  \begin{gather}\label{eqB1}
      B_1(n)=B_{1,0}(n)\cong \repS_{(2)}V_n,
  \end{gather}
  \begin{gather}\label{eqB2}
    \begin{split}
      B_2(n)&=B_{2,0}(n)\oplus B_{2,1}(n)\oplus B_{2,2}(n)\\
      &\cong (\repS_{(4)}V_n\oplus \repS_{(2,2)}V_n)\oplus \repS_{(1,1,1)}V_n\oplus \repS_{(2)}V_n.
    \end{split}
  \end{gather}
  \end{proposition}

  We observe that the functor $A_d$ induces the functor $B_d$.
  Let $\ab:\F\rightarrow \FAb$ denote the abelianization functor,
  and $\ab^{\op}:\F^{\op}\rightarrow \FAb^{\op}$ its opposite functor.
  Let $\gr:\fVect\rightarrow \gVect$ denote the functor that sends a filtered vector space to its associated graded vector space.
  \begin{proposition}[see Proposition \ref{prop341}]\label{proptheta}
   For $d\geq 0$, there is a natural isomorphism $$\theta_d:\gr\circ A_d\overset{\cong}\Rightarrow B_d\circ\ab^\op.$$
  \end{proposition}

 \subsection{The functor $A_1$}\label{ss14}
  Here we consider the functors $A_1$ and $B_1$.
  By Proposition \ref{proptheta}, we have $A_1\cong B_1\circ \ab^{\op}$.
  Since we have isomorphisms of $\Aut(F_n)$-modules
  $$A_1(n)\cong B_1(n)\cong \repS_{(2)}V_n$$
  by (\ref{eqB1}), the $\Aut(F_n)$-module $A_1(n)$ is simple for any $n\geq 1$.
  It follows that the functor $A_1$ is indecomposable.

 \subsection{An action of $\gr(\IA(n))$ on the space $B_d(n)$}\label{ss15}
  Let $\IA(n)$ denote the IA-automorphism group of $F_n$, which is the kernel of the canonical homomorphism $\Aut(F_n)\rightarrow \Aut(H_1(F_n;\Z))\cong \GL(n;\Z)$.
  Let $\Gamma_{\ast}(\IA(n))=(\Gamma_{r}(\IA(n)))_{r\geq 1}$ denote the lower central series of $\IA(n)$, and
  $\gr(\IA(n))=\bigoplus_{r\geq 1}\opegr^r(\IA(n))$ the associated graded Lie algebra, where $\opegr^r(\IA(n))=\Gamma_{r}(\IA(n))/\Gamma_{r+1}(\IA(n))$.

  To study the $\Aut(F_n)$-module structure of $A_d(n)$, we use a right action of $\gr(\IA(n))$ on the graded vector space $B_d(n)\cong \gr(A_d(n))$.

  \begin{theorem}[see Proposition \ref{p612} and Theorem \ref{th531}]
    There is an action of the graded Lie algebra $\gr(\IA(n))$ on the graded vector space $B_{d}(n)$, which consists of $\GL(n;\Z)$-module homomorphisms
    $$B_{d,k}(n)\otimes_{\Z} \opegr^r(\IA(n)) \rightarrow B_{d,k+r}(n)$$
    for $k\geq 0$ and $r\geq 1$.
  \end{theorem}

 \subsection{$\Aut(F_n)$-module structure of $A_2(n)$ and indecomposable decomposition of $A_2$}\label{ss17}
  Here, we consider the right $\Aut(F_n)$-module structure of $A_2(n)$ and give an indecomposable decomposition of the functor $A_2$.
  
  We use the graphical notation $\centre{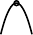}=\centre{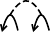}\in A_1(2)$.
  Set
  $$P'=\centre{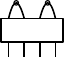},\quad P''=\centre{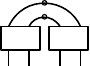}\in A_2(4),$$
  where $\sym_4$ corresponds to the Young symmetrizer $c_{(4)}$ and $\alt_2$ corresponds to the Young symmetrizer $c_{(1,1)}$. (See Section \ref{ss62} for further details.)
  Let
  $$A_2', A_2'':\F^\op\rightarrow\fVect$$
  be the subfunctors of the functor $A_2$ such that
  \begin{gather*}
      A_2'(n):=\Span_\K\{A_2(f)(P'):f\in\F^\op(4,n)\}\subset A_2(n),\\
      A_2''(n):=\Span_\K\{A_2(f)(P''):f\in\F^\op(4,n)\}\subset A_2(n).
  \end{gather*}

  We use the $\GL(V_n)$-module structure \eqref{eqB2} of $B_2(n)$ and the $\gr(\IA(n))$-action on $B_2(n)$ to study the $\Aut(F_n)$-module structure of $A_2(n)$.

  \begin{theorem}[see Proposition \ref{prop1021} and Theorem \ref{th1042}]\label{thintro}
   For $n\geq 3$, we have an indecomposable decomposition of $\Aut(F_n)$-modules
   $$
     A_2(n)=A_2'(n)\oplus A_2''(n).
   $$
   Here, $A_2'(n)$ is simple, and $A_2''(n)$ admits a unique composition series of length $3$
   $$A_2''(n)\supsetneq A_{2,1}(n)\supsetneq A_{2,2}(n)\supsetneq 0;$$
   that is, $A_2''(n)$ has no nonzero proper $\Aut(F_n)$-submodules other than $A_{2,1}(n)$ and $A_{2,2}(n)$.
   (For $n=1,2$, see Theorem \ref{th1042}.)
  \end{theorem}


  By using Theorem \ref{thintro}, we obtain an indecomposable decomposition of the functor $A_2$.
  \begin{theorem}[see Proposition \ref{prop1021} and Theorem \ref{th1041}]
   We have an indecomposable decomposition
    $$A_2=A_2'\oplus A_2''$$
    in the functor category $\fVect^{\F^{\op}}$.
  \end{theorem}
  In the subsequent paper \cite{Mai}, we will study the case where $d\geq 3$ to obtain an indecomposable decomposition and the radical filtration of $A_d(n)$. For $d\ge 3$, it is rather difficult to compute the $\gr(\IA(n))$-action on $B_d(n)$ directly. In order to simplify computation of the $\gr(\IA(n))$-action on $B_d(n)$, we will reconstruct the action in a different way.
  We will also study the \emph{Johnson filtration} $\jfE_{\ast}(n)$ of the endomorphism monoid $\End(F_n)$, which is an enlargement of the \emph{Johnson filtration} $\jfA_{\ast}(n)$ of $\Aut(F_n)$ and the lower central series $\Gamma_{\ast}(\IA(n))$ of $\IA(n)$.
  We will show that $\jfE_{\ast}(n)$ acts on $A_d(n)$ and therefore the extended N-series $\jfA_{\ast}(n)$ acts on the filtered vector space $A_d(n)$.

 \subsection{Organization of the paper}
  In Section \ref{pre}, we recall some notions and definitions about Jacobi diagrams, open Jacobi diagrams and the category of Jacobi diagrams in handlebodies.
  In Section \ref{s3}, we construct functors $A_d: \F^\op \rightarrow \fVect$ and $B_d: \FAb^\op \rightarrow \gVect$ and observe that $A_d$ induces $B_d$.
  In Section \ref{s4}, we compute the functors $A_1$ and $B_1$ explicitly.
  In Section \ref{s5}, we define an action of the graded Lie algebra $\gr(\IA(n))$ on the graded vector space $B_d(n)$.
  In Section \ref{sBd}, we establish the notation about representations of $\GL(V_n)$ and consider the dimension of the $\K$-vector spaces $A_d(n)$ and $B_d(n)$.
  In Section \ref{s6}, we consider the $\Aut(F_n)$-module structure of $A_2(n)$ and give an indecomposable decomposition of $A_2$.
  In Section \ref{s7}, we consider the polynomiality of the functor $A_d$.
  In Appendix \ref{ss52}, we define an action of an extended N-series on a filtered vector space and an action of an extended graded Lie algebras on a graded vector space.

 \subsection{Acknowledgments}
  The author would like to thank Kazuo Habiro for careful reading and valuable advice,
  and Christine Vespa for letting us know that our functor $A_d$ is a polynomial functor and some relations between our study and their paper \cite{POWELLVESPA}.
  She also thanks Gw\'{e}na\"{e}l Massuyeau, Takefumi Nosaka and Sakie Suzuki for helpful comments.

\section{Preliminaries}\label{pre}
 In this section, we recall some notions of Jacobi diagrams and open Jacobi diagrams and the category $\A$ of Jacobi diagrams in handlebodies. In what follows, we work over a fixed field $\K$ of characteristic $0$.

 \subsection{Jacobi diagrams and open Jacobi diagrams}\label{ssJacobi}
  In this section, we recall Jacobi diagrams and open Jacobi diagrams defined in \cite{BN_v}, \cite{BN_sl} and \cite{Ohtsuki}.

  A \emph{uni-trivalent graph} is a finite graph whose vertices are either univalent or trivalent. A trivalent vertex is \emph{oriented} if it has a fixed cyclic order of the three edges around it. A \emph{vertex-oriented} uni-trivalent graph is a uni-trivalent graph such that each trivalent vertex is oriented.

  For $n\geq 0$, let $X_n$ be the oriented $1$-manifold consisting of $n$ arc components as depicted in Figure \ref{Fig1}.
  \begin{figure}[h]
      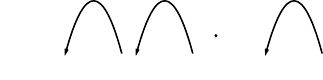
      \caption{The oriented $1$-manifold $X_n$.}
      \label{Fig1}
  \end{figure}

  A \emph{Jacobi diagram} on $X_n$ is a vertex-oriented uni-trivalent graph such that univalent vertices are embedded into the interior of $X_n$ and each connected component has at least one univalent vertex.
  Two Jacobi diagrams $D$ and $D'$ on $X_n$ are regarded as the same if there is a homeomorphism $f:D\cup X_n\rightarrow D'\cup X_n$ whose restriction to $X_n$ is isotopic to the identity map of $X_n$.
  In figures, we depict $X_n$ as solid lines and Jacobi diagrams as dashed lines in such a way that each trivalent vertex is oriented in the counterclockwise order.

  Let $\flA(X_n)$ denote the $\K$-vector space spanned by Jacobi diagrams on $X_n$ modulo the STU relation, which is described in Figure \ref{Fig2}.
  \begin{figure}[ht]
     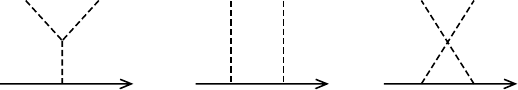
     \caption{The STU relation.}
     \label{Fig2}
  \end{figure}

  The \emph{degree} of a Jacobi diagram is defined to be half the number of its vertices. Since the STU relation is homogeneous with respect to the degree, we have a grading
  $$
      \flA(X_n)=\bigoplus_{d\geq 0}\flA_d(X_n),
  $$
  where $\flA_d(X_n)\subset \flA(X_n)$ is the subspace spanned by Jacobi diagrams of degree $d$.

  For $k\geq 0$, let $\flA_{d,k}(X_n)\subset \flA_d(X_n)$ be the subspace spanned by Jacobi diagrams with at least $k$ trivalent vertices.
  We have $\flA_0(X_n)=\flA_{0,0}(X_n)\cong\K$ for $d=0$.
  For $d\geq 1$, we have a filtration
  $$
      \flA_d(X_n)=\flA_{d,0}(X_n)\supset \flA_{d,1}(X_n)\supset \flA_{d,2}(X_n)\supset\cdots\supset \flA_{d,2d-1}(X_n)=0.
  $$
  Note that we have $\flA_{d,2d-1}(X_n)=0$ since a Jacobi diagram on $X_n$ with only one univalent vertex vanishes by using the STU relations.
  We consider the graded vector space $\gr(\flA_d(X_n)):=\bigoplus_{k\geq 0} \opegr^k(\flA_d(X_n))$ associated to the above filtration $\flA_{d,\ast}(X_n)$, where $\opegr^k(\flA_d(X_n)):=\flA_{d,k}(X_n)/\flA_{d,k+1}(X_n)$.

  An \emph{open Jacobi diagram} is a vertex-oriented uni-trivalent graph such that each connected component has at least one univalent vertex.

  Let $T$ be a set.
  A \emph{$T$-colored open Jacobi diagram} is an open Jacobi diagram such that each univalent vertex is colored by an element of $T$.
  In figures, we depict $T$-colored open Jacobi diagrams as solid lines in such a way that each trivalent vertex is oriented in the counterclockwise order.

  Let $\flB(T)$ denote the $\K$-vector space spanned by $T$-colored open Jacobi diagrams modulo the AS and IHX relations, which are depicted in Figure \ref{Fig3}.
  \begin{figure}[ht]
     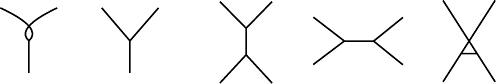
     \caption{The AS and IHX relations.}
     \label{Fig3}
  \end{figure}

  The \emph{degree} of a $T$-colored open Jacobi diagram is defined to be half the number of vertices. Since the AS and IHX relations are homogeneous with respect to the degree, we have a grading
  $$
      \flB(T)=\bigoplus_{d\geq 0}\flB_d(T),
  $$
  where $\flB_d(T)\subset\flB(T)$ is the subspace spanned by $T$-colored open Jacobi diagrams of degree $d$.

  For $k\geq 0$, let $\flB_{d,k}(T)\subset \flB_d(T)$ be the subspace spanned by open Jacobi diagrams with exactly $k$ trivalent vertices.
  We have $\flB_0(T)=\flB_{0,0}(T)=\K\emptyset$ for $d=0$.
  For $d\geq 1$, we have
  $$
    \flB_d(T)=\bigoplus_{k=0}^{2d-2} \flB_{d,k}(T).
  $$
  Note that $\flB_{d,k}(T)=0$ for $k\geq 2d$ since an open Jacobi diagram has at least one univalent vertex and for $k=2d-1$ since an open Jacobi diagram with only one univalent vertex vanishes by using the AS and IHX relations.

  We consider the case where the coloring set is $[n]:= \{1,\cdots, n\}\subset\N$. Bar-Natan \cite{BN_v, BN_sl} proved that $\flA(X_n)$ is isomorphic to $\flB([n])$.
  This is a diagrammatic interpretation of the Poincar\'{e}--Birkhoff--Witt theorem.
  \begin{proposition}[PBW theorem \cite{BN_v, BN_sl})\label{PBW}]
   For $d\geq 0$, we have an isomorphism of vector spaces
   $$\chi_d: \flB_d([n])\xrightarrow{\cong} \flA_d(X_n).$$
   If $D\in\flB([n])$ is an $[n]$-colored open Jacobi diagram of degree $d$ such that for any $i\in[n]$, $D$ has $k_i$ univalent vertices colored by $i$, then $\chi(D)\in\flA(X_n)$ is the average of the $\prod_{i\in[n]} (k_i)!$ ways of attaching the univalent vertices colored by $i$ to the $i$-th component of $X_n$.

   Moreover, the map $\chi_d$ induces an isomorphism
   $$\chi_{d,k}:\flB_{d,k}([n])\xrightarrow{\cong}\opegr^k(\flA_{d,\ast}(X_n)).$$
  \end{proposition}
   Note that two Jacobi diagrams of $\flA_{d,k}(X_n)$ appearing in the average of the $\prod_{i\in[n]}(k_i)!$ ways are equivalent in the quotient space $\opegr^{k}(\flA_{d,\ast}(X_n))$ by the STU relations. Therefore, the average of the $\prod_{i\in[n]} (k_i)!$ ways of attaching univalent vertices coincides with an arbitrary way of attaching them in $\opegr^{k}(\flA_{d,\ast}(X_n))$.

 \subsection{The category $\A$ of Jacobi diagrams in handlebodies}\label{sscatA}
  Here we briefly review the category $\A$ of Jacobi diagrams in handlebodies defined in \cite{HM_k}.

  The objects in $\A$ are nonnegative integers. To define the hom-set $\A(m,n)$, we need the notion of $(m,n)$-Jacobi diagrams, which we explain below.

  Let $I=[-1,1]$.
  For $m\geq 0$, let $U_m \subset \R^3$ denote the handlebody of genus $m$ that is obtained from the cube $I^3$ by attaching $m$ handles on the top square $I^2 \times \{1\}$ as depicted in Figure \ref{Fig4}.
  We call $l := I \times \{0\} \times \{-1\}$ the \emph{bottom line} of $U_m$.
  We call $S := I^2 \times \{-1\}$ the \emph{bottom square} of $U_m$.
  For $i=1,\cdots, m$, let $x_i$ be a loop which goes through only the $i$-th handle of the handlebody $U_m$ just once and let $x_i$ denote its homotopy class as well. In what follows, for loops $\gamma_1$ and $\gamma_2$ with base points on $l$, let $\gamma_2\gamma_1$ denote the loop that goes through $\gamma_1$ first and then goes through $\gamma_2$. That is, we write a product of elements of the fundamental group of $U_m$ in the opposite order to the usual one.
  Let $\bar{x}_i\in H_1(U_m;\K)$ denote the homology class of $x_i$. We have $\pi_1(U_m)=\langle x_1,\cdots,x_m\rangle$ and $H_1(U_m;\K)=\bigoplus_{i=1}^m \K\bar{x}_i$.
  \begin{figure}[ht]
      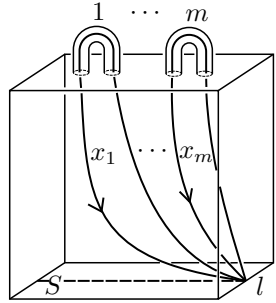
      \caption{The handlebody $U_m$.}
      \label{Fig4}
  \end{figure}

  For $m,n\geq 0$, an \emph{$(m,n)$-Jacobi diagram} $(D,f)$ consists of a Jacobi diagram $D$ on $X_n$ and a map $f: X_n\cup D \rightarrow U_m$ which maps $\partial X_n$ into the bottom line $l$ of $U_m$ in such a way that the endpoints of $X_n$ are uniformly distributed and that for $i=1,\cdots,n$, the $i$-th arc component of $X_n$ goes from the $2i$-th point to the $2i-1$-st point, where we count the endpoints from left to right. In what follows, we simply write $D$ for an $(m,n)$-Jacobi diagram.
  We identify two $(m,n)$-Jacobi diagrams if they are homotopic in $U_m$ relative to the endpoints of $X_n$. Figure \ref{(m,n)-Jacobi diagram} shows a $(2,3)$-Jacobi diagram $D$.
  \begin{figure}[ht]
      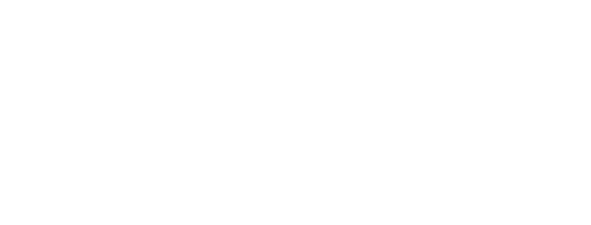
      \caption{A $(2,3)$-Jacobi diagram.}
      \label{(m,n)-Jacobi diagram}
  \end{figure}
  For $m,n\geq 0$, the hom-set $\A(m,n)$ is the $\K$-vector space spanned by $(m,n)$-Jacobi diagrams modulo the STU relations.
  We usually depict $(m,n)$-Jacobi diagrams by drawing their images under the orthogonal projection of $\R^3$ onto $\R\times \{0\}\times \R$.

  In order to define the composition in the category $\A$, we use the \emph{box notation} as depicted in Figure \ref{box}. (See Example 3.2 of \cite{HM_k}.)
  Dashed and solid lines are allowed to go through the box, and a dashed line is attached to the left side of the box. 
  The box notation represents a sum with sign of each Jacobi diagram which is obtained by attaching the univalent vertex of the dashed line to each line which goes through the box.
  The sign of a summand corresponding to a solid line is determined by the compatibility of its orientation with the direction of the box, and the sign of a summand corresponding to a dashed line is determined to be positive.
  We also define the box notation with a dashed line attached to the right side of the box by the box notation with the dashed line attached to the left side of the box as depicted in Figure \ref{box}.
 
  \begin{figure}[ht]
      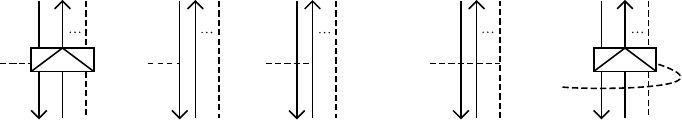
      \caption{The box notation.}
      \label{box}
  \end{figure}

  For $D:m\rightarrow n$ and $D':p\rightarrow m$, the composition $D\circ D'$ is defined as follows.
  By using isotopies of $U_m$, we can transform $D$ into an $(m,n)$-Jacobi diagram $\ti{D}$ each of whose handle has only solid and dashed lines parallel to the handle core.
  The composition $D\circ D'$ is obtained by stacking on the top of the square part of $\ti{D}$ a suitable cabling of $D'$. Here, the cabling is obtained from $D'$ by replacing each component of $X_m$ with its parallel copies so that the target of the cabling matches the source of $\ti{D}$, and each univalent vertex is replaced by the box notation.
  Figure \ref{composition} shows the composition $D\circ D'$ of the $(2,3)$-Jacobi diagram $D$, which is given in Figure \ref{(m,n)-Jacobi diagram}, and the following $(3,2)$-Jacobi diagram $D'$.
  \begin{figure}[h]
      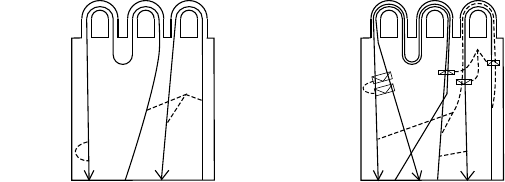
      \caption{The composition $D\circ D'$.}
      \label{composition}
  \end{figure}

  The identity morphism of an object $n$ is $\centre{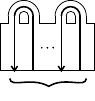}$.

  The \emph{degree} of an $(m,n)$-Jacobi diagram is the degree of its Jacobi diagram.
  Let $\A_d(m,n)\subset \A(m,n)$ be the subspace spanned by $(m,n)$-Jacobi diagrams of degree $d$.
  We have $\A(m,n)=\bigoplus_{d\geq0} \A_d(m,n)$.
  Note that we have $$\A_d(0,n)\cong \flA_d(X_n).$$

  The category $\A$ has a structure of a linear symmetric strict monoidal category. See \cite{Kassel} for the definition of symmetric strict monoidal categories.
  The tensor product on objects is addition. The monoidal unit is $0$. The tensor product on morphisms is juxtaposition followed by horizontal rescaling and relabelling of indices. For example, Figure \ref{tensor} shows the tensor product of a $(1,1)$-Jacobi diagram and a $(2,2)$-Jacobi diagram.
  \begin{figure}[h]
      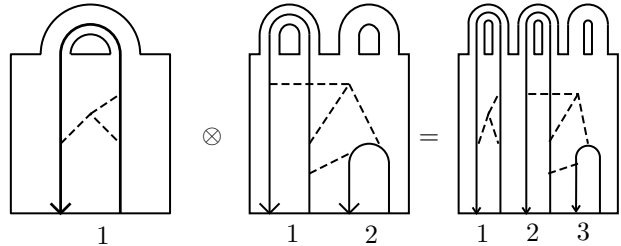
      \caption{The tensor product.}
      \label{tensor}
  \end{figure}
  The symmetry is determined by $P_{1,1}: 2\rightarrow 2$ which is depicted in Figure \ref{symm}.
  \begin{figure}[ht]
      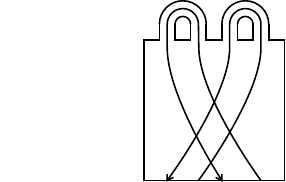
      \caption{The symmetry.}
      \label{symm}
  \end{figure}



\section{Functors $A_d$ and $B_d$}\label{s3}
 In this section, we define a functor $A_d: \F^\op \rightarrow \fVect$ from the opposite category $\F^\op$ of the category $\F$ of finitely generated free groups to the category $\fVect$ of filtered vector spaces over $\K$.
 We define another functor $B_d: \FAb^\op \rightarrow \gVect$ from the opposite category $\FAb^\op$ of the category $\FAb$ of finitely generated free abelian groups to the category $\gVect$ of graded vector spaces over $\K$.
 We prove that the functor $A_d$ induces the functor $B_d$.

 \subsection{The categories $\F$, $\FAb$, $\fVect$ and $\gVect$}\label{ss31}
  Let us start with the definitions of the categories $\F$, $\FAb$, $\fVect$ and $\gVect$.

  For $n\geq 0$, let $F_n =\langle x_1, \cdots, x_n\rangle$ be the free group of rank $n$.
  The category $\F$ of finitely generated free groups is the full subcategory of the category $\Grp$ of groups such that the class of objects is $\{F_n: n\geq0\}$. We identify the object $F_n$ with the integer $n$. Thus, $\F(n,m)=\Hom(F_n,F_m)\cong F_m^n$.
  The category $\F$ is a symmetric strict monoidal category.

  The category $\FAb$ of finitely generated free abelian groups is the full subcategory of the category $\Ab$ of abelian groups such that the class of objects is $\{\Z^n: n\geq 0\}$. We identify the object $\Z^n$ with the integer $n$. Thus, $\FAb(n,m)=\Hom(\Z^n,\Z^m)\cong\Mat(m,n;\Z)$.
  The category $\FAb$ is also a symmetric strict monoidal category.

  Let $\ab: \F\rightarrow \FAb$ denote the restriction of the abelianization functor $\ab: \Grp\rightarrow \Ab$.
  Here the functor $\ab$ maps $F_n$ to its abelianization $\ab(F_n)=F_n/[F_n,F_n]\cong\bigoplus_{i=1}^n \Z[x_i]$, which is naturally identified with $\Z^n$.
  In the following sections, we use the opposite functor $\ab^{\op}: \F^{\op}\rightarrow \FAb^{\op}$.

  Let $\fVect$ denote the category of filtered vector spaces and filter-preserving morphisms.
  A filtered vector space is a $\K$-vector space $V$ with a decreasing sequence of vector spaces $V=V_0\supset V_1\supset\cdots$.

  Let $\gVect$ denote the category of graded vector spaces and degree-preserving morphisms.
  A graded vector space is a $\K$-vector space $W=\bigoplus_{d\geq 0}W_d$.

  For a filtered vector space $V$, set $\opegr^{d}(V):= V_d/V_{d+1}$ for $d\geq 0$. We call $\gr(V):=\bigoplus_{d\geq 0}\opegr^{d}(V)$ the associated graded vector space of $V$.
  Let $\gr: \fVect\rightarrow \gVect$ be the functor that sends a filtered vector space $V$ to the associated graded vector space $\gr(V)$ and a filter-preserving morphism $f:V\rightarrow W$ to a degree-preserving morphism $\gr(f):\gr(V)\rightarrow \gr(W)$ defined by $\gr(f)([v]_{V_{d+1}})=[f(v)]_{W_{d+1}}$ for $v\in V_d$.

 \subsection{The functor $A_d: \F^\op \rightarrow \fVect$}\label{ss32}
  We define a functor $A_d: \F^\op \rightarrow \fVect$.

  Let $d, n\geq 0$. Set
  $$A_d(n):= \A_d(0,n)\cong\flA_d(X_n).$$
  For $k\geq 0$, let $A_{d,k}(n)\subset A_d(n)$ be the subspace spanned by Jacobi diagrams with at least $k$ trivalent vertices. We have an isomorphism
  $$A_{d,k}(n)\cong\flA_{d,k}(X_n).$$
  Thus, we have $A_0(n)=A_{0,0}(n)\cong\K$.
  For $d\geq 1$, we have a filtration
  $$A_d(n)=A_{d,0}(n)\supset A_{d,1}(n)\supset A_{d,2}(n)\supset\cdots\supset A_{d,2d-1}(n)=0.$$

  Let $\K\F$ be the $\K$-\emph{linearization} of the category $\F$.
  Here, the class of objects in $\K\F$ is the same as that in $\F$ and the hom-set $\K\F(m,n)$ is the $\K$-vector space spanned by all of the morphisms $m\rightarrow n$ in $\F$ for $m,n\geq 0$.
  We have an isomorphism $\K\F^\op(m,n)\xrightarrow{\cong}\A_0(m,n)$ of $\K$-vector spaces (see Section 1.5 of \cite{HM_k}).
  Note that
  \begin{gather*}
    \begin{split}
        \A_0(m,n)&=\K\{(m,n)\text{-Jacobi diagrams with empty Jacobi diagram}\}\\
                 &=\K\{\text{homotopy classes of maps } X_n\rightarrow U_m \text{ relative to the boundary}\}.
    \end{split}
  \end{gather*}
  For a map $f:X_n\rightarrow U_m$ such that $f(\partial X_n)\subset l$, let $\ti{f}=f \cup \id_l:X_n\cup\; l\rightarrow U_m$ and $\ti{f}_{\ast}:\pi_1(X_n\cup l)\cong F_n\rightarrow\pi_1(U_m)\cong F_m$ be the induced map on the fundamental groups. The linear map $\A_0(m,n)\rightarrow\K\F^\op(m,n)$ that sends $f$ to $\ti{f}_{\ast}$
  is an isomorphism.

  We define a map
  $$A_d: \F^\op(m,n) \rightarrow \fVect(A_d(m),A_d(n))$$ by
  $$A_d: \F^\op(m,n)\hookrightarrow \K\F^\op(m,n)\xrightarrow{\cong}\A_0(m,n)\xrightarrow{\circ}\fVect(A_d(m),A_d(n)),
  $$
  where the last map is the composition in the category $\A$ and we recall that $A_d(m)=\A_d(0,m)\subset \A(0,m)$. 
  Note that since any element of $\A_0(m,n)$ has an empty Jacobi diagram, the composition of an element of $\A_0(m,n)$ with an element of $A_d(m)$ preserves the filtration.
  It can be easily checked that $A_d$ is a functor.

 \subsection{The functor $B_d: \FAb^\op \rightarrow \gVect$}\label{ss33}
  In this section, we define a functor $B_d: \FAb^\op \rightarrow \gVect$.

  Let
  $$V_n:= H^{1}(U_n;\K)=\Hom(H_1(U_n;\K),\K)$$
  and let $\{v_i\}$ denote the dual basis of $\{\bar{x}_i\}$. We fix the basis $\{v_i\}$ for $V_n$ and we have $V_n= \bigoplus_{i=1}^n \K v_i$.

  Let $B_d(n)$ denote the $\K$-vector space spanned by $V_n$-colored open Jacobi diagrams of degree $d$ modulo the AS, IHX and multilinearity relations, where the multilinearity relation is shown in Figure \ref{multi}.
  Since $V_n=\bigoplus_{i=1}^n \K v_i$, the space $B_d(n)$ is isomorphic to the space $\flB_d([n])$ defined in Section \ref{ssJacobi}.
  \begin{figure}[h]
    $\centre{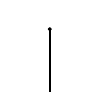}
    =a\:\centre{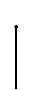}+b\:\centre{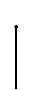}$
    for $a,b\in\K,w_1,w_2\in V_n$.
    \caption{Multilinearity.}
    \label{multi}
  \end{figure}

  For $k\geq 0$, let $B_{d,k}(n)\subset B_d(n)$ be the subspace spanned by open Jacobi diagrams with exactly $k$ trivalent vertices.
  We have an isomorphism $$B_{d,k}(n)\cong \flB_{d,k}([n]).$$
  Thus, we have $B_0(n)=B_{0,0}(n)=\K\emptyset$.
  For $d\geq 1$, we have a grading
  $$
    B_d(n)=\bigoplus_{k=0}^{2d-2} B_{d,k}(n).
  $$

  Let $T$ be a finite set.
  A $T$-colored open Jacobi diagram $D$ is called \emph{special} if the map $\{\text{univalent vertices of $D$}\}\rightarrow T$ that gives the coloring of $D$ is a bijection.

  Define $D_{d,k}$ as the $\K$-vector space spanned by special $[2d-k]$-colored open Jacobi diagrams of degree $d$ modulo the AS and IHX relations.
  The space $D_{d,k}$ has an $\gpS_{2d-k}$ action given by the action on the colorings.
  Considering $V_n^{\otimes 2d-k}$ as a right $\gpS_{2d-k}$-module by the action which permutes the factors, we have an isomorphism
  \begin{equation}\label{eqBD}
    B_{d,k}(n)\cong V_n^{\otimes 2d-k}\otimes_{\K\gpS_{2d-k}} D_{d,k}.
  \end{equation}
  Thus, any element of $B_{d,k}(n)$ can be written in the form
  $$
    u(w_1,\cdots,w_{2d-k}):=(w_1\otimes\cdots\otimes w_{2d-k})\otimes u
  $$
  for $u\in D_{d,k}$ and $w_1,\cdots, w_{2d-k}\in V_n$.

  For $m,n\geq 0$, we define a map
  $$B_d: \FAb^\op(m,n)\rightarrow \gVect(B_d(m),B_d(n))$$ as follows.
  We consider an element of $\FAb^\op(m,n)=\Mat(m,n;\Z)$ as an $(m\times n)$-matrix and an element of $V_n$ as a $(1\times n)$-matrix.
  For example, we consider $v_i\in V_n$ as the $i$-th standard basis.
  For $f\in\FAb^\op(m,n)$ and $u(w_1,\cdots,w_{2d-k})\in B_d(m)$, we define
  $$
    B_d(f)(u(w_1,\cdots,w_{2d-k})):=u(w_1\cdot f,\cdots,w_{2d-k}\cdot f).
  $$
  It can be easily seen that $B_d$ is a functor.

  We can apply the definition of  $B_d:\FAb^\op(m,n)\rightarrow\gVect(B_d(m),B_d(n))$ to the opposite group $\GL(n;\K)^{\op}$ of the general linear group $\GL(n;\K)$ with coefficient in $\K$ to obtain a group homomorphism
   $$
     B_d:\GL(n;\K)^\op\rightarrow\Aut_{\gVect}(B_d(n)).
   $$
  Then we have a $\GL(n;\K)$-action on $B_d(n)$ by identifying $\GL(n;\K)$ with $\GL(n;\K)^{\op}$ by taking an element to its inverse.

  On the other hand, we consider the $\GL(V_n)$-action on $B_d(n)$ that is determined by the standard action of $\GL(V_n)$ on each coloring. Here, we consider an element of $V_n=\bigoplus_{i=1}^n\K v_i$ as an $(n\times 1)$-matrix.
  The $\GL(n;\K)$-action on $B_d(n)$ factors through the dual action of $\GL(n;\K)$ on $V_n$ and the standard action of $\GL(V_n)$ on $B_d(n)$:
  \begin{equation}\label{GLaction}
   \GL(n;\K)\xrightarrow{{}^t\!(\cdot)^{-1}}\GL(V_n)
   \rightarrow\Aut_{\gVect}(B_d(n)).
  \end{equation}
   Note that the isomorphism (\ref{eqBD}) is a $\GL(V_n)$-module isomorphism.

 \subsection{Relation between the functors $A_d$ and $B_d$}\label{ss34}
  In this section, we show that the functor $A_d$ defined in Section \ref{ss32} induces the functor $B_d$ defined in Section \ref{ss33}.

  In the following lemma, we observe that we can identify the associated graded vector space $\gr(A_d(n))$ of the filtered vector space $A_d(n)$ with the graded vector space $B_d(n)$.
  \begin{lemma}\label{PBW'}
    For $d,n,k\geq 0$, we have an isomorphism of $\K$-vector spaces
    $$
      \theta_{d,n,k}: \opegr^{k}(A_d(n))\xrightarrow{\cong}B_{d,k}(n),
    $$
    which maps a Jacobi diagram $D$ on $X_n$ to an open Jacobi diagram $\theta_{d,n,k}(D)$ that is obtained from $D$ by assigning the color $v_i$ to a univalent vertex which is attached to the $i$-th arc component of $X_n$ for any $i=1,\cdots,n$.

    Taking direct sum, we have an isomorphism of graded vector spaces
    $$
      \theta_{d,n}: \gr(A_d(n))\xrightarrow{\cong}B_d(n).
    $$
    We call $\theta_{d,n}$ the PBW map.
  \end{lemma}
  \begin{proof}
   By identifying $B_{d,k}(n)$ with $\flB_{d,k}([n])$ and $A_{d,k}(n)$ with $\flA_{d,k}(X_n)$ through the canonical isomorphisms, it follows from Proposition \ref{PBW} that we have an isomorphism
   $$\chi_{d,n,k}: B_{d,k}(n)\xrightarrow{\cong} \opegr^k(A_d(n)).$$
   Thus, we have an isomorphism $\theta_{d,n,k}:={\chi_{d,n,k}}^{-1}$.
  \end{proof}

  For $d\geq 0$, we define another functor $\wti{B_d}:\F^\op\rightarrow \gVect$ as follows.
  For an object $n\geq 0$, let $\wti{B_d}(n):= B_d(n)$.
  For a morphism $f:m\rightarrow n$ in $\F^\op$, let
  $$
   \wti{B_d}(f) := \theta_{d,n}\circ\gr(A_d(f))\circ\theta_{d,m}^{-1}
   : B_d(m) \rightarrow B_d(n).
  $$
  The family of morphisms $\theta_d:=(\theta_{d,n})_{n\geq 0}:\gr\circ A_d\Rightarrow\wti{B_d}$ is a natural isomorphism. This is because the PBW maps $\theta_{d,m}$ and $\theta_{d,n}$ are isomorphisms and because the following diagram commutes:
  \begin{gather*}
   \xymatrix{
    \gr(A_d(m))\ar[rr]^{\gr(A_d(f))}\ar[d]^{\cong}_{\theta_{d,m}}
    &&
     \gr(A_d(n))\ar[d]^{\theta_{d,n}}_{\cong}
      \\
     B_d(m)\ar[rr]_{\wti{B_d}(f)}
    &&
    B_d(n)\ar@{}[llu]|{\circlearrowright}.
   }
  \end{gather*}

  \begin{proposition}\label{prop341}
    For $d\geq 0$, we have $\wti{B_d}=B_d\circ\ab^\op$.
    Thus, the family of the PBW maps $\theta_d$ can be rewritten as a natural isomorphism $\theta_d:\gr\circ A_d\overset{\cong}\Rightarrow B_d\circ\ab^\op$. In diagram, we have
    \begin{gather*}
     \xymatrix{
      \F^\op\ar[rr]^{A_d}\ar[d]_{\ab^\op}
      &&
       \fVect\ar[d]^{\gr}
        \\
       \FAb^\op\ar[rr]_{B_d}
      &&
      \gVect \ar@{}[llu]|{{\cong}\:\big\Downarrow\:{\theta_d}}.
     }
    \end{gather*}
  \end{proposition}
  \begin{proof}
   We show that $\wti{B_d}=B_d\circ\ab^\op$.
   For an element $f\in\F^\op(m,n)=\F(n,m)$, let $\ti{a}_{i,j}\in \N$ (resp. $a_{i,j}\in \Z$) be the number (resp. the sum of signs) of copies of $x_i^{\pm 1}$ that appear in the word $f(x_j)$ for $i=1,\cdots,m$ and $j=1,\cdots,n.$
   For example, if $f:F_2\rightarrow F_2$ is defined by
   \begin{equation}\label{exf}
     f(x_1)=x_1x_2x_1^{-1},\; f(x_2)=x_1^{-1}x_2,
   \end{equation}
   then the corresponding matrices $(\ti{a}_{i,j})$ and $(a_{i,j})$ are
   $$(\ti{a}_{i,j})=
   \begin{pmatrix}
     2&1\\
     1&1\\
   \end{pmatrix},\quad
   (a_{i,j})=
   \begin{pmatrix}
     0&-1\\
     1&1\\
   \end{pmatrix}.
   $$
   Note that the matrix $A=(a_{i,j})\in\Mat(m,n;\Z)$ corresponds to the morphism $\ab^{\op}(f)\in\FAb^{\op}(m,n)$.

   For a diagram $u'=u(w_1,\cdots,w_{2d-k})\in B_{d,k}(m)$, we prove that $\wti{B_d}(f)(u')=B_d\circ\ab^\op(f)(u').$
   It suffices to prove the case where $w_l=v_{i_l}$ for $l\in[2d-k]$ by multilinearity. By the definition of the map $B_d$, we have
   $$B_d\circ\ab^\op(f)(u(v_{i_1},\cdots,v_{i_{2d-k}}))=u(v_{i_1}\cdot A,\cdots,v_{i_{2d-k}}\cdot A).$$

   By Lemma \ref{PBW'}, $\theta_{d,m}^{-1}(u')$ is obtained from $u$ by attaching the $l$-colored univalent vertex of $u$ to the $i_l$-th component of $X_m$ for $l\in[2d-k]$.
   We consider the image of $\theta_{d,m}^{-1}(u')$ under the map $\gr(A_d(f))$.
   First, we take a look at an example. For the morphism $f:F_2\rightarrow F_2$ defined by (\ref{exf}), we have
   \begin{align*}
       &\gr(A_2(f))(\theta_{2,2}^{-1}(\centre{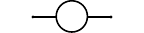}))
         =\gr(A_2(f))(\centre{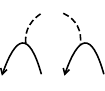})\\
         &=\centre{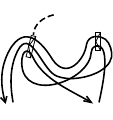}
         =\centre{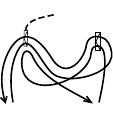}\;-\;\centre{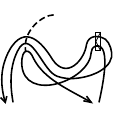}
         \in \opegr^2(A_2(2)).
   \end{align*}
   The first term vanishes because two Jacobi diagrams with the same uni-trivalent graph and with different ways of attaching univalent vertices to a component of $X_2$ are equivalent in $\opegr^2(A_2(2))$ by the STU relations. Thus, we have
   $$
      \gr(A_2(f))(\theta_{2,2}^{-1}(\centre{\input{u2.pdf_tex}}))
      =\:-\;\centre{\input{grA2fu22.pdf_tex}}.
   $$
   As we observed in the example, the map $\gr(A_d(f))$ sends $\theta_{d,m}^{-1}(u')$ to a linear combination of diagrams which are obtained from $u$ by attaching univalent vertices to $X_n$.
   In particular, the map $\gr(A_d(f))$ sends the $l$-colored univalent vertex of $u$ to the signed sum of $\ti{a}_{i_l,j}$ copies of the vertex which are attached to the $j$-th component of $X_n$ for any $j=1,\cdots,n$. In the associated graded vector space $\opegr^k(A_d(n))$, the image of the $l$-colored vertex is actually the signed sum of $|a_{i_l,j}|$ copies of the vertex which are attached to the $j$-th component of $X_n$.

   Through the PBW map $\theta_{d,n}$ again, the Jacobi diagram of  $\wti{B_d}(f)(u')$ is $u$. The coloring of $\wti{B_d}(f)(u')$ that corresponds to the $l$-colored univalent vertex of $u$ is $\sum_{j=1}^n a_{i_l,j} v_j=v_{i_l}\cdot A$, which is equal to that of $B_d\circ\ab^\op(f)(u')$.
  \end{proof}


\section{The functors $A_1$ and $B_1$}\label{s4}
 In this section, we compute the functors $A_1$ and $B_1$.
 
 The vector space $B_1(n)$ has a basis $\{d_{i,j}=\centre{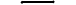}:1\leq i\leq j\leq n\}$. 
 We have a linear isomorphism
 \begin{equation}\label{B_1}
   B_1(n)\xrightarrow{\cong} \Sym^2(V_n)
 \end{equation}
 that maps $d_{i,j}$ to $v_i\cdot v_j$ for $i\leq j$, where $\Sym^2(V_n)$ isthe symmetric square of $V_n$.
 
 We can compute the functor $B_1$ explicitly as follows.
 We extend the notation $d_{i,j}$ by letting $d_{j,i}:=d_{i,j}$ for $i<j$ and let $D:=(d_{i,j})\in \Mat(n,n; B_1(n))$.
 For a morphism $P\in \FAb^{\op}(m,n)=\Mat(m,n;\Z)$,
 it is easily checked that
 $$B_1(P)(d_{i,j})=(P\;D\;{}^t\!P)_{i,j}$$
 for $1\leq i\leq j\leq m$.
 Let $\Sym^2:\Vect \rightarrow \Vect$ denote the functor that maps a vector space $V$ to its symmetric square $\Sym^2(V)$.
 \begin{proposition}\label{B1}
  We have
  $$B_1\cong \Sym^2(\Hom(-,\K)).$$
  Therefore, the linear isomorphism (\ref{B_1}) gives a $\GL(V_n)$-module isomorphism.
 \end{proposition}
 
  Since $A_{1,1}(n)=0$, we have
 $$A_1(n)=A_{1,0}(n)=\gr(A_1(n))\xrightarrow[\cong]{\theta_{1,n}} B_1(n)=B_{1,0}(n)$$ via the PBW map.
 The space $A_1(n)$ has a basis $\{c_{i,j}:1\leq i\leq j\leq n\}$ corresponding to $\{d_{i,j}\}$, where
 $$c_{i,j}=
  \begin{cases}
    \centre{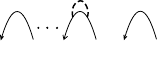} & (i=j)\\
    \centre{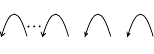} & (i< j).
  \end{cases}
 $$

 Considering the target categories of the functors $A_1$ and $B_1$ as the category $\Vect$ of vector spaces over $\K$, we have
 $$A_1\cong B_1\circ \ab^{\op}$$ by Proposition \ref{prop341}.


\section{Action of $\gr(\IA(n))$ on $B_{d}(n)$}\label{s5}
 The functor $A_d$ gives an $\Aut(F_n)^{\op}$-action on $A_d(n)$, where $\Aut(F_n)^{\op}$ denotes the opposite group of $\Aut(F_n)$.
 We have a right $\Aut(F_n)$-action on $A_d(n)$ by letting $$u\cdot g:=A_d(g)(u)$$ for $u\in A_d(n)$ and $g\in\Aut(F_n)$.

 We first consider the case $n=1$.
 We have
 $$\Aut(F_1)=\{1,s\}\cong \GL(1;\Z)\cong \Z/2\Z.$$
 The action of $s$ on $B_{d,k}(1)$ is multiplication by
  $(-1)^{2d-k}=(-1)^{k}$, although it is known that $B_{d,k}(1)=0$ for $d\leq 9$ and odd $k$ and it is open whether or not we have $B_{d,k}(1)=0$ for all odd $k$ \cite{BN_v}.

 Let $\IA(n)$ denote the \emph{IA-automorphism group} of $F_n$, which is the kernel of the canonical homomorphism $\Aut(F_n)\rightarrow\Aut(H_1(F_n;\Z))\cong \GL(n;\Z)$.
 In this section, we construct an action of the associated graded Lie algebra $\gr(\IA(n))$ of the lower central series of $\IA(n)$ on the graded vector space $B_d(n)$, consisting of group homomorphisms
 $$\beta_{d,k}^{r}:\opegr^r(\IA(n))
 \rightarrow\Hom(B_{d,k}(n),B_{d,k+r}(n))$$
 for $k\geq 0$ and $r\geq 1$, which we define in Section \ref{ss51}.
 In Section \ref{ss53}, we extend this action by adding the case where $r=0$, to obtain an action of an extended graded Lie algebra $\gr(\Aut(F_n)^\op)$ on the graded vector space $B_d(n)$.
 See Appendix \ref{ss52} for extended graded Lie algebras. 
 \subsection{$\Out(F_n)$-action on $A_d(n)$}
  The \emph{inner automorphism group} $\Inn(F_n)$ of $F_n$ is the normal subgroup of $\Aut(F_n)$ consisting of automorphisms $\sigma_a$ for any $a\in F_n$, defined by $\sigma_a(x)=axa^{-1}$ for $x\in F_n$.
  By the definitions of $\Inn(F_n)$ and $\IA(n)$, it follows that $\Inn(F_n)$ is a normal subgroup of $\IA(n)$ for any $n\geq 1$.
  Here, we consider the $\Inn(F_n)$-action on $A_d(n)$.

  \begin{theorem}\label{out}
   The $\Inn(F_n)$-action on $A_d(n)$ is trivial for any $d,n \geq 0$.
   Therefore, the $\Aut(F_n)$-action on $A_d(n)$ induces an action on $A_d(n)$ of the \emph{outer automorphism group} $\Out(F_n)=\Aut(F_n)/\Inn(F_n)$ of $F_n$.

   Thus, the functor $A_d$ is an \emph{outer functor} in the sense of \cite{POWELLVESPA} for any $d\geq 0$.
  \end{theorem}
  \begin{proof}
   We show that the $\Inn(F_n)$-action on $A_d(n)$ is trivial.
   Since $\Inn(F_0)=\Inn(F_1)=1$ and since $A_0(n)=\K\emptyset$, we have only to consider for $n\geq 2, d\geq 1$.
   
   Since $F_n\cong \Inn(F_n)$ for $n\geq 2$,
   the inner automorphism group $\Inn(F_n)$ is generated by $\sigma_{x_1},\cdots, \sigma_{x_n}$.
   For each $i=2,\cdots,n$, define $P_{1,i}\in \Aut(F_n)$ by
   $$P_{1,i}(x_1)=x_i,\quad P_{1,i}(x_i)=x_1,\quad P_{1,i}(x_j)=x_j\quad(j\neq 1,i).$$
   Then, we have
   $\sigma_{x_i}= P_{1,i}\sigma_{x_1}P_{1,i}$.
   If we have $u\cdot \sigma_{x_1} =u$ for any $u\in A_d(n)$, then we have 
   $$u\cdot \sigma_{x_i}=u\cdot P_{1,i}\sigma_{x_1}P_{1,i}
   =(u\cdot P_{1,i})\cdot \sigma_{x_1}P_{1,i}
   =u\cdot P_{1,i}P_{1,i}=u.$$
   Therefore, we need to prove $u\cdot \sigma_{x_1} =u$ for any $u\in A_d(n)$.

   For $u=\centre{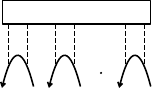}\in A_d(n)$,
   we have $u\cdot \sigma_{x_1}=\scalebox{0.7}{$\centre{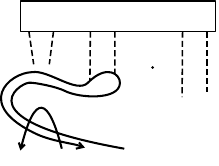}$}$.
   Since the box notation satisfies $\centre{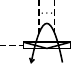}=\centre{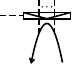}$, by pulling the first arc component to the right, we have 
   $$u\cdot \sigma_{x_1}=\centre{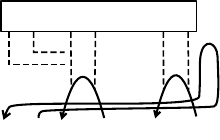}.$$
   Moreover, by pulling the leftmost dashed line up, we obtain 
   $$u\cdot \sigma_{x_1} =\centre{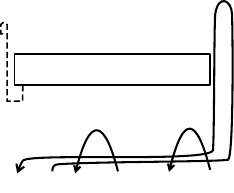}.$$
   By iterating similar operations, we obtain
   \begin{gather*}
     u\cdot \sigma_{x_1}
     =\centre{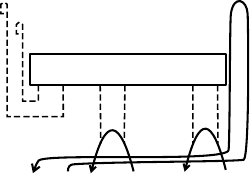}=\centre{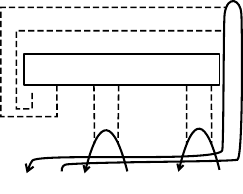}=u.
   \end{gather*}
   This completes the proof.
  \end{proof}

 \subsection{$\IA(n)$-action on $A_d(n)$ for $n=1,2$}\label{ss50}
  For $n=1$, we have $\IA(1)=1$. Therefore, the $\IA(1)$-action on $A_d(1)$ is trivial.

  We use the following fact due to Nielsen \cite{Nielsen} and Magnus \cite{Magnus19}. See also \cite{Magnus} for the statement.
  \begin{theorem}[Nielsen ($n\leq 3$), Magnus (for all $n$)] \label{f611}
    Let $n\geq 2$. The IA-automorphism group $\IA(n)$ is normally generated in $\Aut(F_n)$ by an element
    $K_{2,1}$ defined by
    $$K_{2,1}(x_2)=x_1x_2x_1^{-1},\quad K_{2,1}(x_j)=x_j\quad \text{for }j\neq 2.$$
  \end{theorem}

  For $n=2$, we have $\IA(2)=\Inn(F_2)$ by the above theorem. Therefore, we have the following corollary of Theorem \ref{out}.
  \begin{corollary}\label{IA2}
    The $\IA(2)$-action on $A_d(2)$ is trivial for any $d\geq 0$.
    Therefore, the $\Aut(F_2)$-action on $A_d(2)$ induces an action of $\GL(2;\Z)$ on $A_d(2)$.
  \end{corollary}

 \subsection{Bracket map $[\cdot,\cdot]: B_{d,k}(n)\otimes_{\Z} \opegr^r(\IA(n)) \rightarrow B_{d,k+r}(n)$}\label{ss51}
  We define
  \begin{equation}\label{braA}
    [\cdot,\cdot]:A_d(n)\times\IA(n)\rightarrow A_d(n)
  \end{equation}
  by $[u,g]:=u\cdot g-u$ for $u\in A_d(n)$, $g\in\IA(n)$, which we call the \emph{bracket map}.
  \begin{lemma}\label{l611}
   Let $k\geq 0$.
   We have $$[A_{d,k}(n), \IA(n)]\subset A_{d,k+1}(n).$$
  \end{lemma}
  \begin{proof}
   If we identify the associated graded vector space $\gr(A_d(n))$ with $B_d(n)$, then Proposition \ref{prop341} implies that the $\Aut(F_n)$-action on $A_d(n)$ induces the $\GL(n;\Z)$-action on $B_d(n)$. It follows that the restriction of the $\Aut(F_n)$-action on $\gr(A_d(n))$ to $\IA(n)$ is trivial. This implies that $[A_{d,k}(n), \IA(n)]\subset A_{d,k+1}(n)$.
  \end{proof}

  The following lemma easily follows from the definition of the bracket map.
  \begin{lemma}\label{l612}
   \begin{enumerate}
      \item For $g,h\in\IA(n)$, $u\in A_{d,k}(n)$, we have
         $$[u,gh]=[u,g]+[u,h]+[[u,g],h].$$\label{l6121}
      \item For $g\in\IA(n)$, $u\in A_{d,k}(n)$, we have
         $$[u,g^{-1}]=-[u,g]-[[u,g],g^{-1}].$$\label{l6122}
   \end{enumerate}
  \end{lemma}

  \begin{proposition} \label{p611}
   The bracket map (\ref{braA}) induces a map
   \begin{equation}\label{braB}
     [\cdot, \cdot]:B_{d,k}(n)\times\IA(n)\rightarrow B_{d,k+1}(n).
   \end{equation}
   The map
   $$\beta_{d,k}:\IA(n)\rightarrow \Hom(B_{d,k}(n),B_{d,k+1}(n))$$
   defined by $\beta_{d,k}(g)(u)=[u,g]$ for $g\in\IA(n),u\in B_{d,k}(n)$ is a group homomorphism.
  \end{proposition}
  \begin{proof}
   By Lemma \ref{l611}, the map (\ref{braA}) induces a map
   $$
   [\cdot,\cdot]:\opegr^k(A_{d,\ast}(n))\times \IA(n)\rightarrow \opegr^{k+1}(A_{d,\ast}(n)).
   $$
   By identifying $\opegr^{k}(A_{d,\ast}(n))$ with $B_{d,k}(n)$ via the PBW map, we have a map (\ref{braB}).
   Since we have $[u,gh]=[u,g]+[u,h]+[[u,g],h]$ and $[[u,g],h]\in A_{d,k+2}(n)$ for $g,h\in \IA(n), u\in A_{d,k}(n)$ by Lemmas \ref{l611} and \ref{l612}, it follows that
   $$\beta_{d,k}(gh)(u)=[u,gh]=[u,g]+[u,h] =\beta_{d,k}(g)(u) + \beta_{d,k}(h)(u)$$
   for $g,h\in \IA(n), u\in B_{d,k}(n)$, so the map $\beta_{d,k}$ is a group homomorphism.
  \end{proof}

  Now we consider the lower central series $\Gamma_{\ast}(\IA(n))$ of $\IA(n)$:
  $$
    \IA(n)=\Gamma_1(\IA(n))\rhd \Gamma_2(\IA(n))\rhd \cdots,
  $$
  where $\Gamma_{r+1}(\IA(n))= [\Gamma_{r}(\IA(n)), \IA(n)]$ for $r\geq 1$.
  Note that the commutator bracket $[x,y]$ of $x$ and $y$ is defined to be $[x,y]:=xyx^{-1}y^{-1}$ for elements $x,y$ of a group.
  Let $\gr(\IA(n)):=\bigoplus_{r\geq 1}\opegr^r(\IA(n))=\bigoplus_{r\geq 1}\Gamma_{r}(\IA(n))/\Gamma_{r+1}(\IA(n))$ denote the associated graded Lie algebra with respect to the lower central series of $\IA(n)$.
  We improve the bracket map (\ref{braB}) and the map $\beta_{d,k}$ by restricting the maps to the lower central series.

  \begin{lemma}\label{l613}
   Let $r\geq 1$.
   We have
   \begin{equation}\label{eqbr}
     [A_{d,k}(n),\Gamma_r(\IA(n))]\subset A_{d,k+r}(n).
   \end{equation}
  \end{lemma}
  \begin{proof}
   We prove (\ref{eqbr}) by induction on $r$.
   The case $r=1$ is Lemma \ref{l611}.
   Suppose that (\ref{eqbr}) holds for $r-1\geq 1$. By Lemma \ref{l612}, we have
   \begin{gather*}
    \begin{split}
      [u,[g,h]]
      &=u\cdot (ghg^{-1}h^{-1})-u=(u\cdot gh-u\cdot hg)\cdot g^{-1}h^{-1}\\
      &=([u,gh]-[u,hg])g^{-1}h^{-1}=([[u,g],h]-[[u,h],g])\cdot g^{-1}h^{-1}
    \end{split}
   \end{gather*}
   for any $g\in \Gamma_{r-1}(\IA(n))$ and $h\in\IA(n)$.
   From the induction hypothesis and Lemma \ref{l611}, we have $[u,[g,h]]\in A_{d,k+r}(n)$.
   Therefore, by Lemma \ref{l612}, we have $[u,g]\in A_{d,k+r}(n)$ for any  $g\in\Gamma_r(\IA(n))$ and $u\in A_{d,k}(n)$.
  \end{proof}

  Since $A_{d,2d-1}(n)=0$ for any $d\geq 1$ and $n\geq 0$, we have the following corollary.
  \begin{corollary}
    The $\Aut(F_n)$-action on $A_d(n)$ induces an $\Aut(F_n)/\Gamma_{2d-1}(\IA(n))$-action on $A_d(n)$ for $d\geq 1$.
  \end{corollary}

  We have a canonical isomorphism $\GL(n;\Z)\cong \Aut(F_n)/\IA(n)$ by the definition of $\IA(n)$.
  For any $r\geq 1$, the abelian group $\opegr^r(\IA(n))$ is a right $\GL(n;\Z)$-module by the action induced from the adjoint action of $\Aut(F_n)$ on $\IA(n)$.

  \begin{proposition}\label{p612}
   Let $r\geq 1$.
   The bracket map (\ref{braA}) induces a $\GL(n;\Z)$-module homomorphism
   \begin{equation}\label{braBn}
     [\cdot,\cdot]: B_{d,k}(n)\otimes_{\Z} \opegr^r(\IA(n)) \rightarrow B_{d,k+r}(n).
   \end{equation}
  \end{proposition}
  \begin{proof}
   By Lemmas \ref{l612} and \ref{l613}, we have a $\Z$-bilinear map
   $$[\cdot,\cdot]:B_{d,k}(n)\times \opegr^r(\IA(n))\rightarrow B_{d,k+r}(n)$$
   and thus, we have a $\K$-linear map (\ref{braBn}).

   Since we have $[u, h]\cdot g= [u\cdot g, g^{-1}hg]$ for $g\in\Aut(F_n),h\in\Gamma_r(\IA(n)),u\in A_d(n)$, it follows that the following diagram commutes:
   \begin{gather*}
    \xymatrix{
      B_{d,k}(n)\otimes_{\Z}\opegr^r(\IA(n))\ar[d]_{\cdot A}\ar[r]^-{[\cdot,\cdot]}
      &
      B_{d,k+r}(n)\ar[d]^{\cdot A}
      \\
      B_{d,k}(n)\otimes_{\Z}\opegr^r(\IA(n))\ar[r]_-{[\cdot,\cdot]}
      &
      B_{d,k+r}(n),
    }
   \end{gather*}
   for any $A\in \GL(n;\Z)$.
   Therefore, (\ref{braBn}) is a right $\GL(n;\Z)$-module map.
  \end{proof}

  We define a group homomorphism
  \begin{equation}\label{beta}
    \beta_{d,k}^r: \opegr^r(\IA(n)) \rightarrow \Hom(B_{d,k}(n), B_{d,k+r}(n))
  \end{equation}
  by $\beta_{d,k}^r(g)(u)=[u,g]$ for $g\in\opegr^r(\IA(n))$ and $u\in B_{d,k}(n)$.

 
 \subsection{Action of $\gr(\IA(n))$ on $B_d(n)$}\label{ss53}
 Here we state the $\Aut(F_n)$-action on $A_d(n)$ in terms of extended N-series and the two induced actions on $B_d(n)$ of $\GL(n;\Z)$ and $\gr(\IA(n))$, the latter of which means the $\GL(n;\Z)$-module homomorphisms \eqref{braBn}, in terms of extended graded Lie algebras. See Appendix \ref{ss52} for extended N-series and extended graded Lie algebras.
 
  Set $$\Aut_r(F_n):=\begin{cases}
        \Aut(F_n) &(r=0)\\
        \Gamma_r(\IA(n)) &(r\geq 1).
  \end{cases}$$
  The descending series $\Aut_{\ast}(F_n)=(\Aut_r(F_n))_{r\geq 0}$ is an extended N-series, so the descending series $\Aut_{\ast}(F_n)^{\op}:=(\Aut_r(F_n)^{\op})_{r\geq 0}$ of opposite groups of $\Aut_r(F_n)$ is also an extended N-series.
  Let $\gr(\Aut(F_n)^{\op})$ denote the image of $\Aut_{\ast}(F_n)^{\op}$ under the functor $\gr_{\bullet}$.
  The functor $A_d$ induces an action of $\Aut_{\ast}(F_n)^{\op}$ on the filtered vector space $A_{d,\ast}(n)$:
  $$A_d:\Aut_{\ast}(F_n)^{\op}\rightarrow \Aut_{\ast}(A_{d,\ast}(n)).$$

  \begin{theorem}\label{th531}
    The $\Aut_{\ast}(F_n)^{\op}$-action on the filtered vector space $A_{d,\ast}(n)$ induces an action of the extended graded Lie algebra $\gr(\Aut(F_n)^{\op})$ on the graded vector space $B_{d}(n)$. This action is determined by the functor $B_d$ and the group homomorphisms $\beta_{d,k}^r$ in (\ref{beta}).
    
    In particular, we have a right action of the graded Lie algebra $\gr(\IA(n))$ on the graded vector space $B_d(n)$, which consists of $\GL(n;\Z)$-module homomorphisms
    $$[\cdot,\cdot]:B_{d,k}(n)\otimes_{\Z} \opegr^r(\IA(n)) \rightarrow B_{d,k+r}(n).$$
  \end{theorem}
  \begin{proof}
    By Proposition \ref{p621}, the action of $\Aut_{\ast}(F_n)^{\op}$ on $A_{d,\ast}(n)$ induces an action of the extended graded Lie algebra $\gr(\Aut(F_n)^{\op})$ on the graded vector space $\gr(A_{d,\ast}(n))\cong B_{d}(n)$.
    This action is a pair of a group homomorphism
    $$B_d:\GL(n;\Z)^{\op}\rightarrow\Aut_{\gVect}(B_d(n))$$
    and a graded Lie algebra homomorphism
    $$\bigoplus_{r\geq 1}
    (\opegr^r(\Aut_{\ast}(F_n)^{\op}))
    \rightarrow\bigoplus_{r\geq 1}\End_r(B_d(n)),$$
    which can be regarded as the group homomorphisms $\beta_{d,k}^r$
    by considering the action of $\opegr^r(\Aut_{\ast}(F_n)^{\op})$ on $B_d(n)$ as a right action of $\opegr^r(\IA(n))$ on $B_d(n)$.
  \end{proof}


\section{The $\GL(V_n)$-module $B_d(n)$}\label{sBd}
 In this section, we recall elementary facts about the representation theory to establish the notation and consider the dimension of the $\K$-vector space $B_d(n)$, which is equal to that of $A_d(n)$.
 
  Let $N$ be a nonnegative integer and $\lambda$ a partition of $N$.
  Let $c_{\lambda}\in \K\gpS_N$ denote the Young symmetrizer.
  Let $S^{\lambda}=\K\gpS_N\cdot c_{\lambda}$ denote the Specht module corresponding to $\lambda$, which is a simple $\gpS_N$-module.
  Let $\repS_{\lambda}V_n=V_n^{\otimes N}\cdot c_{\lambda}$ denote the image of $V_n$ under the Schur functor $\repS_{\lambda}$ corresponding to $\lambda$, which is a $\GL(V_n)$-module.
  Let $r(\lambda)$ be the number of rows of $\lambda$.
  If $r(\lambda)\leq n$, then the $\GL(V_n)$-module $\repS_{\lambda}V_n\neq 0$ is simple.
  If $r(\lambda)>n$, we have $\repS_{\lambda}V_n=0$.
  It is well known that any $\gpS_{N}$-module can be decomposed into a direct sum of the Specht modules and that any polynomial representation of $\GL(V_n)$ can be decomposed into a direct sum of the images of $V_n$ under the Schur functors corresponding to partitions.
  See \cite{F,FH} for some basic facts about the representation theory of $\GL(n;\Z)$ and $\GL(V_n)$.

 \begin{proposition}
  The dimension of the $\K$-vector space $B_{d,k}(n)$ is a polynomial of degree $2d-k$ on $n$.
  Therefore, the dimension of the $\K$-vector space $A_d(n)\cong B_d(n)$ is a polynomial of degree $2d$ on $n$. 
 \end{proposition}
\begin{proof}
 Since $B_{d,k}(n)$ is a polynomial representation of $\GL(V_n)$ corresponding to the $\gpS_{2d-k}$-module $D_{d,k}$, the $\GL(V_n)$-module $B_{d,k}(n)$ can be decomposed into a direct sum of $\repS_{\lambda}V_n$ for $\lambda\vdash 2d-k$.
 The dimension of the $\K$-vector space $\repS_{\lambda}V_n$ is a polynomial of degree $2d-k$ on $n$ and so is the dimension of the $\K$-vector space $B_{d,k}(n)$.
\end{proof}
 
\section{$\Aut(F_n)$-module structure of $A_2(n)$ and indecomposable decomposition of the functor $A_2$}\label{s6}
 In this section, we study the $\Aut(F_n)$-module structure of $A_2(n)$ and give an indecomposable decomposition of the functor $A_2$.

 \subsection{Irreducible decomposition of the $\GL(V_n)$-module $B_2(n)$}
  Here we give an irreducible decomposition of the $\GL(V_n)$-module $B_2(n)$.

  Let $B_{2,k}^c(n)\subset B_{2,k}(n)$ be the \emph{connected part} of $B_{2,k}(n)$, which is spanned by connected $V_n$-colored open Jacobi diagrams, and $D_{2,k}^c \subset D_{2,k}$ the \emph{connected part} of $D_{2,k}$, which is spanned by connected special $[4-k]$-colored open Jacobi diagrams. The subspace $D_{2,k}^c$ is an $\gpS_{4-k}$-submodule of $D_{2,k}$.
  We have an isomorphism of $\GL(V_n)$-modules
  \begin{equation}\label{eqBDc}
    B_{2,k}^c(n)\cong V_n^{\otimes 4-k}\otimes_{\K\gpS_{4-k}} D_{2,k}^c,
  \end{equation}
  which is the connected version of (\ref{eqBD}).

  \begin{proposition}[Bar-Natan \cite{BN}]\label{pbn}
   We have isomorphisms of $\gpS_{4-k}$-modules
   $$D^c_{2,1}\cong S^{(1,1,1)},\quad D^c_{2,2}\cong S^{(2)}.$$
  \end{proposition}
  \begin{proposition}\label{prop531}
   We have $B_2(n)=B_{2,0}(n)\oplus B_{2,1}(n)\oplus B_{2,2}(n),$ where
     \begin{gather*}
       \begin{split}
         B_{2,0}(n)&\cong\repS_{(4)}V_n\oplus\repS_{(2,2)}V_n,\\
         B_{2,1}(n)&=B^c_{2,1}(n)\cong\repS_{(1,1,1)}V_n,\\
         B_{2,2}(n)&=B^c_{2,2}(n)\cong\repS_{(2)}V_n
       \end{split}
     \end{gather*}
     as $\GL(V_n)$-modules.
  \end{proposition}
  \begin{proof}
    The cases where $k=1,2$ follow from the isomorphism (\ref{eqBDc}) and Proposition \ref{pbn}.
    For $k=0$, we have an isomorphism of $\GL(V_n)$-modules
    $$
      \Phi:B_{2,0}(n)\xrightarrow{\cong} \Sym^2(B^c_{1,0}(n))=\repS_{(2)}(B^c_{1,0}(n))
    $$
    defined by $\Phi(\centre{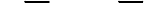}\:)=\centre{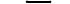}\;\cdot \centre{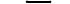}$\; for $w_1,\cdots,w_4\in V_n$.
    By Proposition \ref{B1} and plethysm, it follows that
    $$
      B_{2,0}(n)\cong\Sym^2(B^c_{1,0}(n))\cong\repS_{(2)}(\repS_{(2)}V_n)\cong\repS_{(4)}V_n\oplus\repS_{(2,2)}V_n.
    $$
  \end{proof}

  Let $B_{2,0}'(n)$ (resp. $B_{2,0}''(n)$) denote the subspace of $B_{2,0}(n)$ that is isomorphic to $\repS_{(4)}V_n$ (resp. $\repS_{(2,2)}V_n$).
  By Proposition \ref{prop531}, we have an irreducible decomposition of the $\GL(V_n)$-module $B_2(n)$
  \begin{gather}\label{eqdecompB2}
      B_2(n)=B_{2,0}'(n)\oplus B_{2,0}''(n)\oplus B_{2,1}(n)\oplus B_{2,2}(n).
  \end{gather}
  Here $B_{2,0}''(n)$ vanishes only when $n=0,1$, $B_{2,1}(n)$ vanishes only when $n=0,1,2$, and $B_{2,0}'(n)$ and $B_{2,2}(n)$ vanish only when $n=0$.
  By (\ref{GLaction}), the $\GL(n;\Z)$-action on $B_d(n)$ factors through the $\GL(V_n)$-action on $B_d(n)$:
  $$\GL(n;\Z)\hookrightarrow \GL(n;\K)\xrightarrow{{}^t\!(\cdot)^{-1}}\GL(V_n)\rightarrow \Aut_{\gVect}(B_d(n)).$$
  Therefore, the irreducible decomposition (\ref{eqdecompB2}) of $B_2(n)$ holds as the $\GL(n;\Z)$-modules.

  \begin{remark}
    For $n=2$, $\repS_{(2,2)}V_n\cong \det^{2}=\K$ is the simple $\GL(2;\Z)$-module given by the square of the determinant, so it is trivial.
    For $n=3$, $\repS_{(1,1,1)}V_n\cong \det$ is the simple $\GL(3;\Z)$-module given by the determinant.
  \end{remark}

 \subsection{Direct decomposition of $A_2$}\label{ss62}
  Here we give a direct decomposition of the functor $A_2$.

  The category $\A$ has morphisms
  $$
  \mu=\centre{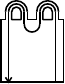}:2\rightarrow 1,\quad
  \eta=\centre{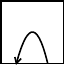}:0\rightarrow 1,\quad
   c=\centre{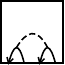}:0\rightarrow 2.
  $$
  Depict $c$ as
  $\centre{\input{cb.pdf_tex}}\:.$
  The iterated multiplications $\mu^{[q]}: q\rightarrow 1$ for $q\geq 0$ are inductively defined by
  $$\mu^{[0]}=\eta,\quad\mu^{[1]}=\id_1, \quad \mu^{[q+1]}=\mu\circ(\mu^{[q]}\otimes\id_1)\quad (q\geq 1).$$
  For $m\geq 0$, there is a group homomorphism
  $$
    \gpS_{m}\rightarrow \A(m,m),\quad \sigma \mapsto P_{\sigma},
  $$
  where $P_{\sigma}$ is the symmetry in $\A$ corresponding to $\sigma$.
  Set
  $$\centre{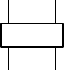}:= \sum_{\sigma\in\gpS_m}P_{\sigma},\quad
  \centre{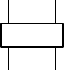}:= \sum_{\sigma\in\gpS_m}\sgn(\sigma)P_{\sigma}\in\A(m,m).$$

  By Habiro--Massuyeau \cite[Lemma 5.16]{HM_k}, every element of $A_2(n)$ is a linear combination of morphisms of the form
  $$
    (\mu^{[q_1]}\otimes\cdots\otimes\mu^{[q_n]})\circ P_{\sigma}\circ c^{\otimes 2}
  $$
  for $\sigma\in\gpS_4$ and $q_1,\cdots, q_n\geq 0$ such that $q_1+\cdots+q_n=4$. 
  For example, we have 
  $$\centre{\input{A21m.pdf_tex}}=(\mu^{[1]}\otimes\mu^{[2]}\otimes\mu^{[1]}\otimes(\mu^{[0]})^{\otimes n-3})\circ (P_{(2,3)}-P_{\id})\circ c^{\otimes 2}.$$
  
  Since we have $(\mu^{[q_1]}\otimes\cdots\otimes\mu^{[q_n]})\circ P_{\sigma}\in \A_0(4,n)$ and $\A_0(m,n)\cong \K\F^{\op}(m,n)$, by the above decomposition of morphisms of $A_2(n)$, we have the following lemma.
  \begin{lemma}\label{l1021}
   For $n\geq 0$, we have
   $$
     A_2(n)=\Span_\K\{A_2(f)(c\otimes c):f\in\F^\op(4,n)\}.
   $$
  \end{lemma}

  Set
  \begin{gather*}
    \begin{split}
    P'&=\centre{\input{A21gen.pdf_tex}}=8\:c\otimes c+8\:\centre{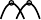}+8\:\centre{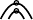} \in A_2(4),\\
    P''&=\centre{\input{A22gen.pdf_tex}}=2\:\centre{\input{cc1423.pdf_tex}}-2\:\centre{\input{cc1324.pdf_tex}} \in A_2(4).
    \end{split}
  \end{gather*}
  
  Let
  $$A_2',A_2'':\F^\op\rightarrow\fVect$$
  be the subfunctors of the functor $A_2$ such that
  $$A_2'(n):=\Span_\K\{A_2(f)(P'):f\in\F^\op(4,n)\}\subset A_2(n),$$
  $$A_2''(n):=\Span_\K\{A_2(f)(P''):f\in\F^\op(4,n)\}\subset A_2(n),$$
  respectively.

  \begin{proposition}\label{prop1021}
   We have a direct decomposition
   $$
     A_2=A_2'\oplus A_2''
   $$
   in the functor category $\fVect^{\F^\op}$.
  \end{proposition}
  \begin{proof}
   We prove that $A_2(n)=A_2'(n)+A_2''(n)$ for $n\geq 0$.
   Since we have
   \begin{gather*}
    \begin{split}
     &\centre{\input{A21gen.pdf_tex}}
     + 4\centre{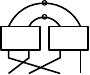}
     +4\centre{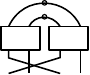}\\
     &=(8\:c\otimes c+8\:\centre{\input{cc1324.pdf_tex}}+8\:\centre{\input{cc1423.pdf_tex}})
     +(8\:c\otimes c-8\:\centre{\input{cc1423.pdf_tex}})+(8\:c\otimes c
     -8\:\centre{\input{cc1324.pdf_tex}})\\
     &=24\:c\otimes c,
    \end{split}
   \end{gather*}
   it follows that $c\otimes c\in A_2'(n)+ A_2''(n)$.
   Thus, we have $A_2(n)=A_2'(n)+A_2''(n)$ by Lemma \ref{l1021}.

   In order to prove that $A_2'(n)\cap A_2''(n)=0$, it suffices to show that
   $$\theta_{2,n}(\gr(A_2'(n)))\subset B_{2,0}'(n),\quad \theta_{2,n}(\gr(A_2''(n)))\subset B_{2,0}''(n)\oplus B_{2,1}(n)\oplus B_{2,2}(n).$$

   Let $P'_{ijkl}=\centre{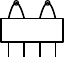}$ be a linear sum of elements of $A_2(n)$ such that each endpoint of two chords are attached to the $i,j,k,l$-th component of $X_n$, respectively, where $1\leq i\leq j\leq k\leq l\leq n$.
   Note that $P'_{ijkl}$ is defined independently of how to attach endpoints to the same component of $X_n$ because of the symmetrizer.
   Since an element of $A_2'(n)$ is a linear sum of $P'_{ijkl}$ and since we have
   $$\theta_{2,n}(P'_{ijkl})
   =\; \centre{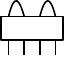}\;\in B_{2,0}'(n),$$
   it follows that $\theta_{2,n}(\gr(A_2'(n)))\subset B_{2,0}'(n)$.

   Let $P''_{ijkl}=\centre{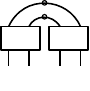}$ be a linear sum of elements of $A_2(n)$ such that each endpoint of two chords are attached to the $i,j,k,l$-th component, respectively, where $i,j,k,l \in\{1,\cdots, n\}$.
   Note that $P''_{ijkl}$ has ambiguity of how to attach endpoints to the same component of $X_n$, but the difference is an element of $A_{2,1}(n)$.
   Since an element of $A_2''(n)$ is a linear sum of $P''_{ijkl}$ and since we have
   $$
     \theta_{2,n}(P''_{ijkl})\;
     -\;\centre{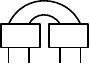} \;\in B_{2,1}(n)\oplus B_{2,2}(n)
   $$
   by Lemma \ref{PBW'}, it follows that $\theta_{2,n}(\gr(A_2''(n)))\subset B_{2,0}''(n)\oplus B_{2,1}(n)\oplus B_{2,2}(n).$
  \end{proof}
  \begin{proposition}
   We have
   \begin{gather}\label{eqB212}
     \begin{split}
       \theta_{2,n}(\gr(A_2'(n)))&=B_{2,0}'(n),\\
       \theta_{2,n}(\gr(A_2''(n)))&=B_{2,0}''(n)\oplus B_{2,1}(n)\oplus B_{2,2}(n).
     \end{split}
   \end{gather}
  \end{proposition}
  \begin{proof}
   This follows from Proposition \ref{prop1021} and Lemma \ref{PBW'}.
  \end{proof}

 \subsection{Action of $\gr(\IA(n))$ on $B_2(n)$}\label{ss63}
   In order to study the $\Aut(F_n)$-module structure of $A_2''(n)$, we consider whether the restrictions of the $\GL(n;\Z)$-module homomorphism (\ref{braBn})
   \begin{equation}\label{braB202}
    [\cdot,\cdot]:B_{2,0}''(n)\otimes_{\Z}\opegr^{1}(\IA(n))\rightarrow B_{2,1}(n),
   \end{equation}
   \begin{equation}\label{braB21}
    [\cdot,\cdot]:B_{2,1}(n)\otimes_{\Z}\opegr^{1}(\IA(n))\rightarrow B_{2,2}(n)
   \end{equation}
   vanish or not.

   For $n=1,2$, the maps (\ref{braB202}) and (\ref{braB21}) vanish, because the $\IA(n)$-actions on $A_2(n)$ are trivial by Corollary \ref{IA2}.

   The bracket maps (\ref{braB202}) and (\ref{braB21}) induce $\GL(n;\Z)$-module homomorphisms
   \begin{equation*}
     \rho_1: B_{2,0}''(n)\rightarrow \Hom(\opegr^1(\IA(n)), B_{2,1}(n)),
   \end{equation*}
   \begin{equation*}
     \rho_2: B_{2,1}(n)\rightarrow \Hom(\opegr^1(\IA(n)), B_{2,2}(n)),
   \end{equation*}
   respectively.

   For distinct elements $i,j,k\in [n]$, define $K_{i,j,k}\in\IA(n)$ by
   \begin{equation*}
      K_{i,j,k}(x_i)=x_i [x_j,x_k],\quad K_{i,j,k}(x_l)=x_l\;\text{ for }l\neq i.
   \end{equation*}

   \begin{lemma}\label{prop1031}
    For $n\geq 3$, the $\GL(n;\Z)$-module homomorphisms $\rho_1$ and $\rho_2$ are injective.
   \end{lemma}
   \begin{proof}
    Let
    $$u=\:\centre{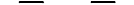}\:-\:\centre{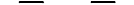}\;\in B_{2,0}''(n).$$

    We have
    $$\rho_1(u)(K_{3,1,2})=6\:\centre{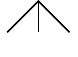}\neq 0\in B_{2,1}(n).$$
    Thus, we have $\rho_1\neq 0$. Since $B_{2,0}''(n)$ is simple, it follows that $\rho_1$ is injective.

    We have
    $$\rho_2(\centre{\input{u_123.pdf_tex}})(K_{1,3,2})=\:\centre{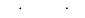}\neq 0\in B_{2,2}(n).$$
    Since $B_{2,1}(n)$ is simple, it follows that $\rho_2$ is injective in a similar way.
   \end{proof}
   \begin{remark}
    The restriction of the $\GL(n;\Z)$-module homomorphism (\ref{braBn})
    $$[\cdot,\cdot]:B_{2,0}''(n)\otimes_{\Z}\opegr^{2}(\IA(n))\rightarrow B_{2,2}(n)$$
    also induces a $\GL(n;\Z)$-module homomorphism
    $$\rho_3:B_{2,0}''(n)\rightarrow \Hom(\opegr^2(\IA(n)), B_{2,2}(n)).$$
    We can also check that $\rho_3$ is injective.
    This is because we have
    \begin{gather*}
      \begin{split}
      [u,[K_{3,1,2},K_{1,3,2}]]&= [[u,K_{3,1,2}],K_{1,3,2}]-[[u,K_{1,3,2}],K_{3,1,2}]\\
      &=[6\:\centre{\input{u_123.pdf_tex}}\;,K_{1,3,2}]=6\:\centre{\input{B22.pdf_tex}}\:\neq 0\in B_{2,2}(n)
      \end{split}
    \end{gather*}
    since $[u,K_{1,3,2}]=0$.
   \end{remark}

 \subsection{$\Aut(F_n)$-module structure of $A_2(n)$}\label{ss64}
  Here, we consider the $\Aut(F_n)$-module structure of $A_2(n)$.

  By Proposition \ref{prop1021}, we have a decomposition of $\Aut(F_n)$-modules
  $$
    A_2(n)=A_2'(n)\oplus A_2''(n)
  $$
  and a filtration of $\Aut(F_n)$-modules
  $$A_2''(n)\supset A_{2,1}(n)\supset A_{2,2}(n)\supset 0.$$

  Moreover, by (\ref{eqdecompB2}) and (\ref{eqB212}), we have $\GL(n;\Z)$-module isomorphisms
  \begin{gather*}
  \begin{split}
    \theta_{2,n}(\gr(A_2'(n)))&=B_{2,0}'(n)\cong \repS_{(4)}V_n,\\
    \theta_{2,n}(\gr(A_2''(n)))&=B_{2,0}''(n)\oplus B_{2,1}(n)\oplus B_{2,2}(n)
     \cong \repS_{(2,2)}V_n\oplus \repS_{(1,1,1)}V_n\oplus \repS_{(2)}V_n,\\
    \theta_{2,n}(\gr(A_{2,1}(n)))&=B_{2,1}(n)\oplus B_{2,2}(n)
     \cong\repS_{(1,1,1)}V_n\oplus \repS_{(2)}V_n,\\
    \theta_{2,n}(\gr(A_{2,2}(n)))&=B_{2,2}(n)
     \cong \repS_{(2)}V_n.
  \end{split}
  \end{gather*}
  Thus, it follows that $A_2'(n)$ and $A_{2,2}(n)$ are simple $\Aut(F_n)$-modules for any $n\geq 1$.

  \begin{theorem}\label{th1042}
   The $\Aut(F_n)$-module $A_2'(n)$ is simple for any $n\geq 1$.

   For $n=1$, $\Aut(F_1)\cong \Z/2\Z$ acts on $A_2''(1)\cong \K$ trivially.

   For $n=2$, $A_2''(2)$ has an irreducible decomposition
   $$A_2''(2)=W\oplus A_{2,2}(2),$$
   where $\Aut(F_2)$ acts on $W\cong \K$ trivially.

   For $n\geq 3$, $A_2''(n)$ admits a unique composition series of length $3$
   $$A_2''(n)\supsetneq A_{2,1}(n)\supsetneq A_{2,2}(n)\supsetneq 0;$$
   that is, $A_2''(n)$ has no nonzero proper $\Aut(F_n)$-submodules other than $A_{2,1}(n)$ and $A_{2,2}(n)$. Therefore, $A_2''(n)$ and $A_{2,1}(n)$ are indecomposable.
  \end{theorem}

  To prove this theorem, we use the following fact due to Nielsen \cite{Nielsengen}. See also \cite{Magnus} for the statement.
  Define $U_{1,2}, P_{1,2}, \sigma \in \Aut(F_2)$ by
  \begin{gather*}
    U_{1,2}(x_1)=x_1x_2,\quad U_{1,2}(x_2)=x_2,\\
    P_{1,2}(x_1)=x_2,\quad P_{1,2}(x_2)=x_1,\\
    \sigma(x_1)=x_1^{-1},\quad \sigma(x_2)=x_2.
  \end{gather*}
  \begin{theorem}[Nielsen \cite{Nielsengen}]\label{Nielsen}
   The automorphism group $\Aut(F_2)$ is generated by $U_{1,2},P_{1,2}$ and $\sigma$.
  \end{theorem}

  \begin{proof}[Proof of Theorem \ref{th1042}]
   For $n=1$, $A_2''(1)=A_{2,2}(1)$ has a basis  $\{\centre{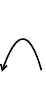}\}$.
   We can check that the action of $\Aut(F_1)\cong \Z/2\Z$ on $A_2''(1)$ is trivial.

   For $n=2$, let
   \begin{gather*}
     u=2\;\centre{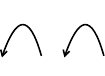}-\centre{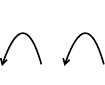}-\centre{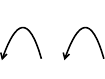}\in A_2''(2)\setminus A_{2,2}(2),\\
     u_{1,1}=\frac{1}{2}\;\centre{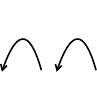},\; u_{1,2}=\centre{\input{thetau2.pdf_tex}},\; u_{2,2}=\frac{1}{2}\;\centre{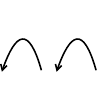}\in A_{2,2}(2).
   \end{gather*}
   It is easily checked that $\{u,u_{1,1},u_{1,2},u_{2,2}\}$ is a basis for $A_2''(2)$ and that the representation matrices of $U_{1,2}, P_{1,2}$ and $\sigma$ for this basis are
   \begin{gather*}
     U_{1,2}=
     \begin{pmatrix}
       1&1&1&0\\
       0&1&0&0\\
       0&2&1&0\\
       0&1&1&1\\
     \end{pmatrix},\quad
     P_{1,2}=
     \begin{pmatrix}
       1&0&0&0\\
       0&0&0&1\\
       0&0&1&0\\
       0&1&0&0\\
     \end{pmatrix},\quad
     \sigma=
     \begin{pmatrix}
       1&0&0&0\\
       0&1&0&0\\
       0&0&-1&0\\
       0&0&0&1\\
     \end{pmatrix}.
   \end{gather*}
   Let $w=(1,-1,0,-1)=u-u_{1,1}-u_{2,2}\in A_2''(2)\setminus A_{2,2}(2)$ and $W=\K w$.
   We have
   $$w\cdot U_{1,2}=w\cdot P_{1,2}=w\cdot \sigma=w.$$
   Thus, by Theorem \ref{Nielsen}, it follows that the $\Aut(F_2)$-action on $W$ is trivial.
   Therefore, we have an irreducible decomposition
   $$A_2''(2)=W\oplus A_{2,2}(2).$$

   For $n\geq 3$, since
   $$ A_2''(n)/A_{2,1}(n)\cong B_{2,0}''(n)\cong \repS_{(2,2)}V_n,\quad A_{2,1}(n)/ A_{2,2}(n)\cong B_{2,1}(n)\cong \repS_{(1,1,1)}V_n$$
   are simple $\Aut(F_n)$-modules, we have a composition series of length $3$
   $$A_2''(n)\supsetneq A_{2,1}(n)\supsetneq A_{2,2}(n)\supsetneq 0.$$

   We next prove that $A_{2,1}(n)$ does not have any nonzero proper submodules other than $A_{2,2}(n)$. (Then, it follows that $A_{2,1}(n)$ is indecomposable.)
   Let $A$ be a nonzero submodule of $A_{2,1}(n)$ other than $A_{2,2}(n)$.
   Since $A_{2,2}(n)$ is simple, there is an element $a\in A\setminus A_{2,2}(n)$.
   We have $\theta_{2,n}(\gr(A_{2,1}(n)))=B_{2,1}(n)\oplus B_{2,2}(n)$, so we can write $a$ as $a=u+v$, for some elements $u\neq 0\in \theta_{2,n}^{-1}(B_{2,1}(n))$, $v\in\theta_{2,n}^{-1}(B_{2,2}(n))=A_{2,2}(n)$.
   By Lemma \ref{prop1031}, there is $g\in\IA(n)$ such that $[u,g]\neq 0\in A_{2,2}(n)$.
   Therefore, we have
   $$[a,g]=[u+v,g]=[u,g]+[v,g]=[u,g]\neq 0\in A_{2,2}(n).$$
   Since $A_{2,2}(n)$ is simple, we have $A_{2,2}(n)\subsetneq A$.
   Since $A_{2,1}(n)$ has a composition series of length $2$, by the Jordan--H\"{o}lder theorem, we have $A=A_{2,1}(n)$.

   We now prove that $A_2''(n)$ does not have any nonzero proper submodules other than $A_{2,1}(n)$ and $A_{2,2}(n)$. (Then, it follows that $A_2''(n)$ is indecomposable.)
   Let $A$ be a nonzero submodule of $A_2''(n)$ other than $A_{2,1}(n),A_{2,2}(n)$.
   Since $A_{2,2}(n)$ is the only nonzero proper submodule of $A_{2,1}(n)$, we have $A\nsubseteq A_{2,1}(n)$.
   Thus, there is an element $a\in A\setminus A_{2,1}(n)$.
   Since we have $\theta_{2,n}(\gr(A_2''(n)))=B_{2,0}''(n)\oplus B_{2,1}(n)\oplus B_{2,2}(n)$, we can write $a$ as $a=u+v$ for some elements $u\neq 0\in\theta_{2,n}^{-1}(B_{2,0}''(n))$, $v\in \theta_{2,n}^{-1}(B_{2,1}(n)\oplus B_{2,2}(n))=A_{2,1}(n)$.
   By Lemma \ref{prop1031}, there is $g\in \IA(n)$ such that $[u,g]\in A_{2,1}(n)\setminus A_{2,2}(n)$.
   Therefore, we have
   $$[a,g]=[u+v,g]=[u,g]+[v,g]\in A_{2,1}(n)\setminus A_{2,2}(n),$$
   because $[v,g]\in A_{2,2}(n)$.
   Since $A_{2,2}(n)$ is the only nonzero proper submodule of $A_{2,1}(n)$, we have $A\cap A_{2,1}(n)=A_{2,1}(n)$ and therefore, $A_{2,1}(n)\subsetneq A$.
   Since $A_2''(n)$ has a composition series of length $3$, by the Jordan--H\"{o}lder theorem, we have $A=A_2''(n)$.
  \end{proof}

  \begin{corollary}\label{corAut}
    There are exact sequences of $\Aut(F_n)$-modules
    \begin{align*}
      0\rightarrow \repS_{(1,1,1)}V_n\rightarrow &A_2''(n)/A_{2,2}(n)\rightarrow \repS_{(2,2)}V_n\rightarrow 0,\\
      0\rightarrow \repS_{(2)}V_n\rightarrow &A_{2,1}(n)\rightarrow \repS_{(1,1,1)}V_n\rightarrow 0,
    \end{align*}
    which do not split for $n\geq 3$.
    Thus, we have
    $$\Ext_{\K\Aut(F_n)}^{1}(\repS_{(2,2)}V_n,\repS_{(1,1,1)}V_n)\neq 0,$$
    $$\Ext_{\K\Aut(F_n)}^{1}(\repS_{(1,1,1)}V_n, \repS_{(2)}V_n)\neq 0$$
    for $n\geq 3$.
  \end{corollary}
  \begin{remark}
    Corollary \ref{corAut} also holds as $\Out(F_n)$-modules.
    In the context of outer functors, the latter fact of Corollary \ref{corAut} about extensions corresponds to some specific cases of Corollary 19.15 in \cite{POWELLVESPA}.
  \end{remark}
  \begin{remark}
    By Theorem \ref{out}, the $\Aut(F_n)$-module $A_{2,1}(n)$ can be considered as an $\Out(F_n)$-module.
    There is no $\Out(F_3)$-modules of dimension less than $7$ which do not factor through the canonical surjection $\Out(F_3)\twoheadrightarrow \GL(3;\Z)$ \cite{Kielak}.
    Turchin and Willwacher \cite{TW} constructed the first $7$-dimensional $\Out(F_3)$-module $U_3^{I}$ with such property.
    We can check that the $\Out(F_3)$-module $A_{2,1}(3)$ is isomorphic to $U_3^{I}$.
    We have another $7$-dimensional $\Out(F_3)$-module $A_2''(3)/A_{2,2}(3)$ which does not factor through $\GL(3;\Z)$.
    At the level of the associated graded $\GL(3;\Z)$-module, $$\gr(A_2''(3)/A_{2,2}(3))\cong \repS_{(2,2)}V_3\oplus \repS_{(1,1,1)}V_3\cong (\repS_{(2)}V_3)^{\ast}\oplus (\repS_{(1,1,1)}V_3)^{\ast}\cong \gr(A_{2,1}(3))^{\ast}.$$
    We conjecture that the $\Out(F_3)$-module $A_2''(3)/A_{2,2}(3)$ is isomorphic to the dual of $A_{2,1}(3)$ and that $A_2''(3)$ is self-dual.
  \end{remark}

 \subsection{Indecomposable decomposition of $A_2$}\label{ss65}
  Finally, we show that the subfunctors $A_2'$ and $A_2''$ of $A_2$, which we observed in Section \ref{ss62}, are indecomposable.

  \begin{theorem}\label{th1041}
   The direct decomposition $A_2=A_2'\oplus A_2''$ in Proposition \ref{prop1021} is indecomposable in the functor category $\fVect^{\F^\op}$.
  \end{theorem}
  \begin{proof}
   Since $A_2'(n)$ are simple $\Aut(F_n)$-modules for any $n\geq 1$ by Theorem \ref{th1042}, the functor $A_2'$ is indecomposable in $\fVect^{\F^\op}$.

   Suppose that we have a direct decomposition
   $$A_2''=G\oplus G'\in \fVect^{\F^\op},$$
   where $G'$ is possibly $0$.
   By Theorem \ref{th1042}, the $\Aut(F_n)$-modules $A_2''(n)$ are indecomposable for $n\geq 3$, we have
   \begin{equation}\label{eqB221}
      G(n)=A_2''(n),\quad G'(n)=0\quad (n\geq 3).
   \end{equation}
   Since $G$ and $G'$ are subfunctors of $A_2$, (\ref{eqB221}) holds for any $n\geq 0$.
   Therefore, we have
   $$G=A_2'',\quad G'=0.$$
   This implies that the functor $A_2''$ is indecomposable in $\fVect^{\F^\op}$.
  \end{proof}


\section{Polynomiality of the functor $A_d$}\label{s7}
In Section \ref{81}, we show that the functor $A_d$ is a polynomial functor of degree $2d$ as the author was informed by Christine Vespa. Moreover, she informed the author that the filtration of $A_d(n)$ corresponds to the polynomial filtration of the polynomial functor $A_d$. (The reader can consult \cite{POWELLVESPA} for the definition of polynomial filtrations.)

In Section \ref{82}, we give some remarks about a polynomial functor $U_d$ associated to a Casimir Lie algebra and a weight system natural transformation from $A_d$ to $U_d$.

\subsection{The polynomial functor $A_d$}\label{81}
We recall the definition of polynomial functors (see Section 2 of \cite{HPV}).
Let $\catC$ be a \emph{pointed monoidal category}, that is, a monoidal category $(\catC,\otimes,0)$ with a null object $0$ as the monoidal unit.
For $X_1,\cdots,X_n\in \catC$, 
let $$r^n_{\hat{k}}:X_1\otimes \cdots\otimes X_n\to X_1\otimes\cdots \hat{X_k}\cdots \otimes X_n$$
be the composition 
$$X_1\otimes\cdots\otimes X_k\otimes\cdots\otimes X_n\to X_1\otimes \cdots \otimes 0\otimes\cdots \otimes X_n\xrightarrow{\cong} X_1\otimes \cdots \hat{X_k} \cdots\otimes X_n,$$
where the first map is determined by the unique morphism $X_k\to 0$.
Let $\catD$ be an additive category.
A functor $F:\catC\to\catD$ is a \emph{polynomial functor} of degree $\leq n$ if 
$$\hat{r}^{F}=(F(r^{n+1}_{\hat{1}}),\cdots,F(r^{n+1}_{\hat{n+1}}))^{t}:F(X_1\otimes\cdots\otimes X_{n+1})\to \bigoplus_{k=1}^{n+1}F(X_1\otimes\cdots\hat{X_k}\cdots\otimes X_{n+1})$$
is monic for any $X_1,\cdots,X_{n+1}\in \catC$.
Note that if $\catD$ is an abelian category, then the above definition of polynomial functors coincides with the definition in \cite{HPV}.
Since the category $\fVect$ is an additive category but not an abelian category, we need this generalized definition of polynomial functors.
\begin{proposition}\label{poly}
 The functor $A_d:\F^{\op}\to \fVect$ is a polynomial functor of degree $2d$.
\end{proposition}
\begin{proof}
For a Jacobi diagram $D\in A_d(n)$, define the \emph{support} $\mathrm{supp}(D)\subset[n]$ of $D$ to be the set of $i\in[n]$ such that at least one of the univalent vertices of $D$ is attached to the $i$-th component of $X_n$.
For $S\subset[n]$, let $A_d(n)_S$ denote the subspace of $A_d(n)$ spanned by the diagrams with support $S$.
Since the three terms in an STU relation have the same support, we have
\begin{gather*}
A_d(n)=\bigoplus_{S\subset[n]}A_d(n)_S.
\end{gather*}

To prove that the functor $A_d$ is a polynomial functor of degree $2d$,
it suffices to show that for each $S\subset[2d+1]$, there is $k\in [2d+1]$ such that $A_d(r^{2d+1}_{\hat{k}}):A_d(2d+1)_S\to A_d(2d)$ is injective.
Since any support $S\subset [2d+1]$ of an element of $A_d(2d+1)$ has at most $2d$ elements, we can choose $k\in [2d+1]\setminus S$.
For the morphism $r^{2d+1}_{\hat{k}}: 2d+1\to 2d$ in $\F^{\op}$, we have $$A_d(r^{2d+1}_{\hat{k}})=\centre{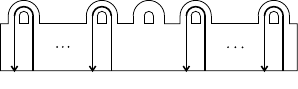}\circ -: A_d(2d+1)\to A_d(2d).$$
Since for any element $u\in A_d(2d+1)_S$, the element $A_d(r^{2d+1}_{\hat{k}})(u)\in A_d(2d)$ is obtained from $u$ by taking away the $k$-th arc component of $X_{2d+1}$ , it follows that the map $A_d(r^{2d+1}_{\hat{k}})$ is injective on $A_d(2d+1)_S$.
\end{proof}

\subsection{Polynomial functor $U_d$ and the weight system map}\label{82}
Here we give some remarks about another polynomial functor $U_d$ and weight systems, which relate $A_d$ to $U_d$.

Let $\g$ be a Casimir Lie algebra in the sense of \cite{HM_k}, which is a Lie algebra with an ad-invariant symmetric $2$-tensor $c_\g\in\g\otimes\g$ (for example, a quadratic Lie algebra is a Casimir Lie algebra).
Let $U(\g)$ denote the universal enveloping algebra of $\g$.
We have an increasing filtration $F_{\ast}(U(\g)^{\otimes n})$ of $U(\g)^{\otimes n}$, which is induced by the usual filtration of the tensor algebra of $\g$.

Since $U(\g)$ is a cocommutative Hopf algebra, by \cite{Conant}, $U(\g)^{\otimes n}$ has a left $\Aut(F_n)$-module structure, and the ad-invariant part $(U(\g)^{\otimes n})^{U(\g)}$ of $U(\g)^{\otimes n}$ has a left $\Out(F_n)$-module structure.
Moreover, we can construct a polynomial functor
$$U_d:\F^{\op}\rightarrow \Vect$$
of degree $2d$ that maps $n\in\N$ to $F_{2d}(U(\g)^{\otimes n})^{U(\g)}$, which is also an outer functor.
   
By Section 7.1 of \cite{HM_k}, there is a Hopf algebra $H$ in the category $\A$, and we have a unique linear symmetric monoidal functor $W(\g):\A\rightarrow \mathbf{Mod}_{U(\g)}$ from the category $\A$ to the category $\mathbf{Mod}_{U(\g)}$ of $U(\g)$-modules that sends the pair $(H,c=\centre{\input{cinA.pdf_tex}})$ to $(U(\g),c_{\g})$.
Thus, we have a linear map $$W_n(\g):A_d(n)\rightarrow F_{2d}(U(\g)^{\otimes n}),$$ which we call the \emph{weight system} of the Casimir Lie algebra $(\g,c_{\g})$.
Considering $A_d(n)$ as a left $\Aut(F_n)$-module via $\Aut(F_n)^{\op}\cong \Aut(F_n)$, which takes an element to its inverse, the weight system $W_n(\g)$ preserves the $\Aut(F_n)$-module structure.
Since a Casimir element is ad-invariant, the weight system takes values in the ad-invariant part $(U(\g)^{\otimes n})^{U(\g)}$ of $U(\g)^{\otimes n}$.
Therefore, the weight system induces an $\Out(F_n)$-module map $$W_n(\g):A_d(n)\rightarrow F_{2d}(U(\g)^{\otimes n})^{U(\g)}.$$
Moreover, the family $(W_n(\g))_{n\in \N}$ of weight systems forms a natural transformation between polynomial functors $A_d$ and $U_d$ of degree $2d$.

\appendix
\section{Extended N-series and extended graded Lie algebras}\label{ss52}
  We briefly review the definition of extended N-series and extended graded Lie algebras, which are defined in \cite{HM_j}, and define an action of an extended N-series on a filtered vector space and an action of an extended graded Lie algebras on a graded vector space. Then we prove that an action of an extended N-series on a filtered vector space induces an action of the associated extended graded Lie algebras on the associated graded vector space.
  
  An \emph{extended N-series} $K_{\ast}=(K_n)_{n\geq 0}$ of a group $K$ is a descending series
  $$
    K=K_0\geq K_1\geq K_2\geq \cdots
  $$
  such that $[K_n,K_m]\leq K_{n+m}$ for all $n,m\geq 0$.
  A \emph{morphism} $f: G_{\ast}\rightarrow K_{\ast}$ between extended N-series is a group homomorphism $f: G_0 \rightarrow K_0$ such that we have  $f(G_n)\subset K_n$ for all $n\geq 0.$

  For a filtered vector space $W_{\ast}$, set
  \begin{align*}
    &\Aut_0(W_{\ast}):=\Aut_{\fVect}(W_{\ast}),\\
    &\Aut_n(W_{\ast}):=\{\phi\in\Aut_0(W_{\ast}):[\phi,W_k]\subset W_{k+n}\text{ for all }k\geq 0\}\quad (n\geq 1),
  \end{align*}
  where $[\phi,w]:=\phi(w)-w$ for $w\in W_k$.
  We can easily check that $\Aut_{\ast}(W_{\ast}):=(\Aut_n(W_{\ast}))_{n\geq 0}$ is an extended N-series.

  \begin{definition}(Action of extended N-series on filtered vector spaces)
   Let $K_{\ast}$ be an extended N-series and $W_{\ast}$ be a filtered vector space.
   An \emph{action} of $K_{\ast}$ on $W_{\ast}$ is a morphism $f:K_{\ast}\rightarrow \Aut_{\ast}(W_{\ast})$ between extended N-series.
  \end{definition}

  An \emph{extended graded Lie algebra} (abbreviated as eg-Lie algebra) $L_{\bullet}=(L_n)_{n\geq0}$ is a pair of
  \begin{itemize}
    \item  a graded Lie algebra $L_{+}=\bigoplus_{n\geq 1} L_n$,
    \item  a group $L_0$ acting on $L_{+}$ in a degree-preserving way.
  \end{itemize}
  A \emph{morphism} $f_{\bullet}=(f_n:L_n\rightarrow L'_n)_{n\geq 0}: L_{\bullet} \rightarrow L'_{\bullet}$ between eg-Lie algebras consists of
  \begin{itemize}
    \item a group homomorphism $f_0: L_0 \rightarrow L'_0$,
    \item a graded Lie algebra homomorphism $f_{+}=(f_n)_{n\geq 1}:L_{+}\rightarrow L'_{+}$,
  \end{itemize}
  such that we have $f_n({}^x y)={}^{f_0(x)}(f_n(y))$ for $n\geq 1, x\in L_0$ and $y\in L_n$.

  We have a functor $\gr_{\bullet}$ from the category of extended N-series to the category of eg-Lie algebras, which maps an extended N-series $K_{\ast}$ to an eg-Lie algebra $\gr_{\bullet}(K_{\ast})=(K_0/K_1,\bigoplus_{n\geq 1}K_n/K_{n+1})$, where Lie bracket is given by the commutator and the action of $K_0/K_1$ on $\bigoplus_{n\geq 1}K_n/K_{n+1}$ is given by the adjoint action.

  For a graded vector space $W=\bigoplus_{k\geq 0}W_k$, set
  \begin{align*}
   &\End_0(W):=\Aut_{\gVect}(W),\\
   &\End_n(W):=\{\phi\in\End(W):\phi(W_k)\subset W_{k+n}\text{ for }k\geq 0\}\quad(n\geq 1).
  \end{align*}
  We can check that $\End_{\bullet}(W)=(\Aut_{\gVect}(W),\bigoplus_{n\geq 1}\End_n(W))$ is an eg-Lie algebra, where the Lie bracket is defined by
  $$[f,g]:=f\circ g-g\circ f \text{ for }f\in\End_k(W),g\in\End_l(W)$$
  and the action of $\Aut_{\gVect}(W)$ on $\bigoplus_{n\geq 1}\End_n(W)$ is defined by the adjoint action
  $${}^{g}f:=g\circ f\circ g^{-1}\text{ for }g\in\Aut_{\gVect}(W),f\in\End_k(W).$$

  \begin{definition}(Action of graded Lie algebras on graded vector spaces)\label{graction}
    Let $L_{+}=\bigoplus_{n\geq 1}L_n$ be a graded Lie algebra and $W=\bigoplus_{k\geq 0}W_k$ be a graded vector space.
    An \emph{action} of $L_{+}$ on $W$ is a morphism $f:L_{+}\rightarrow \bigoplus_{n\geq 1}\End_{n}(W)$ between graded Lie algebras.
  \end{definition}
  \begin{definition}(Action of eg-Lie algebras on graded vector spaces)
   Let $L_{\bullet}$ be an eg-Lie algebra and $W=\bigoplus_{k\geq 0}W_k$ be a graded vector space.
   An \emph{action} of $L_{\bullet}$ on $W$ is a morphism $f:L_{\bullet}\rightarrow \End_{\bullet}(W)$ between eg-Lie algebras.
  \end{definition}

  \begin{proposition}\label{p621}
   Let an extended N-series $K_{\ast}$ act on a filtered vector space $W_{\ast}$.
   Then we have an action of the eg-Lie algebra $\gr_{\bullet}(K_{\ast})$ on the graded vector space $\gr(W_{\ast})$ as follows.
   The group homomorphism
   $$\rho_0:\opegr^0(K_{\ast})\rightarrow \Aut_{\gVect}(\gr(W_{\ast}))$$
   is defined by $\rho_0(gK_1)([v]_{W_{k+1}})=[g(v)]_{W_{k+1}}$ for
   $gK_1\in \opegr^0(K_{\ast}),$
   and the graded Lie algebra homomorphism
   $$\rho_{+}:\bigoplus_{n\geq 1}\opegr^n(K_{\ast})\rightarrow\bigoplus_{n\geq 1}\End_n(\gr(W_{\ast}))$$ is defined by $\rho_{+}(gK_{n+1})([v]_{W_{k+1}})=[[g,v]]_{W_{k+n+1}}$ for $gK_{n+1}\in \opegr^n(K_{\ast})$.
  \end{proposition}
  \begin{proof}
   Firstly, we prove that there is a well-defined group homomorphism $\rho_0$.
   Since $K_{\ast}$ acts on $W_{\ast}$, we have a group homomorphism $$K_0\rightarrow \Aut(\opegr^k(W_{\ast})).$$
   Moreover, since $[g(v)]_{W_{k+1}}=[[g,v]+v]_{W_{k+1}}=[v]_{W_{k+1}}$ for $g\in K_1$ and $v\in W_k$, it follows that
   $K_1\rightarrow \Aut(\opegr^k(W_{\ast}))$ is trivial.
   Thus, the group homomorphism $$\rho_0:\opegr^0(K_{\ast})=K_0/K_1\rightarrow \Aut_{\gVect}(\gr(W_{\ast}))$$ is induced.

   Secondly, we prove that there is a well-defined Lie algebra homomorphism $\rho_{+}$.
   Since $K_{\ast}$ acts on $W_{\ast}$, we can check that $\rho_{+}$ is well defined.
   Moreover, the map $\rho_{+}$ is a Lie algebra homomorphism because for $g\in K_n$, $h\in K_{n'}$ and $v\in W_k$, we have
   \begin{gather*}
    \begin{split}
      [[gK_{n+1},hK_{n'+1}],[v]_{W_{k+1}}]
      &=[[g,h]K_{n+n'+1},[v]_{W_{k+1}}]\\
      &=[[g,h],v]_{W_{k+n+n'+1}}\\
      &=[[g,[h,v]]-[h,[g,v]]]_{W_{k+n+n'+1}}\\
      =&[gK_{n+1},[hK_{n'+1},[v]_{W_{k+1}}]]-[hK_{n'+1},[gK_{n+1},[v]_{W_{k+1}}]].
    \end{split}
   \end{gather*}

   Finally, we check that $\rho_{+}$ is compatible with $\rho_0$.
   For $g\in K_0, h\in K_n$ and $v\in W_k$, we have
   \begin{gather*}
    \begin{split}
      \rho_{+}({}^{gK_1}hK_{n+1})([v]_{W_{k+1}})
      &=[[ghg^{-1},v]]_{W_{k+n+1}}\\
      &=[g([h,g^{-1}(v)])]_{W_{k+n+1}}\\
      &={}^{\rho_0(gK_1)}(\rho_{+}(hK_{n+1}))([v]_{W_{k+1}}).
    \end{split}
   \end{gather*}
   Therefore, this is an action of $\gr_{\bullet}(K_{\ast})$ on $\gr(W_{\ast})$.
  \end{proof}


\end{document}

%% file: X_n.pdf_tex
\begingroup%
  \makeatletter%
  \providecommand\color[2][]{%
    \errmessage{(Inkscape) Color is used for the text in Inkscape, but the package 'color.sty' is not loaded}%
    \renewcommand\color[2][]{}%
  }%
  \providecommand\transparent[1]{%
    \errmessage{(Inkscape) Transparency is used (non-zero) for the text in Inkscape, but the package 'transparent.sty' is not loaded}%
    \renewcommand\transparent[1]{}%
  }%
  \providecommand\rotatebox[2]{#2}%
  \newcommand*\fsize{\dimexpr\f@size pt\relax}%
  \newcommand*\lineheight[1]{\fontsize{\fsize}{#1\fsize}\selectfont}%
  \ifx\svgwidth\undefined%
    \setlength{\unitlength}{68.96467443bp}%
    \ifx\svgscale\undefined%
      \relax%
    \else%
      \setlength{\unitlength}{\unitlength * \real{\svgscale}}%
    \fi%
  \else%
    \setlength{\unitlength}{\svgwidth}%
  \fi%
  \global\let\svgwidth\undefined%
  \global\let\svgscale\undefined%
  \makeatother%
  \begin{picture}(1,0.35281587)%
    \lineheight{1}%
    \setlength\tabcolsep{0pt}%
    \put(0,0){\includegraphics[width=\unitlength,page=1]{X_n.pdf}}%
    \put(0.09653969,0.00511188){\makebox(0,0)[lt]{\lineheight{1.45000005}\smash{\begin{tabular}[t]{l}$1$\end{tabular}}}}%
    \put(0.85441458,0.00521592){\makebox(0,0)[lt]{\lineheight{1.45000005}\smash{\begin{tabular}[t]{l}$n$\end{tabular}}}}%
    \put(0.5629133,0.22591732){\makebox(0,0)[lt]{\lineheight{1.45000005}\smash{\begin{tabular}[t]{l}$\cdots$\end{tabular}}}}%
    \put(0,0){\includegraphics[width=\unitlength,page=2]{X_n.pdf}}%
    \put(0.37031179,0.00700557){\makebox(0,0)[lt]{\lineheight{1.45000005}\smash{\begin{tabular}[t]{l}$2$\end{tabular}}}}%
  \end{picture}%
\endgroup%

%% file: A20m.pdf_tex
\begingroup%
  \makeatletter%
  \providecommand\color[2][]{%
    \errmessage{(Inkscape) Color is used for the text in Inkscape, but the package 'color.sty' is not loaded}%
    \renewcommand\color[2][]{}%
  }%
  \providecommand\transparent[1]{%
    \errmessage{(Inkscape) Transparency is used (non-zero) for the text in Inkscape, but the package 'transparent.sty' is not loaded}%
    \renewcommand\transparent[1]{}%
  }%
  \providecommand\rotatebox[2]{#2}%
  \newcommand*\fsize{\dimexpr\f@size pt\relax}%
  \newcommand*\lineheight[1]{\fontsize{\fsize}{#1\fsize}\selectfont}%
  \ifx\svgwidth\undefined%
    \setlength{\unitlength}{68.60125996bp}%
    \ifx\svgscale\undefined%
      \relax%
    \else%
      \setlength{\unitlength}{\unitlength * \real{\svgscale}}%
    \fi%
  \else%
    \setlength{\unitlength}{\svgwidth}%
  \fi%
  \global\let\svgwidth\undefined%
  \global\let\svgscale\undefined%
  \makeatother%
  \begin{picture}(1,0.41107595)%
    \lineheight{1}%
    \setlength\tabcolsep{0pt}%
    \put(0,0){\includegraphics[width=\unitlength,page=1]{A20m.pdf}}%
    \put(0.07488336,0.00385417){\makebox(0,0)[lt]{\lineheight{1.45000005}\smash{\begin{tabular}[t]{l}$1$\end{tabular}}}}%
    \put(0.31867711,0.00385421){\makebox(0,0)[lt]{\lineheight{1.45000005}\smash{\begin{tabular}[t]{l}$2$\end{tabular}}}}%
    \put(0.8904487,0.00385417){\makebox(0,0)[lt]{\lineheight{1.45000005}\smash{\begin{tabular}[t]{l}$n$\end{tabular}}}}%
    \put(0.65704317,0.1800465){\makebox(0,0)[lt]{\lineheight{1.45000005}\smash{\begin{tabular}[t]{l}$\cdots$\end{tabular}}}}%
    \put(0,0){\includegraphics[width=\unitlength,page=2]{A20m.pdf}}%
    \put(0.56247076,0.00385417){\makebox(0,0)[lt]{\lineheight{1.45000005}\smash{\begin{tabular}[t]{l}$3$\end{tabular}}}}%
  \end{picture}%
\endgroup%

%% file: A21m.pdf_tex
\begingroup%
  \makeatletter%
  \providecommand\color[2][]{%
    \errmessage{(Inkscape) Color is used for the text in Inkscape, but the package 'color.sty' is not loaded}%
    \renewcommand\color[2][]{}%
  }%
  \providecommand\transparent[1]{%
    \errmessage{(Inkscape) Transparency is used (non-zero) for the text in Inkscape, but the package 'transparent.sty' is not loaded}%
    \renewcommand\transparent[1]{}%
  }%
  \providecommand\rotatebox[2]{#2}%
  \newcommand*\fsize{\dimexpr\f@size pt\relax}%
  \newcommand*\lineheight[1]{\fontsize{\fsize}{#1\fsize}\selectfont}%
  \ifx\svgwidth\undefined%
    \setlength{\unitlength}{68.60125996bp}%
    \ifx\svgscale\undefined%
      \relax%
    \else%
      \setlength{\unitlength}{\unitlength * \real{\svgscale}}%
    \fi%
  \else%
    \setlength{\unitlength}{\svgwidth}%
  \fi%
  \global\let\svgwidth\undefined%
  \global\let\svgscale\undefined%
  \makeatother%
  \begin{picture}(1,0.41524619)%
    \lineheight{1}%
    \setlength\tabcolsep{0pt}%
    \put(0,0){\includegraphics[width=\unitlength,page=1]{A21m.pdf}}%
    \put(0.07311277,0.00661912){\makebox(0,0)[lt]{\lineheight{1.45000005}\smash{\begin{tabular}[t]{l}$1$\end{tabular}}}}%
    \put(0.31986033,0.00385417){\makebox(0,0)[lt]{\lineheight{1.45000005}\smash{\begin{tabular}[t]{l}$2$\end{tabular}}}}%
    \put(0.89018137,0.00792008){\makebox(0,0)[lt]{\lineheight{1.45000005}\smash{\begin{tabular}[t]{l}$n$\end{tabular}}}}%
    \put(0.66907828,0.17933233){\makebox(0,0)[lt]{\lineheight{1.45000005}\smash{\begin{tabular}[t]{l}$\cdots$\end{tabular}}}}%
    \put(0,0){\includegraphics[width=\unitlength,page=2]{A21m.pdf}}%
    \put(0.56326776,0.00514217){\makebox(0,0)[lt]{\lineheight{1.45000005}\smash{\begin{tabular}[t]{l}$3$\end{tabular}}}}%
    \put(0,0){\includegraphics[width=\unitlength,page=3]{A21m.pdf}}%
  \end{picture}%
\endgroup%

%% file: A22m.pdf_tex
\begingroup%
  \makeatletter%
  \providecommand\color[2][]{%
    \errmessage{(Inkscape) Color is used for the text in Inkscape, but the package 'color.sty' is not loaded}%
    \renewcommand\color[2][]{}%
  }%
  \providecommand\transparent[1]{%
    \errmessage{(Inkscape) Transparency is used (non-zero) for the text in Inkscape, but the package 'transparent.sty' is not loaded}%
    \renewcommand\transparent[1]{}%
  }%
  \providecommand\rotatebox[2]{#2}%
  \newcommand*\fsize{\dimexpr\f@size pt\relax}%
  \newcommand*\lineheight[1]{\fontsize{\fsize}{#1\fsize}\selectfont}%
  \ifx\svgwidth\undefined%
    \setlength{\unitlength}{51.72732402bp}%
    \ifx\svgscale\undefined%
      \relax%
    \else%
      \setlength{\unitlength}{\unitlength * \real{\svgscale}}%
    \fi%
  \else%
    \setlength{\unitlength}{\svgwidth}%
  \fi%
  \global\let\svgwidth\undefined%
  \global\let\svgscale\undefined%
  \makeatother%
  \begin{picture}(1,0.57767929)%
    \lineheight{1}%
    \setlength\tabcolsep{0pt}%
    \put(0,0){\includegraphics[width=\unitlength,page=1]{A22m.pdf}}%
    \put(0.09578764,0.00681115){\makebox(0,0)[lt]{\lineheight{1.45000005}\smash{\begin{tabular}[t]{l}$1$\end{tabular}}}}%
    \put(0.42303183,0.00655209){\makebox(0,0)[lt]{\lineheight{1.45000005}\smash{\begin{tabular}[t]{l}$2$\end{tabular}}}}%
    \put(0.8531435,0.00511151){\makebox(0,0)[lt]{\lineheight{1.45000005}\smash{\begin{tabular}[t]{l}$n$\end{tabular}}}}%
    \put(0.55282075,0.23017457){\makebox(0,0)[lt]{\lineheight{1.45000005}\smash{\begin{tabular}[t]{l}$\cdots$\end{tabular}}}}%
    \put(0,0){\includegraphics[width=\unitlength,page=2]{A22m.pdf}}%
  \end{picture}%
\endgroup%

%% file: A23intro.pdf_tex
\begingroup%
  \makeatletter%
  \providecommand\color[2][]{%
    \errmessage{(Inkscape) Color is used for the text in Inkscape, but the package 'color.sty' is not loaded}%
    \renewcommand\color[2][]{}%
  }%
  \providecommand\transparent[1]{%
    \errmessage{(Inkscape) Transparency is used (non-zero) for the text in Inkscape, but the package 'transparent.sty' is not loaded}%
    \renewcommand\transparent[1]{}%
  }%
  \providecommand\rotatebox[2]{#2}%
  \newcommand*\fsize{\dimexpr\f@size pt\relax}%
  \newcommand*\lineheight[1]{\fontsize{\fsize}{#1\fsize}\selectfont}%
  \ifx\svgwidth\undefined%
    \setlength{\unitlength}{53.96467632bp}%
    \ifx\svgscale\undefined%
      \relax%
    \else%
      \setlength{\unitlength}{\unitlength * \real{\svgscale}}%
    \fi%
  \else%
    \setlength{\unitlength}{\svgwidth}%
  \fi%
  \global\let\svgwidth\undefined%
  \global\let\svgscale\undefined%
  \makeatother%
  \begin{picture}(1,0.57691996)%
    \lineheight{1}%
    \setlength\tabcolsep{0pt}%
    \put(0,0){\includegraphics[width=\unitlength,page=1]{A23intro.pdf}}%
    \put(0.1329116,0.00653278){\makebox(0,0)[lt]{\lineheight{1.45000005}\smash{\begin{tabular}[t]{l}$1$\end{tabular}}}}%
    \put(0.8307539,0.00786389){\makebox(0,0)[lt]{\lineheight{1.45000005}\smash{\begin{tabular}[t]{l}$3$\end{tabular}}}}%
    \put(0,0){\includegraphics[width=\unitlength,page=2]{A23intro.pdf}}%
    \put(0.47672691,0.0083499){\makebox(0,0)[lt]{\lineheight{1.45000005}\smash{\begin{tabular}[t]{l}$2$\end{tabular}}}}%
    \put(0,0){\includegraphics[width=\unitlength,page=3]{A23intro.pdf}}%
  \end{picture}%
\endgroup%

%% file: A22intro.pdf_tex
\begingroup%
  \makeatletter%
  \providecommand\color[2][]{%
    \errmessage{(Inkscape) Color is used for the text in Inkscape, but the package 'color.sty' is not loaded}%
    \renewcommand\color[2][]{}%
  }%
  \providecommand\transparent[1]{%
    \errmessage{(Inkscape) Transparency is used (non-zero) for the text in Inkscape, but the package 'transparent.sty' is not loaded}%
    \renewcommand\transparent[1]{}%
  }%
  \providecommand\rotatebox[2]{#2}%
  \newcommand*\fsize{\dimexpr\f@size pt\relax}%
  \newcommand*\lineheight[1]{\fontsize{\fsize}{#1\fsize}\selectfont}%
  \ifx\svgwidth\undefined%
    \setlength{\unitlength}{37.27560949bp}%
    \ifx\svgscale\undefined%
      \relax%
    \else%
      \setlength{\unitlength}{\unitlength * \real{\svgscale}}%
    \fi%
  \else%
    \setlength{\unitlength}{\svgwidth}%
  \fi%
  \global\let\svgwidth\undefined%
  \global\let\svgscale\undefined%
  \makeatother%
  \begin{picture}(1,1.13232422)%
    \lineheight{1}%
    \setlength\tabcolsep{0pt}%
    \put(0,0){\includegraphics[width=\unitlength,page=1]{A22intro.pdf}}%
    \put(0.22030661,0.00945764){\makebox(0,0)[lt]{\lineheight{1.45000005}\smash{\begin{tabular}[t]{l}$1$\end{tabular}}}}%
    \put(0,0){\includegraphics[width=\unitlength,page=2]{A22intro.pdf}}%
    \put(0.7233165,0.0131489){\makebox(0,0)[lt]{\lineheight{1.45000005}\smash{\begin{tabular}[t]{l}$2$\end{tabular}}}}%
    \put(0,0){\includegraphics[width=\unitlength,page=3]{A22intro.pdf}}%
  \end{picture}%
\endgroup%

%% file: A24intro.pdf_tex
\begingroup%
  \makeatletter%
  \providecommand\color[2][]{%
    \errmessage{(Inkscape) Color is used for the text in Inkscape, but the package 'color.sty' is not loaded}%
    \renewcommand\color[2][]{}%
  }%
  \providecommand\transparent[1]{%
    \errmessage{(Inkscape) Transparency is used (non-zero) for the text in Inkscape, but the package 'transparent.sty' is not loaded}%
    \renewcommand\transparent[1]{}%
  }%
  \providecommand\rotatebox[2]{#2}%
  \newcommand*\fsize{\dimexpr\f@size pt\relax}%
  \newcommand*\lineheight[1]{\fontsize{\fsize}{#1\fsize}\selectfont}%
  \ifx\svgwidth\undefined%
    \setlength{\unitlength}{35.21547439bp}%
    \ifx\svgscale\undefined%
      \relax%
    \else%
      \setlength{\unitlength}{\unitlength * \real{\svgscale}}%
    \fi%
  \else%
    \setlength{\unitlength}{\svgwidth}%
  \fi%
  \global\let\svgwidth\undefined%
  \global\let\svgscale\undefined%
  \makeatother%
  \begin{picture}(1,0.97039155)%
    \lineheight{1}%
    \setlength\tabcolsep{0pt}%
    \put(0,0){\includegraphics[width=\unitlength,page=1]{A24intro.pdf}}%
    \put(0.20367556,0.01001092){\makebox(0,0)[lt]{\lineheight{1.45000005}\smash{\begin{tabular}[t]{l}$1$\end{tabular}}}}%
    \put(0,0){\includegraphics[width=\unitlength,page=2]{A24intro.pdf}}%
    \put(0.73054287,0.01279551){\makebox(0,0)[lt]{\lineheight{1.45000005}\smash{\begin{tabular}[t]{l}$2$\end{tabular}}}}%
    \put(0,0){\includegraphics[width=\unitlength,page=3]{A24intro.pdf}}%
  \end{picture}%
\endgroup%

%% file: A25intro.pdf_tex
\begingroup%
  \makeatletter%
  \providecommand\color[2][]{%
    \errmessage{(Inkscape) Color is used for the text in Inkscape, but the package 'color.sty' is not loaded}%
    \renewcommand\color[2][]{}%
  }%
  \providecommand\transparent[1]{%
    \errmessage{(Inkscape) Transparency is used (non-zero) for the text in Inkscape, but the package 'transparent.sty' is not loaded}%
    \renewcommand\transparent[1]{}%
  }%
  \providecommand\rotatebox[2]{#2}%
  \newcommand*\fsize{\dimexpr\f@size pt\relax}%
  \newcommand*\lineheight[1]{\fontsize{\fsize}{#1\fsize}\selectfont}%
  \ifx\svgwidth\undefined%
    \setlength{\unitlength}{35.21547439bp}%
    \ifx\svgscale\undefined%
      \relax%
    \else%
      \setlength{\unitlength}{\unitlength * \real{\svgscale}}%
    \fi%
  \else%
    \setlength{\unitlength}{\svgwidth}%
  \fi%
  \global\let\svgwidth\undefined%
  \global\let\svgscale\undefined%
  \makeatother%
  \begin{picture}(1,0.97039155)%
    \lineheight{1}%
    \setlength\tabcolsep{0pt}%
    \put(0,0){\includegraphics[width=\unitlength,page=1]{A25intro.pdf}}%
    \put(0.20367556,0.01001092){\makebox(0,0)[lt]{\lineheight{1.45000005}\smash{\begin{tabular}[t]{l}$1$\end{tabular}}}}%
    \put(0,0){\includegraphics[width=\unitlength,page=2]{A25intro.pdf}}%
    \put(0.73054287,0.01279551){\makebox(0,0)[lt]{\lineheight{1.45000005}\smash{\begin{tabular}[t]{l}$2$\end{tabular}}}}%
    \put(0,0){\includegraphics[width=\unitlength,page=3]{A25intro.pdf}}%
  \end{picture}%
\endgroup%

%% file: A26intro.pdf_tex
\begingroup%
  \makeatletter%
  \providecommand\color[2][]{%
    \errmessage{(Inkscape) Color is used for the text in Inkscape, but the package 'color.sty' is not loaded}%
    \renewcommand\color[2][]{}%
  }%
  \providecommand\transparent[1]{%
    \errmessage{(Inkscape) Transparency is used (non-zero) for the text in Inkscape, but the package 'transparent.sty' is not loaded}%
    \renewcommand\transparent[1]{}%
  }%
  \providecommand\rotatebox[2]{#2}%
  \newcommand*\fsize{\dimexpr\f@size pt\relax}%
  \newcommand*\lineheight[1]{\fontsize{\fsize}{#1\fsize}\selectfont}%
  \ifx\svgwidth\undefined%
    \setlength{\unitlength}{35.21547439bp}%
    \ifx\svgscale\undefined%
      \relax%
    \else%
      \setlength{\unitlength}{\unitlength * \real{\svgscale}}%
    \fi%
  \else%
    \setlength{\unitlength}{\svgwidth}%
  \fi%
  \global\let\svgwidth\undefined%
  \global\let\svgscale\undefined%
  \makeatother%
  \begin{picture}(1,0.88408007)%
    \lineheight{1}%
    \setlength\tabcolsep{0pt}%
    \put(0,0){\includegraphics[width=\unitlength,page=1]{A26intro.pdf}}%
    \put(0.20367556,0.01001092){\makebox(0,0)[lt]{\lineheight{1.45000005}\smash{\begin{tabular}[t]{l}$1$\end{tabular}}}}%
    \put(0,0){\includegraphics[width=\unitlength,page=2]{A26intro.pdf}}%
    \put(0.73054287,0.01279551){\makebox(0,0)[lt]{\lineheight{1.45000005}\smash{\begin{tabular}[t]{l}$2$\end{tabular}}}}%
    \put(0,0){\includegraphics[width=\unitlength,page=3]{A26intro.pdf}}%
  \end{picture}%
\endgroup%

%% file: A27intro.pdf_tex
\begingroup%
  \makeatletter%
  \providecommand\color[2][]{%
    \errmessage{(Inkscape) Color is used for the text in Inkscape, but the package 'color.sty' is not loaded}%
    \renewcommand\color[2][]{}%
  }%
  \providecommand\transparent[1]{%
    \errmessage{(Inkscape) Transparency is used (non-zero) for the text in Inkscape, but the package 'transparent.sty' is not loaded}%
    \renewcommand\transparent[1]{}%
  }%
  \providecommand\rotatebox[2]{#2}%
  \newcommand*\fsize{\dimexpr\f@size pt\relax}%
  \newcommand*\lineheight[1]{\fontsize{\fsize}{#1\fsize}\selectfont}%
  \ifx\svgwidth\undefined%
    \setlength{\unitlength}{35.21547439bp}%
    \ifx\svgscale\undefined%
      \relax%
    \else%
      \setlength{\unitlength}{\unitlength * \real{\svgscale}}%
    \fi%
  \else%
    \setlength{\unitlength}{\svgwidth}%
  \fi%
  \global\let\svgwidth\undefined%
  \global\let\svgscale\undefined%
  \makeatother%
  \begin{picture}(1,0.88408007)%
    \lineheight{1}%
    \setlength\tabcolsep{0pt}%
    \put(0,0){\includegraphics[width=\unitlength,page=1]{A27intro.pdf}}%
    \put(0.20367556,0.01001092){\makebox(0,0)[lt]{\lineheight{1.45000005}\smash{\begin{tabular}[t]{l}$1$\end{tabular}}}}%
    \put(0,0){\includegraphics[width=\unitlength,page=2]{A27intro.pdf}}%
    \put(0.73054287,0.01279551){\makebox(0,0)[lt]{\lineheight{1.45000005}\smash{\begin{tabular}[t]{l}$2$\end{tabular}}}}%
    \put(0,0){\includegraphics[width=\unitlength,page=3]{A27intro.pdf}}%
  \end{picture}%
\endgroup%

%% file: B20n.pdf_tex
\begingroup%
  \makeatletter%
  \providecommand\color[2][]{%
    \errmessage{(Inkscape) Color is used for the text in Inkscape, but the package 'color.sty' is not loaded}%
    \renewcommand\color[2][]{}%
  }%
  \providecommand\transparent[1]{%
    \errmessage{(Inkscape) Transparency is used (non-zero) for the text in Inkscape, but the package 'transparent.sty' is not loaded}%
    \renewcommand\transparent[1]{}%
  }%
  \providecommand\rotatebox[2]{#2}%
  \newcommand*\fsize{\dimexpr\f@size pt\relax}%
  \newcommand*\lineheight[1]{\fontsize{\fsize}{#1\fsize}\selectfont}%
  \ifx\svgwidth\undefined%
    \setlength{\unitlength}{75.19064665bp}%
    \ifx\svgscale\undefined%
      \relax%
    \else%
      \setlength{\unitlength}{\unitlength * \real{\svgscale}}%
    \fi%
  \else%
    \setlength{\unitlength}{\svgwidth}%
  \fi%
  \global\let\svgwidth\undefined%
  \global\let\svgscale\undefined%
  \makeatother%
  \begin{picture}(1,0.03050625)%
    \lineheight{1}%
    \setlength\tabcolsep{0pt}%
    \put(0,0){\includegraphics[width=\unitlength,page=1]{B20n.pdf}}%
    \put(-0.00201311,0.00468861){\makebox(0,0)[lt]{\lineheight{1.45000005}\smash{\begin{tabular}[t]{l}$v_1$\end{tabular}}}}%
    \put(0.54281882,0.00931457){\makebox(0,0)[lt]{\lineheight{1.45000005}\smash{\begin{tabular}[t]{l}$v_2$\end{tabular}}}}%
    \put(0.9183845,0.01003744){\makebox(0,0)[lt]{\lineheight{1.45000005}\smash{\begin{tabular}[t]{l}$v_3$\end{tabular}}}}%
    \put(0.37673274,0.01003737){\makebox(0,0)[lt]{\lineheight{1.45000005}\smash{\begin{tabular}[t]{l}$v_1$\end{tabular}}}}%
  \end{picture}%
\endgroup%

%% file: B21n.pdf_tex
\begingroup%
  \makeatletter%
  \providecommand\color[2][]{%
    \errmessage{(Inkscape) Color is used for the text in Inkscape, but the package 'color.sty' is not loaded}%
    \renewcommand\color[2][]{}%
  }%
  \providecommand\transparent[1]{%
    \errmessage{(Inkscape) Transparency is used (non-zero) for the text in Inkscape, but the package 'transparent.sty' is not loaded}%
    \renewcommand\transparent[1]{}%
  }%
  \providecommand\rotatebox[2]{#2}%
  \newcommand*\fsize{\dimexpr\f@size pt\relax}%
  \newcommand*\lineheight[1]{\fontsize{\fsize}{#1\fsize}\selectfont}%
  \ifx\svgwidth\undefined%
    \setlength{\unitlength}{37.2932625bp}%
    \ifx\svgscale\undefined%
      \relax%
    \else%
      \setlength{\unitlength}{\unitlength * \real{\svgscale}}%
    \fi%
  \else%
    \setlength{\unitlength}{\svgwidth}%
  \fi%
  \global\let\svgwidth\undefined%
  \global\let\svgscale\undefined%
  \makeatother%
  \begin{picture}(1,0.72197429)%
    \lineheight{1}%
    \setlength\tabcolsep{0pt}%
    \put(-0.00405884,0.00945316){\makebox(0,0)[lt]{\lineheight{1.45000005}\smash{\begin{tabular}[t]{l}$v_1$\end{tabular}}}}%
    \put(0.40665432,0.01638878){\makebox(0,0)[lt]{\lineheight{1.45000005}\smash{\begin{tabular}[t]{l}$v_2$\end{tabular}}}}%
    \put(0.83544689,0.01214915){\makebox(0,0)[lt]{\lineheight{1.45000005}\smash{\begin{tabular}[t]{l}$v_3$\end{tabular}}}}%
    \put(0,0){\includegraphics[width=\unitlength,page=1]{B21n.pdf}}%
  \end{picture}%
\endgroup%

%% file: B22n.pdf_tex
\begingroup%
  \makeatletter%
  \providecommand\color[2][]{%
    \errmessage{(Inkscape) Color is used for the text in Inkscape, but the package 'color.sty' is not loaded}%
    \renewcommand\color[2][]{}%
  }%
  \providecommand\transparent[1]{%
    \errmessage{(Inkscape) Transparency is used (non-zero) for the text in Inkscape, but the package 'transparent.sty' is not loaded}%
    \renewcommand\transparent[1]{}%
  }%
  \providecommand\rotatebox[2]{#2}%
  \newcommand*\fsize{\dimexpr\f@size pt\relax}%
  \newcommand*\lineheight[1]{\fontsize{\fsize}{#1\fsize}\selectfont}%
  \ifx\svgwidth\undefined%
    \setlength{\unitlength}{43.60679155bp}%
    \ifx\svgscale\undefined%
      \relax%
    \else%
      \setlength{\unitlength}{\unitlength * \real{\svgscale}}%
    \fi%
  \else%
    \setlength{\unitlength}{\svgwidth}%
  \fi%
  \global\let\svgwidth\undefined%
  \global\let\svgscale\undefined%
  \makeatother%
  \begin{picture}(1,0.22361611)%
    \lineheight{1}%
    \setlength\tabcolsep{0pt}%
    \put(-0.00347119,0.09597589){\makebox(0,0)[lt]{\lineheight{1.45000005}\smash{\begin{tabular}[t]{l}$v_1$\end{tabular}}}}%
    \put(0.85927141,0.09518567){\makebox(0,0)[lt]{\lineheight{1.45000005}\smash{\begin{tabular}[t]{l}$v_2$\end{tabular}}}}%
    \put(0,0){\includegraphics[width=\unitlength,page=1]{B22n.pdf}}%
  \end{picture}%
\endgroup%

%% file: cb.pdf_tex
\begingroup%
  \makeatletter%
  \providecommand\color[2][]{%
    \errmessage{(Inkscape) Color is used for the text in Inkscape, but the package 'color.sty' is not loaded}%
    \renewcommand\color[2][]{}%
  }%
  \providecommand\transparent[1]{%
    \errmessage{(Inkscape) Transparency is used (non-zero) for the text in Inkscape, but the package 'transparent.sty' is not loaded}%
    \renewcommand\transparent[1]{}%
  }%
  \providecommand\rotatebox[2]{#2}%
  \newcommand*\fsize{\dimexpr\f@size pt\relax}%
  \newcommand*\lineheight[1]{\fontsize{\fsize}{#1\fsize}\selectfont}%
  \ifx\svgwidth\undefined%
    \setlength{\unitlength}{15.72738029bp}%
    \ifx\svgscale\undefined%
      \relax%
    \else%
      \setlength{\unitlength}{\unitlength * \real{\svgscale}}%
    \fi%
  \else%
    \setlength{\unitlength}{\svgwidth}%
  \fi%
  \global\let\svgwidth\undefined%
  \global\let\svgscale\undefined%
  \makeatother%
  \begin{picture}(1,1.05203066)%
    \lineheight{1}%
    \setlength\tabcolsep{0pt}%
    \put(0,0){\includegraphics[width=\unitlength,page=1]{cb.pdf}}%
  \end{picture}%
\endgroup%

%% file: c.pdf_tex
\begingroup%
  \makeatletter%
  \providecommand\color[2][]{%
    \errmessage{(Inkscape) Color is used for the text in Inkscape, but the package 'color.sty' is not loaded}%
    \renewcommand\color[2][]{}%
  }%
  \providecommand\transparent[1]{%
    \errmessage{(Inkscape) Transparency is used (non-zero) for the text in Inkscape, but the package 'transparent.sty' is not loaded}%
    \renewcommand\transparent[1]{}%
  }%
  \providecommand\rotatebox[2]{#2}%
  \newcommand*\fsize{\dimexpr\f@size pt\relax}%
  \newcommand*\lineheight[1]{\fontsize{\fsize}{#1\fsize}\selectfont}%
  \ifx\svgwidth\undefined%
    \setlength{\unitlength}{23.975989bp}%
    \ifx\svgscale\undefined%
      \relax%
    \else%
      \setlength{\unitlength}{\unitlength * \real{\svgscale}}%
    \fi%
  \else%
    \setlength{\unitlength}{\svgwidth}%
  \fi%
  \global\let\svgwidth\undefined%
  \global\let\svgscale\undefined%
  \makeatother%
  \begin{picture}(1,1.08854031)%
    \lineheight{1}%
    \setlength\tabcolsep{0pt}%
    \put(0,0){\includegraphics[width=\unitlength,page=1]{c.pdf}}%
    \put(0.14029882,0.01470385){\makebox(0,0)[lt]{\lineheight{1.45000005}\smash{\begin{tabular}[t]{l}$1$\end{tabular}}}}%
    \put(0.76480238,0.01778309){\makebox(0,0)[lt]{\lineheight{1.45000005}\smash{\begin{tabular}[t]{l}$2$\end{tabular}}}}%
  \end{picture}%
\endgroup%

%% file: A21gen.pdf_tex
\begingroup%
  \makeatletter%
  \providecommand\color[2][]{%
    \errmessage{(Inkscape) Color is used for the text in Inkscape, but the package 'color.sty' is not loaded}%
    \renewcommand\color[2][]{}%
  }%
  \providecommand\transparent[1]{%
    \errmessage{(Inkscape) Transparency is used (non-zero) for the text in Inkscape, but the package 'transparent.sty' is not loaded}%
    \renewcommand\transparent[1]{}%
  }%
  \providecommand\rotatebox[2]{#2}%
  \newcommand*\fsize{\dimexpr\f@size pt\relax}%
  \newcommand*\lineheight[1]{\fontsize{\fsize}{#1\fsize}\selectfont}%
  \ifx\svgwidth\undefined%
    \setlength{\unitlength}{30.75117737bp}%
    \ifx\svgscale\undefined%
      \relax%
    \else%
      \setlength{\unitlength}{\unitlength * \real{\svgscale}}%
    \fi%
  \else%
    \setlength{\unitlength}{\svgwidth}%
  \fi%
  \global\let\svgwidth\undefined%
  \global\let\svgscale\undefined%
  \makeatother%
  \begin{picture}(1,0.89721676)%
    \lineheight{1}%
    \setlength\tabcolsep{0pt}%
    \put(0,0){\includegraphics[width=\unitlength,page=1]{A21gen.pdf}}%
    \put(0.20732822,0.36583965){\makebox(0,0)[lt]{\lineheight{1.45000005}\smash{\begin{tabular}[t]{l}$sym_4$\end{tabular}}}}%
  \end{picture}%
\endgroup%

%% file: A22gen.pdf_tex
\begingroup%
  \makeatletter%
  \providecommand\color[2][]{%
    \errmessage{(Inkscape) Color is used for the text in Inkscape, but the package 'color.sty' is not loaded}%
    \renewcommand\color[2][]{}%
  }%
  \providecommand\transparent[1]{%
    \errmessage{(Inkscape) Transparency is used (non-zero) for the text in Inkscape, but the package 'transparent.sty' is not loaded}%
    \renewcommand\transparent[1]{}%
  }%
  \providecommand\rotatebox[2]{#2}%
  \newcommand*\fsize{\dimexpr\f@size pt\relax}%
  \newcommand*\lineheight[1]{\fontsize{\fsize}{#1\fsize}\selectfont}%
  \ifx\svgwidth\undefined%
    \setlength{\unitlength}{42.75117945bp}%
    \ifx\svgscale\undefined%
      \relax%
    \else%
      \setlength{\unitlength}{\unitlength * \real{\svgscale}}%
    \fi%
  \else%
    \setlength{\unitlength}{\svgwidth}%
  \fi%
  \global\let\svgwidth\undefined%
  \global\let\svgscale\undefined%
  \makeatother%
  \begin{picture}(1,0.73701733)%
    \lineheight{1}%
    \setlength\tabcolsep{0pt}%
    \put(0,0){\includegraphics[width=\unitlength,page=1]{A22gen.pdf}}%
    \put(0.02565245,0.24009286){\makebox(0,0)[lt]{\lineheight{1.45000005}\smash{\begin{tabular}[t]{l}$alt_2$\end{tabular}}}}%
    \put(0.57684974,0.24009286){\makebox(0,0)[lt]{\lineheight{1.45000005}\smash{\begin{tabular}[t]{l}$alt_2$\end{tabular}}}}%
  \end{picture}%
\endgroup%

%% file: X_m.pdf_tex
\begingroup%
  \makeatletter%
  \providecommand\color[2][]{%
    \errmessage{(Inkscape) Color is used for the text in Inkscape, but the package 'color.sty' is not loaded}%
    \renewcommand\color[2][]{}%
  }%
  \providecommand\transparent[1]{%
    \errmessage{(Inkscape) Transparency is used (non-zero) for the text in Inkscape, but the package 'transparent.sty' is not loaded}%
    \renewcommand\transparent[1]{}%
  }%
  \providecommand\rotatebox[2]{#2}%
  \newcommand*\fsize{\dimexpr\f@size pt\relax}%
  \newcommand*\lineheight[1]{\fontsize{\fsize}{#1\fsize}\selectfont}%
  \ifx\svgwidth\undefined%
    \setlength{\unitlength}{154.99468248bp}%
    \ifx\svgscale\undefined%
      \relax%
    \else%
      \setlength{\unitlength}{\unitlength * \real{\svgscale}}%
    \fi%
  \else%
    \setlength{\unitlength}{\svgwidth}%
  \fi%
  \global\let\svgwidth\undefined%
  \global\let\svgscale\undefined%
  \makeatother%
  \begin{picture}(1,0.24739893)%
    \lineheight{1}%
    \setlength\tabcolsep{0pt}%
    \put(0,0){\includegraphics[width=\unitlength,page=1]{X_m.pdf}}%
    \put(0.26212831,0.00818771){\color[rgb]{0,0,0}\makebox(0,0)[lt]{\lineheight{1.25}\smash{\begin{tabular}[t]{l}$1$\\\end{tabular}}}}%
    \put(0.47577776,0.00825793){\color[rgb]{0,0,0}\makebox(0,0)[lt]{\lineheight{1.25}\smash{\begin{tabular}[t]{l}$2$\\\end{tabular}}}}%
    \put(0.88300992,0.01056712){\color[rgb]{0,0,0}\makebox(0,0)[lt]{\lineheight{1.25}\smash{\begin{tabular}[t]{l}$n$\\\end{tabular}}}}%
    \put(-0.00314193,0.12330668){\makebox(0,0)[lt]{\lineheight{1.45000005}\smash{\begin{tabular}[t]{l}$X_n=$\end{tabular}}}}%
    \put(0,0){\includegraphics[width=\unitlength,page=2]{X_m.pdf}}%
  \end{picture}%
\endgroup%

%% file: STU.pdf_tex
\begingroup%
  \makeatletter%
  \providecommand\color[2][]{%
    \errmessage{(Inkscape) Color is used for the text in Inkscape, but the package 'color.sty' is not loaded}%
    \renewcommand\color[2][]{}%
  }%
  \providecommand\transparent[1]{%
    \errmessage{(Inkscape) Transparency is used (non-zero) for the text in Inkscape, but the package 'transparent.sty' is not loaded}%
    \renewcommand\transparent[1]{}%
  }%
  \providecommand\rotatebox[2]{#2}%
  \newcommand*\fsize{\dimexpr\f@size pt\relax}%
  \newcommand*\lineheight[1]{\fontsize{\fsize}{#1\fsize}\selectfont}%
  \ifx\svgwidth\undefined%
    \setlength{\unitlength}{248.32821418bp}%
    \ifx\svgscale\undefined%
      \relax%
    \else%
      \setlength{\unitlength}{\unitlength * \real{\svgscale}}%
    \fi%
  \else%
    \setlength{\unitlength}{\svgwidth}%
  \fi%
  \global\let\svgwidth\undefined%
  \global\let\svgscale\undefined%
  \makeatother%
  \begin{picture}(1,0.17216977)%
    \lineheight{1}%
    \setlength\tabcolsep{0pt}%
    \put(0,0){\includegraphics[width=\unitlength,page=1]{STU.pdf}}%
    \put(0.26622516,0.07402791){\color[rgb]{0,0,0}\makebox(0,0)[lt]{\lineheight{1.25}\smash{\begin{tabular}[t]{l}$=$\end{tabular}}}}%
    \put(0.64577693,0.07419581){\color[rgb]{0,0,0}\makebox(0,0)[lt]{\lineheight{1.25}\smash{\begin{tabular}[t]{l}$-$\end{tabular}}}}%
  \end{picture}%
\endgroup%

%% file: ASIHX.pdf_tex
\begingroup%
  \makeatletter%
  \providecommand\color[2][]{%
    \errmessage{(Inkscape) Color is used for the text in Inkscape, but the package 'color.sty' is not loaded}%
    \renewcommand\color[2][]{}%
  }%
  \providecommand\transparent[1]{%
    \errmessage{(Inkscape) Transparency is used (non-zero) for the text in Inkscape, but the package 'transparent.sty' is not loaded}%
    \renewcommand\transparent[1]{}%
  }%
  \providecommand\rotatebox[2]{#2}%
  \newcommand*\fsize{\dimexpr\f@size pt\relax}%
  \newcommand*\lineheight[1]{\fontsize{\fsize}{#1\fsize}\selectfont}%
  \ifx\svgwidth\undefined%
    \setlength{\unitlength}{237.99128478bp}%
    \ifx\svgscale\undefined%
      \relax%
    \else%
      \setlength{\unitlength}{\unitlength * \real{\svgscale}}%
    \fi%
  \else%
    \setlength{\unitlength}{\svgwidth}%
  \fi%
  \global\let\svgwidth\undefined%
  \global\let\svgscale\undefined%
  \makeatother%
  \begin{picture}(1,0.16853431)%
    \lineheight{1}%
    \setlength\tabcolsep{0pt}%
    \put(0,0){\includegraphics[width=\unitlength,page=1]{ASIHX.pdf}}%
    \put(0.55712434,0.07129419){\color[rgb]{0,0,0}\makebox(0,0)[lt]{\lineheight{1.25}\smash{\begin{tabular}[t]{l}$=$\end{tabular}}}}%
    \put(0.81611019,0.07505923){\color[rgb]{0,0,0}\makebox(0,0)[lt]{\lineheight{1.25}\smash{\begin{tabular}[t]{l}$-$\end{tabular}}}}%
    \put(0.10876647,0.07199609){\color[rgb]{0,0,0}\makebox(0,0)[lt]{\lineheight{1.25}\smash{\begin{tabular}[t]{l}$= -$\end{tabular}}}}%
    \put(0.32432587,0.02060045){\color[rgb]{0,0,0}\makebox(0,0)[lt]{\lineheight{1.25}\smash{\begin{tabular}[t]{l}$,$\\\end{tabular}}}}%
  \end{picture}%
\endgroup%

%% file: Um.pdf_tex
\begingroup%
  \makeatletter%
  \providecommand\color[2][]{%
    \errmessage{(Inkscape) Color is used for the text in Inkscape, but the package 'color.sty' is not loaded}%
    \renewcommand\color[2][]{}%
  }%
  \providecommand\transparent[1]{%
    \errmessage{(Inkscape) Transparency is used (non-zero) for the text in Inkscape, but the package 'transparent.sty' is not loaded}%
    \renewcommand\transparent[1]{}%
  }%
  \providecommand\rotatebox[2]{#2}%
  \newcommand*\fsize{\dimexpr\f@size pt\relax}%
  \newcommand*\lineheight[1]{\fontsize{\fsize}{#1\fsize}\selectfont}%
  \ifx\svgwidth\undefined%
    \setlength{\unitlength}{102.54267212bp}%
    \ifx\svgscale\undefined%
      \relax%
    \else%
      \setlength{\unitlength}{\unitlength * \real{\svgscale}}%
    \fi%
  \else%
    \setlength{\unitlength}{\svgwidth}%
  \fi%
  \global\let\svgwidth\undefined%
  \global\let\svgscale\undefined%
  \makeatother%
  \begin{picture}(1,1.09631223)%
    \lineheight{1}%
    \setlength\tabcolsep{0pt}%
    \put(0.13113493,0.01988981){\color[rgb]{0,0,0}\makebox(0,0)[lt]{\lineheight{1.25}\smash{\begin{tabular}[t]{l}$S$\end{tabular}}}}%
    \put(0.30637873,1.0286028){\color[rgb]{0,0,0}\makebox(0,0)[lt]{\lineheight{1.25}\smash{\begin{tabular}[t]{l}$1$\end{tabular}}}}%
    \put(0.64827507,1.02664765){\color[rgb]{0,0,0}\makebox(0,0)[lt]{\lineheight{1.25}\smash{\begin{tabular}[t]{l}$m$\end{tabular}}}}%
    \put(0.44415856,1.03241097){\color[rgb]{0,0,0}\makebox(0,0)[lt]{\lineheight{1.25}\smash{\begin{tabular}[t]{l}$\cdots$\end{tabular}}}}%
    \put(0,0){\includegraphics[width=\unitlength,page=1]{Um.pdf}}%
    \put(0.90816026,0.01613584){\makebox(0,0)[lt]{\lineheight{0.44999999}\smash{\begin{tabular}[t]{l}$l$\end{tabular}}}}%
    \put(0.29921937,0.51246831){\makebox(0,0)[lt]{\lineheight{1.45000005}\smash{\begin{tabular}[t]{l}$x_1$\end{tabular}}}}%
    \put(0.62521291,0.52144546){\makebox(0,0)[lt]{\lineheight{1.45000005}\smash{\begin{tabular}[t]{l}$x_m$\end{tabular}}}}%
    \put(0,0){\includegraphics[width=\unitlength,page=2]{Um.pdf}}%
    \put(0.47059748,0.52459663){\makebox(0,0)[lt]{\lineheight{1.45000005}\smash{\begin{tabular}[t]{l}$\cdots$\end{tabular}}}}%
  \end{picture}%
\endgroup%

%% file: m,nJac.pdf_tex
\begingroup%
  \makeatletter%
  \providecommand\color[2][]{%
    \errmessage{(Inkscape) Color is used for the text in Inkscape, but the package 'color.sty' is not loaded}%
    \renewcommand\color[2][]{}%
  }%
  \providecommand\transparent[1]{%
    \errmessage{(Inkscape) Transparency is used (non-zero) for the text in Inkscape, but the package 'transparent.sty' is not loaded}%
    \renewcommand\transparent[1]{}%
  }%
  \providecommand\rotatebox[2]{#2}%
  \newcommand*\fsize{\dimexpr\f@size pt\relax}%
  \newcommand*\lineheight[1]{\fontsize{\fsize}{#1\fsize}\selectfont}%
  \ifx\svgwidth\undefined%
    \setlength{\unitlength}{283.46456693bp}%
    \ifx\svgscale\undefined%
      \relax%
    \else%
      \setlength{\unitlength}{\unitlength * \real{\svgscale}}%
    \fi%
  \else%
    \setlength{\unitlength}{\svgwidth}%
  \fi%
  \global\let\svgwidth\undefined%
  \global\let\svgscale\undefined%
  \makeatother%
  \begin{picture}(1,0.38)%
    \lineheight{1}%
    \setlength\tabcolsep{0pt}%
    \put(0.14479863,0.15558384){\color[rgb]{0,0,0}\makebox(0,0)[lt]{\lineheight{1.25}\smash{\begin{tabular}[t]{l}$D=$ \\\end{tabular}}}}%
    \put(0,0){\includegraphics[width=\unitlength,page=1]{m,nJac.pdf}}%
    \put(0.67462822,0.15480127){\makebox(0,0)[lt]{\lineheight{1.45000005}\smash{\begin{tabular}[t]{l}$:2\rightarrow 3$\end{tabular}}}}%
    \put(0,0){\includegraphics[width=\unitlength,page=2]{m,nJac.pdf}}%
  \end{picture}%
\endgroup%

%% file: signsum.pdf_tex
\begingroup%
  \makeatletter%
  \providecommand\color[2][]{%
    \errmessage{(Inkscape) Color is used for the text in Inkscape, but the package 'color.sty' is not loaded}%
    \renewcommand\color[2][]{}%
  }%
  \providecommand\transparent[1]{%
    \errmessage{(Inkscape) Transparency is used (non-zero) for the text in Inkscape, but the package 'transparent.sty' is not loaded}%
    \renewcommand\transparent[1]{}%
  }%
  \providecommand\rotatebox[2]{#2}%
  \newcommand*\fsize{\dimexpr\f@size pt\relax}%
  \newcommand*\lineheight[1]{\fontsize{\fsize}{#1\fsize}\selectfont}%
  \ifx\svgwidth\undefined%
    \setlength{\unitlength}{326.91670308bp}%
    \ifx\svgscale\undefined%
      \relax%
    \else%
      \setlength{\unitlength}{\unitlength * \real{\svgscale}}%
    \fi%
  \else%
    \setlength{\unitlength}{\svgwidth}%
  \fi%
  \global\let\svgwidth\undefined%
  \global\let\svgscale\undefined%
  \makeatother%
  \begin{picture}(1,0.17531172)%
    \lineheight{1}%
    \setlength\tabcolsep{0pt}%
    \put(0,0){\includegraphics[width=\unitlength,page=1]{signsum.pdf}}%
    \put(0.15322383,0.07601594){\makebox(0,0)[lt]{\lineheight{1.45000005}\smash{\begin{tabular}[t]{l}$=-$\end{tabular}}}}%
    \put(0.50346627,0.07579247){\makebox(0,0)[lt]{\lineheight{1.45000005}\smash{\begin{tabular}[t]{l}$\:+\;\cdots\; +$\end{tabular}}}}%
    \put(0.33387567,0.07657436){\makebox(0,0)[lt]{\lineheight{1.45000005}\smash{\begin{tabular}[t]{l}$\:+$\end{tabular}}}}%
    \put(0.76006418,0.07751645){\makebox(0,0)[lt]{\lineheight{1.45000005}\smash{\begin{tabular}[t]{l}$=:$\end{tabular}}}}%
  \end{picture}%
\endgroup%

%% file: composition.pdf_tex
\begingroup%
  \makeatletter%
  \providecommand\color[2][]{%
    \errmessage{(Inkscape) Color is used for the text in Inkscape, but the package 'color.sty' is not loaded}%
    \renewcommand\color[2][]{}%
  }%
  \providecommand\transparent[1]{%
    \errmessage{(Inkscape) Transparency is used (non-zero) for the text in Inkscape, but the package 'transparent.sty' is not loaded}%
    \renewcommand\transparent[1]{}%
  }%
  \providecommand\rotatebox[2]{#2}%
  \newcommand*\fsize{\dimexpr\f@size pt\relax}%
  \newcommand*\lineheight[1]{\fontsize{\fsize}{#1\fsize}\selectfont}%
  \ifx\svgwidth\undefined%
    \setlength{\unitlength}{241.95752342bp}%
    \ifx\svgscale\undefined%
      \relax%
    \else%
      \setlength{\unitlength}{\unitlength * \real{\svgscale}}%
    \fi%
  \else%
    \setlength{\unitlength}{\svgwidth}%
  \fi%
  \global\let\svgwidth\undefined%
  \global\let\svgscale\undefined%
  \makeatother%
  \begin{picture}(1,0.35969966)%
    \lineheight{1}%
    \setlength\tabcolsep{0pt}%
    \put(-0.00313502,0.13070803){\color[rgb]{0,0,0}\makebox(0,0)[lt]{\lineheight{1.25}\smash{\begin{tabular}[t]{l}$D'=$\\\end{tabular}}}}%
    \put(0.48888667,0.121523){\color[rgb]{0,0,0}\makebox(0,0)[lt]{\lineheight{1.25}\smash{\begin{tabular}[t]{l}$D\circ D'=$\end{tabular}}}}%
    \put(0,0){\includegraphics[width=\unitlength,page=1]{composition.pdf}}%
  \end{picture}%
\endgroup%

%% file: idn.pdf_tex
\begingroup%
  \makeatletter%
  \providecommand\color[2][]{%
    \errmessage{(Inkscape) Color is used for the text in Inkscape, but the package 'color.sty' is not loaded}%
    \renewcommand\color[2][]{}%
  }%
  \providecommand\transparent[1]{%
    \errmessage{(Inkscape) Transparency is used (non-zero) for the text in Inkscape, but the package 'transparent.sty' is not loaded}%
    \renewcommand\transparent[1]{}%
  }%
  \providecommand\rotatebox[2]{#2}%
  \newcommand*\fsize{\dimexpr\f@size pt\relax}%
  \newcommand*\lineheight[1]{\fontsize{\fsize}{#1\fsize}\selectfont}%
  \ifx\svgwidth\undefined%
    \setlength{\unitlength}{45.37417238bp}%
    \ifx\svgscale\undefined%
      \relax%
    \else%
      \setlength{\unitlength}{\unitlength * \real{\svgscale}}%
    \fi%
  \else%
    \setlength{\unitlength}{\svgwidth}%
  \fi%
  \global\let\svgwidth\undefined%
  \global\let\svgscale\undefined%
  \makeatother%
  \begin{picture}(1,1.06706211)%
    \lineheight{1}%
    \setlength\tabcolsep{0pt}%
    \put(0,0){\includegraphics[width=\unitlength,page=1]{idn.pdf}}%
    \put(0.41583703,0.01553918){\makebox(0,0)[lt]{\lineheight{1.45000005}\smash{\begin{tabular}[t]{l}$n$\end{tabular}}}}%
    \put(0,0){\includegraphics[width=\unitlength,page=2]{idn.pdf}}%
  \end{picture}%
\endgroup%

%% file: tensor.pdf_tex
\begingroup%
  \makeatletter%
  \providecommand\color[2][]{%
    \errmessage{(Inkscape) Color is used for the text in Inkscape, but the package 'color.sty' is not loaded}%
    \renewcommand\color[2][]{}%
  }%
  \providecommand\transparent[1]{%
    \errmessage{(Inkscape) Transparency is used (non-zero) for the text in Inkscape, but the package 'transparent.sty' is not loaded}%
    \renewcommand\transparent[1]{}%
  }%
  \providecommand\rotatebox[2]{#2}%
  \newcommand*\fsize{\dimexpr\f@size pt\relax}%
  \newcommand*\lineheight[1]{\fontsize{\fsize}{#1\fsize}\selectfont}%
  \ifx\svgwidth\undefined%
    \setlength{\unitlength}{230.74504522bp}%
    \ifx\svgscale\undefined%
      \relax%
    \else%
      \setlength{\unitlength}{\unitlength * \real{\svgscale}}%
    \fi%
  \else%
    \setlength{\unitlength}{\svgwidth}%
  \fi%
  \global\let\svgwidth\undefined%
  \global\let\svgscale\undefined%
  \makeatother%
  \begin{picture}(1,0.39939419)%
    \lineheight{1}%
    \setlength\tabcolsep{0pt}%
    \put(0.66775552,0.1699771){\color[rgb]{0,0,0}\makebox(0,0)[lt]{\lineheight{1.25}\smash{\begin{tabular}[t]{l}$=$\end{tabular}}}}%
    \put(0.14159487,0.00707571){\color[rgb]{0,0,0}\makebox(0,0)[lt]{\lineheight{1.25}\smash{\begin{tabular}[t]{l}$1$\end{tabular}}}}%
    \put(0.45106526,0.00918401){\color[rgb]{0,0,0}\makebox(0,0)[lt]{\lineheight{1.25}\smash{\begin{tabular}[t]{l}$1$\end{tabular}}}}%
    \put(0.57955996,0.00769481){\color[rgb]{0,0,0}\makebox(0,0)[lt]{\lineheight{1.25}\smash{\begin{tabular}[t]{l}$2$\end{tabular}}}}%
    \put(0.76149017,0.0099325){\color[rgb]{0,0,0}\makebox(0,0)[lt]{\lineheight{1.25}\smash{\begin{tabular}[t]{l}$1$\end{tabular}}}}%
    \put(0.84297638,0.01124503){\color[rgb]{0,0,0}\makebox(0,0)[lt]{\lineheight{1.25}\smash{\begin{tabular}[t]{l}$2$\end{tabular}}}}%
    \put(0.92614608,0.01289961){\color[rgb]{0,0,0}\makebox(0,0)[lt]{\lineheight{1.25}\smash{\begin{tabular}[t]{l}$3$\end{tabular}}}}%
    \put(0,0){\includegraphics[width=\unitlength,page=1]{tensor.pdf}}%
    \put(0.31360137,0.17120042){\makebox(0,0)[lt]{\lineheight{1.45000005}\smash{\begin{tabular}[t]{l}$\otimes$\end{tabular}}}}%
    \put(0,0){\includegraphics[width=\unitlength,page=2]{tensor.pdf}}%
  \end{picture}%
\endgroup%

%% file: symm.pdf_tex
\begingroup%
  \makeatletter%
  \providecommand\color[2][]{%
    \errmessage{(Inkscape) Color is used for the text in Inkscape, but the package 'color.sty' is not loaded}%
    \renewcommand\color[2][]{}%
  }%
  \providecommand\transparent[1]{%
    \errmessage{(Inkscape) Transparency is used (non-zero) for the text in Inkscape, but the package 'transparent.sty' is not loaded}%
    \renewcommand\transparent[1]{}%
  }%
  \providecommand\rotatebox[2]{#2}%
  \newcommand*\fsize{\dimexpr\f@size pt\relax}%
  \newcommand*\lineheight[1]{\fontsize{\fsize}{#1\fsize}\selectfont}%
  \ifx\svgwidth\undefined%
    \setlength{\unitlength}{136.87447458bp}%
    \ifx\svgscale\undefined%
      \relax%
    \else%
      \setlength{\unitlength}{\unitlength * \real{\svgscale}}%
    \fi%
  \else%
    \setlength{\unitlength}{\svgwidth}%
  \fi%
  \global\let\svgwidth\undefined%
  \global\let\svgscale\undefined%
  \makeatother%
  \begin{picture}(1,0.63731642)%
    \lineheight{1}%
    \setlength\tabcolsep{0pt}%
    \put(-0.00626769,0.26980707){\color[rgb]{0,0,0}\makebox(0,0)[lt]{\lineheight{1.25}\smash{\begin{tabular}[t]{l}$P_{1,1}=$\\\end{tabular}}}}%
    \put(0,0){\includegraphics[width=\unitlength,page=1]{symm.pdf}}%
  \end{picture}%
\endgroup%

%% file: multilinearity1.pdf_tex
\begingroup%
  \makeatletter%
  \providecommand\color[2][]{%
    \errmessage{(Inkscape) Color is used for the text in Inkscape, but the package 'color.sty' is not loaded}%
    \renewcommand\color[2][]{}%
  }%
  \providecommand\transparent[1]{%
    \errmessage{(Inkscape) Transparency is used (non-zero) for the text in Inkscape, but the package 'transparent.sty' is not loaded}%
    \renewcommand\transparent[1]{}%
  }%
  \providecommand\rotatebox[2]{#2}%
  \newcommand*\fsize{\dimexpr\f@size pt\relax}%
  \newcommand*\lineheight[1]{\fontsize{\fsize}{#1\fsize}\selectfont}%
  \ifx\svgwidth\undefined%
    \setlength{\unitlength}{44.12409082bp}%
    \ifx\svgscale\undefined%
      \relax%
    \else%
      \setlength{\unitlength}{\unitlength * \real{\svgscale}}%
    \fi%
  \else%
    \setlength{\unitlength}{\svgwidth}%
  \fi%
  \global\let\svgwidth\undefined%
  \global\let\svgscale\undefined%
  \makeatother%
  \begin{picture}(1,0.99751256)%
    \lineheight{1}%
    \setlength\tabcolsep{0pt}%
    \put(-0.01029148,0.89287143){\makebox(0,0)[lt]{\lineheight{1.45000005}\smash{\begin{tabular}[t]{l}$aw_1+bw_2$\end{tabular}}}}%
    \put(0,0){\includegraphics[width=\unitlength,page=1]{multilinearity1.pdf}}%
  \end{picture}%
\endgroup%

%% file: multilinearity2.pdf_tex
\begingroup%
  \makeatletter%
  \providecommand\color[2][]{%
    \errmessage{(Inkscape) Color is used for the text in Inkscape, but the package 'color.sty' is not loaded}%
    \renewcommand\color[2][]{}%
  }%
  \providecommand\transparent[1]{%
    \errmessage{(Inkscape) Transparency is used (non-zero) for the text in Inkscape, but the package 'transparent.sty' is not loaded}%
    \renewcommand\transparent[1]{}%
  }%
  \providecommand\rotatebox[2]{#2}%
  \newcommand*\fsize{\dimexpr\f@size pt\relax}%
  \newcommand*\lineheight[1]{\fontsize{\fsize}{#1\fsize}\selectfont}%
  \ifx\svgwidth\undefined%
    \setlength{\unitlength}{19.31545887bp}%
    \ifx\svgscale\undefined%
      \relax%
    \else%
      \setlength{\unitlength}{\unitlength * \real{\svgscale}}%
    \fi%
  \else%
    \setlength{\unitlength}{\svgwidth}%
  \fi%
  \global\let\svgwidth\undefined%
  \global\let\svgscale\undefined%
  \makeatother%
  \begin{picture}(1,2.22227014)%
    \lineheight{1}%
    \setlength\tabcolsep{0pt}%
    \put(-0.02350978,1.98322873){\makebox(0,0)[lt]{\lineheight{1.45000005}\smash{\begin{tabular}[t]{l}$w_1$\end{tabular}}}}%
    \put(0,0){\includegraphics[width=\unitlength,page=1]{multilinearity2.pdf}}%
  \end{picture}%
\endgroup%

%% file: multilinearity3.pdf_tex
\begingroup%
  \makeatletter%
  \providecommand\color[2][]{%
    \errmessage{(Inkscape) Color is used for the text in Inkscape, but the package 'color.sty' is not loaded}%
    \renewcommand\color[2][]{}%
  }%
  \providecommand\transparent[1]{%
    \errmessage{(Inkscape) Transparency is used (non-zero) for the text in Inkscape, but the package 'transparent.sty' is not loaded}%
    \renewcommand\transparent[1]{}%
  }%
  \providecommand\rotatebox[2]{#2}%
  \newcommand*\fsize{\dimexpr\f@size pt\relax}%
  \newcommand*\lineheight[1]{\fontsize{\fsize}{#1\fsize}\selectfont}%
  \ifx\svgwidth\undefined%
    \setlength{\unitlength}{19.31545887bp}%
    \ifx\svgscale\undefined%
      \relax%
    \else%
      \setlength{\unitlength}{\unitlength * \real{\svgscale}}%
    \fi%
  \else%
    \setlength{\unitlength}{\svgwidth}%
  \fi%
  \global\let\svgwidth\undefined%
  \global\let\svgscale\undefined%
  \makeatother%
  \begin{picture}(1,2.22227014)%
    \lineheight{1}%
    \setlength\tabcolsep{0pt}%
    \put(-0.02350978,1.98322873){\makebox(0,0)[lt]{\lineheight{1.45000005}\smash{\begin{tabular}[t]{l}$w_2$\end{tabular}}}}%
    \put(0,0){\includegraphics[width=\unitlength,page=1]{multilinearity3.pdf}}%
  \end{picture}%
\endgroup%

%% file: u2.pdf_tex
\begingroup%
  \makeatletter%
  \providecommand\color[2][]{%
    \errmessage{(Inkscape) Color is used for the text in Inkscape, but the package 'color.sty' is not loaded}%
    \renewcommand\color[2][]{}%
  }%
  \providecommand\transparent[1]{%
    \errmessage{(Inkscape) Transparency is used (non-zero) for the text in Inkscape, but the package 'transparent.sty' is not loaded}%
    \renewcommand\transparent[1]{}%
  }%
  \providecommand\rotatebox[2]{#2}%
  \newcommand*\fsize{\dimexpr\f@size pt\relax}%
  \newcommand*\lineheight[1]{\fontsize{\fsize}{#1\fsize}\selectfont}%
  \ifx\svgwidth\undefined%
    \setlength{\unitlength}{69.02081041bp}%
    \ifx\svgscale\undefined%
      \relax%
    \else%
      \setlength{\unitlength}{\unitlength * \real{\svgscale}}%
    \fi%
  \else%
    \setlength{\unitlength}{\svgwidth}%
  \fi%
  \global\let\svgwidth\undefined%
  \global\let\svgscale\undefined%
  \makeatother%
  \begin{picture}(1,0.22819203)%
    \lineheight{1}%
    \setlength\tabcolsep{0pt}%
    \put(-0.00438613,0.08828141){\makebox(0,0)[lt]{\lineheight{1.45000005}\smash{\begin{tabular}[t]{l}$v_1$\end{tabular}}}}%
    \put(0.82217786,0.08968742){\makebox(0,0)[lt]{\lineheight{1.45000005}\smash{\begin{tabular}[t]{l}$v_2$\end{tabular}}}}%
    \put(0,0){\includegraphics[width=\unitlength,page=1]{u2.pdf}}%
  \end{picture}%
\endgroup%

%% file: thetau2.pdf_tex
\begingroup%
  \makeatletter%
  \providecommand\color[2][]{%
    \errmessage{(Inkscape) Color is used for the text in Inkscape, but the package 'color.sty' is not loaded}%
    \renewcommand\color[2][]{}%
  }%
  \providecommand\transparent[1]{%
    \errmessage{(Inkscape) Transparency is used (non-zero) for the text in Inkscape, but the package 'transparent.sty' is not loaded}%
    \renewcommand\transparent[1]{}%
  }%
  \providecommand\rotatebox[2]{#2}%
  \newcommand*\fsize{\dimexpr\f@size pt\relax}%
  \newcommand*\lineheight[1]{\fontsize{\fsize}{#1\fsize}\selectfont}%
  \ifx\svgwidth\undefined%
    \setlength{\unitlength}{50.22484093bp}%
    \ifx\svgscale\undefined%
      \relax%
    \else%
      \setlength{\unitlength}{\unitlength * \real{\svgscale}}%
    \fi%
  \else%
    \setlength{\unitlength}{\svgwidth}%
  \fi%
  \global\let\svgwidth\undefined%
  \global\let\svgscale\undefined%
  \makeatother%
  \begin{picture}(1,0.86371758)%
    \lineheight{1}%
    \setlength\tabcolsep{0pt}%
    \put(0,0){\includegraphics[width=\unitlength,page=1]{thetau2.pdf}}%
    \put(0.14967359,0.01403842){\makebox(0,0)[lt]{\lineheight{1.45000005}\smash{\begin{tabular}[t]{l}$1$\end{tabular}}}}%
    \put(0.76176847,0.01576404){\makebox(0,0)[lt]{\lineheight{1.45000005}\smash{\begin{tabular}[t]{l}$2$\end{tabular}}}}%
    \put(0,0){\includegraphics[width=\unitlength,page=2]{thetau2.pdf}}%
  \end{picture}%
\endgroup%

%% file: grA2fu21.pdf_tex
\begingroup%
  \makeatletter%
  \providecommand\color[2][]{%
    \errmessage{(Inkscape) Color is used for the text in Inkscape, but the package 'color.sty' is not loaded}%
    \renewcommand\color[2][]{}%
  }%
  \providecommand\transparent[1]{%
    \errmessage{(Inkscape) Transparency is used (non-zero) for the text in Inkscape, but the package 'transparent.sty' is not loaded}%
    \renewcommand\transparent[1]{}%
  }%
  \providecommand\rotatebox[2]{#2}%
  \newcommand*\fsize{\dimexpr\f@size pt\relax}%
  \newcommand*\lineheight[1]{\fontsize{\fsize}{#1\fsize}\selectfont}%
  \ifx\svgwidth\undefined%
    \setlength{\unitlength}{53.99897926bp}%
    \ifx\svgscale\undefined%
      \relax%
    \else%
      \setlength{\unitlength}{\unitlength * \real{\svgscale}}%
    \fi%
  \else%
    \setlength{\unitlength}{\svgwidth}%
  \fi%
  \global\let\svgwidth\undefined%
  \global\let\svgscale\undefined%
  \makeatother%
  \begin{picture}(1,1.06076948)%
    \lineheight{1}%
    \setlength\tabcolsep{0pt}%
    \put(0,0){\includegraphics[width=\unitlength,page=1]{grA2fu21.pdf}}%
    \put(0.00834735,0.0192672){\makebox(0,0)[lt]{\lineheight{1.45000005}\smash{\begin{tabular}[t]{l}$1$\end{tabular}}}}%
    \put(0.7734996,0.01305724){\makebox(0,0)[lt]{\lineheight{1.45000005}\smash{\begin{tabular}[t]{l}$2$\end{tabular}}}}%
    \put(0,0){\includegraphics[width=\unitlength,page=2]{grA2fu21.pdf}}%
  \end{picture}%
\endgroup%

%% file: grA2fu23.pdf_tex
\begingroup%
  \makeatletter%
  \providecommand\color[2][]{%
    \errmessage{(Inkscape) Color is used for the text in Inkscape, but the package 'color.sty' is not loaded}%
    \renewcommand\color[2][]{}%
  }%
  \providecommand\transparent[1]{%
    \errmessage{(Inkscape) Transparency is used (non-zero) for the text in Inkscape, but the package 'transparent.sty' is not loaded}%
    \renewcommand\transparent[1]{}%
  }%
  \providecommand\rotatebox[2]{#2}%
  \newcommand*\fsize{\dimexpr\f@size pt\relax}%
  \newcommand*\lineheight[1]{\fontsize{\fsize}{#1\fsize}\selectfont}%
  \ifx\svgwidth\undefined%
    \setlength{\unitlength}{53.99897926bp}%
    \ifx\svgscale\undefined%
      \relax%
    \else%
      \setlength{\unitlength}{\unitlength * \real{\svgscale}}%
    \fi%
  \else%
    \setlength{\unitlength}{\svgwidth}%
  \fi%
  \global\let\svgwidth\undefined%
  \global\let\svgscale\undefined%
  \makeatother%
  \begin{picture}(1,1.05975031)%
    \lineheight{1}%
    \setlength\tabcolsep{0pt}%
    \put(0,0){\includegraphics[width=\unitlength,page=1]{grA2fu23.pdf}}%
    \put(0.00951967,0.01860718){\makebox(0,0)[lt]{\lineheight{1.45000005}\smash{\begin{tabular}[t]{l}$1$\end{tabular}}}}%
    \put(0.78924989,0.01305724){\makebox(0,0)[lt]{\lineheight{1.45000005}\smash{\begin{tabular}[t]{l}$2$\end{tabular}}}}%
    \put(0,0){\includegraphics[width=\unitlength,page=2]{grA2fu23.pdf}}%
  \end{picture}%
\endgroup%

%% file: grA2fu22.pdf_tex
\begingroup%
  \makeatletter%
  \providecommand\color[2][]{%
    \errmessage{(Inkscape) Color is used for the text in Inkscape, but the package 'color.sty' is not loaded}%
    \renewcommand\color[2][]{}%
  }%
  \providecommand\transparent[1]{%
    \errmessage{(Inkscape) Transparency is used (non-zero) for the text in Inkscape, but the package 'transparent.sty' is not loaded}%
    \renewcommand\transparent[1]{}%
  }%
  \providecommand\rotatebox[2]{#2}%
  \newcommand*\fsize{\dimexpr\f@size pt\relax}%
  \newcommand*\lineheight[1]{\fontsize{\fsize}{#1\fsize}\selectfont}%
  \ifx\svgwidth\undefined%
    \setlength{\unitlength}{53.99897926bp}%
    \ifx\svgscale\undefined%
      \relax%
    \else%
      \setlength{\unitlength}{\unitlength * \real{\svgscale}}%
    \fi%
  \else%
    \setlength{\unitlength}{\svgwidth}%
  \fi%
  \global\let\svgwidth\undefined%
  \global\let\svgscale\undefined%
  \makeatother%
  \begin{picture}(1,1.06416964)%
    \lineheight{1}%
    \setlength\tabcolsep{0pt}%
    \put(0,0){\includegraphics[width=\unitlength,page=1]{grA2fu22.pdf}}%
    \put(0.00951967,0.01860718){\makebox(0,0)[lt]{\lineheight{1.45000005}\smash{\begin{tabular}[t]{l}$1$\end{tabular}}}}%
    \put(0.76862641,0.01305724){\makebox(0,0)[lt]{\lineheight{1.45000005}\smash{\begin{tabular}[t]{l}$2$\end{tabular}}}}%
    \put(0,0){\includegraphics[width=\unitlength,page=2]{grA2fu22.pdf}}%
  \end{picture}%
\endgroup%

%% file: d_i,j.pdf_tex
\begingroup%
  \makeatletter%
  \providecommand\color[2][]{%
    \errmessage{(Inkscape) Color is used for the text in Inkscape, but the package 'color.sty' is not loaded}%
    \renewcommand\color[2][]{}%
  }%
  \providecommand\transparent[1]{%
    \errmessage{(Inkscape) Transparency is used (non-zero) for the text in Inkscape, but the package 'transparent.sty' is not loaded}%
    \renewcommand\transparent[1]{}%
  }%
  \providecommand\rotatebox[2]{#2}%
  \newcommand*\fsize{\dimexpr\f@size pt\relax}%
  \newcommand*\lineheight[1]{\fontsize{\fsize}{#1\fsize}\selectfont}%
  \ifx\svgwidth\undefined%
    \setlength{\unitlength}{36.32825996bp}%
    \ifx\svgscale\undefined%
      \relax%
    \else%
      \setlength{\unitlength}{\unitlength * \real{\svgscale}}%
    \fi%
  \else%
    \setlength{\unitlength}{\svgwidth}%
  \fi%
  \global\let\svgwidth\undefined%
  \global\let\svgscale\undefined%
  \makeatother%
  \begin{picture}(1,0.05849441)%
    \lineheight{1}%
    \setlength\tabcolsep{0pt}%
    \put(0,0){\includegraphics[width=\unitlength,page=1]{d_i,j.pdf}}%
    \put(-0.00416665,0.00970427){\makebox(0,0)[lt]{\lineheight{1.45000005}\smash{\begin{tabular}[t]{l}$v_i$\end{tabular}}}}%
    \put(0.84720478,0.01612895){\makebox(0,0)[lt]{\lineheight{1.45000005}\smash{\begin{tabular}[t]{l}$v_j$\end{tabular}}}}%
  \end{picture}%
\endgroup%

%% file: c_i,i.pdf_tex
\begingroup%
  \makeatletter%
  \providecommand\color[2][]{%
    \errmessage{(Inkscape) Color is used for the text in Inkscape, but the package 'color.sty' is not loaded}%
    \renewcommand\color[2][]{}%
  }%
  \providecommand\transparent[1]{%
    \errmessage{(Inkscape) Transparency is used (non-zero) for the text in Inkscape, but the package 'transparent.sty' is not loaded}%
    \renewcommand\transparent[1]{}%
  }%
  \providecommand\rotatebox[2]{#2}%
  \newcommand*\fsize{\dimexpr\f@size pt\relax}%
  \newcommand*\lineheight[1]{\fontsize{\fsize}{#1\fsize}\selectfont}%
  \ifx\svgwidth\undefined%
    \setlength{\unitlength}{75.29748464bp}%
    \ifx\svgscale\undefined%
      \relax%
    \else%
      \setlength{\unitlength}{\unitlength * \real{\svgscale}}%
    \fi%
  \else%
    \setlength{\unitlength}{\svgwidth}%
  \fi%
  \global\let\svgwidth\undefined%
  \global\let\svgscale\undefined%
  \makeatother%
  \begin{picture}(1,0.38109198)%
    \lineheight{1}%
    \setlength\tabcolsep{0pt}%
    \put(0,0){\includegraphics[width=\unitlength,page=1]{c_i,i.pdf}}%
    \put(0.08004113,0.01082258){\color[rgb]{0,0,0}\makebox(0,0)[lt]{\lineheight{1.25}\smash{\begin{tabular}[t]{l}$1$\end{tabular}}}}%
    \put(0.49148156,0.00895937){\color[rgb]{0,0,0}\makebox(0,0)[lt]{\lineheight{1.25}\smash{\begin{tabular}[t]{l}$i$\end{tabular}}}}%
    \put(0.87455503,0.01382061){\color[rgb]{0,0,0}\makebox(0,0)[lt]{\lineheight{1.25}\smash{\begin{tabular}[t]{l}$n$\end{tabular}}}}%
    \put(0,0){\includegraphics[width=\unitlength,page=2]{c_i,i.pdf}}%
  \end{picture}%
\endgroup%

%% file: c_i,j.pdf_tex
\begingroup%
  \makeatletter%
  \providecommand\color[2][]{%
    \errmessage{(Inkscape) Color is used for the text in Inkscape, but the package 'color.sty' is not loaded}%
    \renewcommand\color[2][]{}%
  }%
  \providecommand\transparent[1]{%
    \errmessage{(Inkscape) Transparency is used (non-zero) for the text in Inkscape, but the package 'transparent.sty' is not loaded}%
    \renewcommand\transparent[1]{}%
  }%
  \providecommand\rotatebox[2]{#2}%
  \newcommand*\fsize{\dimexpr\f@size pt\relax}%
  \newcommand*\lineheight[1]{\fontsize{\fsize}{#1\fsize}\selectfont}%
  \ifx\svgwidth\undefined%
    \setlength{\unitlength}{73.70679851bp}%
    \ifx\svgscale\undefined%
      \relax%
    \else%
      \setlength{\unitlength}{\unitlength * \real{\svgscale}}%
    \fi%
  \else%
    \setlength{\unitlength}{\svgwidth}%
  \fi%
  \global\let\svgwidth\undefined%
  \global\let\svgscale\undefined%
  \makeatother%
  \begin{picture}(1,0.34680935)%
    \lineheight{1}%
    \setlength\tabcolsep{0pt}%
    \put(0,0){\includegraphics[width=\unitlength,page=1]{c_i,j.pdf}}%
    \put(0.06309489,0.01240675){\color[rgb]{0,0,0}\makebox(0,0)[lt]{\lineheight{1.25}\smash{\begin{tabular}[t]{l}$1$\end{tabular}}}}%
    \put(0.3452176,0.00726286){\color[rgb]{0,0,0}\makebox(0,0)[lt]{\lineheight{1.25}\smash{\begin{tabular}[t]{l}$i$\end{tabular}}}}%
    \put(0.6188899,0.01207061){\color[rgb]{0,0,0}\makebox(0,0)[lt]{\lineheight{1.25}\smash{\begin{tabular}[t]{l}$j$\end{tabular}}}}%
    \put(0.89185603,0.00804456){\color[rgb]{0,0,0}\makebox(0,0)[lt]{\lineheight{1.25}\smash{\begin{tabular}[t]{l}$n$\end{tabular}}}}%
    \put(0,0){\includegraphics[width=\unitlength,page=2]{c_i,j.pdf}}%
  \end{picture}%
\endgroup%

%% file: u.pdf_tex
\begingroup%
  \makeatletter%
  \providecommand\color[2][]{%
    \errmessage{(Inkscape) Color is used for the text in Inkscape, but the package 'color.sty' is not loaded}%
    \renewcommand\color[2][]{}%
  }%
  \providecommand\transparent[1]{%
    \errmessage{(Inkscape) Transparency is used (non-zero) for the text in Inkscape, but the package 'transparent.sty' is not loaded}%
    \renewcommand\transparent[1]{}%
  }%
  \providecommand\rotatebox[2]{#2}%
  \newcommand*\fsize{\dimexpr\f@size pt\relax}%
  \newcommand*\lineheight[1]{\fontsize{\fsize}{#1\fsize}\selectfont}%
  \ifx\svgwidth\undefined%
    \setlength{\unitlength}{72.71900792bp}%
    \ifx\svgscale\undefined%
      \relax%
    \else%
      \setlength{\unitlength}{\unitlength * \real{\svgscale}}%
    \fi%
  \else%
    \setlength{\unitlength}{\svgwidth}%
  \fi%
  \global\let\svgwidth\undefined%
  \global\let\svgscale\undefined%
  \makeatother%
  \begin{picture}(1,0.69989492)%
    \lineheight{1}%
    \setlength\tabcolsep{0pt}%
    \put(0,0){\includegraphics[width=\unitlength,page=1]{u.pdf}}%
    \put(0.07870766,0.00969592){\makebox(0,0)[lt]{\lineheight{1.45000005}\smash{\begin{tabular}[t]{l}$1$\end{tabular}}}}%
    \put(0.38880482,0.01103444){\makebox(0,0)[lt]{\lineheight{1.45000005}\smash{\begin{tabular}[t]{l}$2$\end{tabular}}}}%
    \put(0.85610042,0.01576503){\makebox(0,0)[lt]{\lineheight{1.45000005}\smash{\begin{tabular}[t]{l}$n$\end{tabular}}}}%
    \put(0,0){\includegraphics[width=\unitlength,page=2]{u.pdf}}%
    \put(0.46151562,0.58874351){\makebox(0,0)[lt]{\lineheight{1.45000005}\smash{\begin{tabular}[t]{l}$u$\end{tabular}}}}%
  \end{picture}%
\endgroup%

%% file: uK22.pdf_tex
\begingroup%
  \makeatletter%
  \providecommand\color[2][]{%
    \errmessage{(Inkscape) Color is used for the text in Inkscape, but the package 'color.sty' is not loaded}%
    \renewcommand\color[2][]{}%
  }%
  \providecommand\transparent[1]{%
    \errmessage{(Inkscape) Transparency is used (non-zero) for the text in Inkscape, but the package 'transparent.sty' is not loaded}%
    \renewcommand\transparent[1]{}%
  }%
  \providecommand\rotatebox[2]{#2}%
  \newcommand*\fsize{\dimexpr\f@size pt\relax}%
  \newcommand*\lineheight[1]{\fontsize{\fsize}{#1\fsize}\selectfont}%
  \ifx\svgwidth\undefined%
    \setlength{\unitlength}{104.85941456bp}%
    \ifx\svgscale\undefined%
      \relax%
    \else%
      \setlength{\unitlength}{\unitlength * \real{\svgscale}}%
    \fi%
  \else%
    \setlength{\unitlength}{\svgwidth}%
  \fi%
  \global\let\svgwidth\undefined%
  \global\let\svgscale\undefined%
  \makeatother%
  \begin{picture}(1,0.75790196)%
    \lineheight{1}%
    \setlength\tabcolsep{0pt}%
    \put(0,0){\includegraphics[width=\unitlength,page=1]{uK22.pdf}}%
    \put(0.1477284,0.00995825){\makebox(0,0)[lt]{\lineheight{1.45000005}\smash{\begin{tabular}[t]{l}$1$\end{tabular}}}}%
    \put(0.43446711,0.00515219){\makebox(0,0)[lt]{\lineheight{1.45000005}\smash{\begin{tabular}[t]{l}$2$\end{tabular}}}}%
    \put(0.86291602,0.01186214){\makebox(0,0)[lt]{\lineheight{1.45000005}\smash{\begin{tabular}[t]{l}$n$\end{tabular}}}}%
    \put(0,0){\includegraphics[width=\unitlength,page=2]{uK22.pdf}}%
    \put(0.48301583,0.65358089){\makebox(0,0)[lt]{\lineheight{1.45000005}\smash{\begin{tabular}[t]{l}$u$\end{tabular}}}}%
    \put(0,0){\includegraphics[width=\unitlength,page=3]{uK22.pdf}}%
  \end{picture}%
\endgroup%

%% file: boxarc1.pdf_tex
\begingroup%
  \makeatletter%
  \providecommand\color[2][]{%
    \errmessage{(Inkscape) Color is used for the text in Inkscape, but the package 'color.sty' is not loaded}%
    \renewcommand\color[2][]{}%
  }%
  \providecommand\transparent[1]{%
    \errmessage{(Inkscape) Transparency is used (non-zero) for the text in Inkscape, but the package 'transparent.sty' is not loaded}%
    \renewcommand\transparent[1]{}%
  }%
  \providecommand\rotatebox[2]{#2}%
  \newcommand*\fsize{\dimexpr\f@size pt\relax}%
  \newcommand*\lineheight[1]{\fontsize{\fsize}{#1\fsize}\selectfont}%
  \ifx\svgwidth\undefined%
    \setlength{\unitlength}{34.12500065bp}%
    \ifx\svgscale\undefined%
      \relax%
    \else%
      \setlength{\unitlength}{\unitlength * \real{\svgscale}}%
    \fi%
  \else%
    \setlength{\unitlength}{\svgwidth}%
  \fi%
  \global\let\svgwidth\undefined%
  \global\let\svgscale\undefined%
  \makeatother%
  \begin{picture}(1,0.89946402)%
    \lineheight{1}%
    \setlength\tabcolsep{0pt}%
    \put(0,0){\includegraphics[width=\unitlength,page=1]{boxarc1.pdf}}%
  \end{picture}%
\endgroup%

%% file: boxarc2.pdf_tex
\begingroup%
  \makeatletter%
  \providecommand\color[2][]{%
    \errmessage{(Inkscape) Color is used for the text in Inkscape, but the package 'color.sty' is not loaded}%
    \renewcommand\color[2][]{}%
  }%
  \providecommand\transparent[1]{%
    \errmessage{(Inkscape) Transparency is used (non-zero) for the text in Inkscape, but the package 'transparent.sty' is not loaded}%
    \renewcommand\transparent[1]{}%
  }%
  \providecommand\rotatebox[2]{#2}%
  \newcommand*\fsize{\dimexpr\f@size pt\relax}%
  \newcommand*\lineheight[1]{\fontsize{\fsize}{#1\fsize}\selectfont}%
  \ifx\svgwidth\undefined%
    \setlength{\unitlength}{34.12500065bp}%
    \ifx\svgscale\undefined%
      \relax%
    \else%
      \setlength{\unitlength}{\unitlength * \real{\svgscale}}%
    \fi%
  \else%
    \setlength{\unitlength}{\svgwidth}%
  \fi%
  \global\let\svgwidth\undefined%
  \global\let\svgscale\undefined%
  \makeatother%
  \begin{picture}(1,0.89946402)%
    \lineheight{1}%
    \setlength\tabcolsep{0pt}%
    \put(0,0){\includegraphics[width=\unitlength,page=1]{boxarc2.pdf}}%
  \end{picture}%
\endgroup%

%% file: uK23.pdf_tex
\begingroup%
  \makeatletter%
  \providecommand\color[2][]{%
    \errmessage{(Inkscape) Color is used for the text in Inkscape, but the package 'color.sty' is not loaded}%
    \renewcommand\color[2][]{}%
  }%
  \providecommand\transparent[1]{%
    \errmessage{(Inkscape) Transparency is used (non-zero) for the text in Inkscape, but the package 'transparent.sty' is not loaded}%
    \renewcommand\transparent[1]{}%
  }%
  \providecommand\rotatebox[2]{#2}%
  \newcommand*\fsize{\dimexpr\f@size pt\relax}%
  \newcommand*\lineheight[1]{\fontsize{\fsize}{#1\fsize}\selectfont}%
  \ifx\svgwidth\undefined%
    \setlength{\unitlength}{105.29880857bp}%
    \ifx\svgscale\undefined%
      \relax%
    \else%
      \setlength{\unitlength}{\unitlength * \real{\svgscale}}%
    \fi%
  \else%
    \setlength{\unitlength}{\svgwidth}%
  \fi%
  \global\let\svgwidth\undefined%
  \global\let\svgscale\undefined%
  \makeatother%
  \begin{picture}(1,0.63683198)%
    \lineheight{1}%
    \setlength\tabcolsep{0pt}%
    \put(0,0){\includegraphics[width=\unitlength,page=1]{uK23.pdf}}%
    \put(0.0613269,0.00882228){\makebox(0,0)[lt]{\lineheight{1.45000005}\smash{\begin{tabular}[t]{l}$1$\end{tabular}}}}%
    \put(0.34348328,0.01004021){\makebox(0,0)[lt]{\lineheight{1.45000005}\smash{\begin{tabular}[t]{l}$2$\end{tabular}}}}%
    \put(0.76867393,0.01434453){\makebox(0,0)[lt]{\lineheight{1.45000005}\smash{\begin{tabular}[t]{l}$n$\end{tabular}}}}%
    \put(0,0){\includegraphics[width=\unitlength,page=2]{uK23.pdf}}%
    \put(0.54076533,0.01426736){\makebox(0,0)[lt]{\lineheight{1.45000005}\smash{\begin{tabular}[t]{l}$\cdots$\end{tabular}}}}%
    \put(0.39796955,0.53824668){\makebox(0,0)[lt]{\lineheight{1.45000005}\smash{\begin{tabular}[t]{l}$u$\end{tabular}}}}%
    \put(0,0){\includegraphics[width=\unitlength,page=3]{uK23.pdf}}%
  \end{picture}%
\endgroup%

%% file: uK3.pdf_tex
\begingroup%
  \makeatletter%
  \providecommand\color[2][]{%
    \errmessage{(Inkscape) Color is used for the text in Inkscape, but the package 'color.sty' is not loaded}%
    \renewcommand\color[2][]{}%
  }%
  \providecommand\transparent[1]{%
    \errmessage{(Inkscape) Transparency is used (non-zero) for the text in Inkscape, but the package 'transparent.sty' is not loaded}%
    \renewcommand\transparent[1]{}%
  }%
  \providecommand\rotatebox[2]{#2}%
  \newcommand*\fsize{\dimexpr\f@size pt\relax}%
  \newcommand*\lineheight[1]{\fontsize{\fsize}{#1\fsize}\selectfont}%
  \ifx\svgwidth\undefined%
    \setlength{\unitlength}{112.13948329bp}%
    \ifx\svgscale\undefined%
      \relax%
    \else%
      \setlength{\unitlength}{\unitlength * \real{\svgscale}}%
    \fi%
  \else%
    \setlength{\unitlength}{\svgwidth}%
  \fi%
  \global\let\svgwidth\undefined%
  \global\let\svgscale\undefined%
  \makeatother%
  \begin{picture}(1,0.82320112)%
    \lineheight{1}%
    \setlength\tabcolsep{0pt}%
    \put(0,0){\includegraphics[width=\unitlength,page=1]{uK3.pdf}}%
    \put(0.11665926,0.00484196){\makebox(0,0)[lt]{\lineheight{1.45000005}\smash{\begin{tabular}[t]{l}$1$\end{tabular}}}}%
    \put(0.38160369,0.00598549){\makebox(0,0)[lt]{\lineheight{1.45000005}\smash{\begin{tabular}[t]{l}$2$\end{tabular}}}}%
    \put(0.7808571,0.01002725){\makebox(0,0)[lt]{\lineheight{1.45000005}\smash{\begin{tabular}[t]{l}$n$\end{tabular}}}}%
    \put(0.56685127,0.00995483){\makebox(0,0)[lt]{\lineheight{1.45000005}\smash{\begin{tabular}[t]{l}$\cdots$\end{tabular}}}}%
    \put(0,0){\includegraphics[width=\unitlength,page=2]{uK3.pdf}}%
    \put(0.39970209,0.50925388){\makebox(0,0)[lt]{\lineheight{1.45000005}\smash{\begin{tabular}[t]{l}$u$\end{tabular}}}}%
    \put(0,0){\includegraphics[width=\unitlength,page=3]{uK3.pdf}}%
  \end{picture}%
\endgroup%

%% file: uK31.pdf_tex
\begingroup%
  \makeatletter%
  \providecommand\color[2][]{%
    \errmessage{(Inkscape) Color is used for the text in Inkscape, but the package 'color.sty' is not loaded}%
    \renewcommand\color[2][]{}%
  }%
  \providecommand\transparent[1]{%
    \errmessage{(Inkscape) Transparency is used (non-zero) for the text in Inkscape, but the package 'transparent.sty' is not loaded}%
    \renewcommand\transparent[1]{}%
  }%
  \providecommand\rotatebox[2]{#2}%
  \newcommand*\fsize{\dimexpr\f@size pt\relax}%
  \newcommand*\lineheight[1]{\fontsize{\fsize}{#1\fsize}\selectfont}%
  \ifx\svgwidth\undefined%
    \setlength{\unitlength}{119.63948339bp}%
    \ifx\svgscale\undefined%
      \relax%
    \else%
      \setlength{\unitlength}{\unitlength * \real{\svgscale}}%
    \fi%
  \else%
    \setlength{\unitlength}{\svgwidth}%
  \fi%
  \global\let\svgwidth\undefined%
  \global\let\svgscale\undefined%
  \makeatother%
  \begin{picture}(1,0.77482238)%
    \lineheight{1}%
    \setlength\tabcolsep{0pt}%
    \put(0,0){\includegraphics[width=\unitlength,page=1]{uK31.pdf}}%
    \put(0.17203442,0.00776479){\makebox(0,0)[lt]{\lineheight{1.45000005}\smash{\begin{tabular}[t]{l}$1$\end{tabular}}}}%
    \put(0.42036992,0.00883664){\makebox(0,0)[lt]{\lineheight{1.45000005}\smash{\begin{tabular}[t]{l}$2$\end{tabular}}}}%
    \put(0.7945948,0.01262502){\makebox(0,0)[lt]{\lineheight{1.45000005}\smash{\begin{tabular}[t]{l}$n$\end{tabular}}}}%
    \put(0,0){\includegraphics[width=\unitlength,page=2]{uK31.pdf}}%
    \put(0.59400464,0.01255715){\makebox(0,0)[lt]{\lineheight{1.45000005}\smash{\begin{tabular}[t]{l}$\cdots$\end{tabular}}}}%
    \put(0,0){\includegraphics[width=\unitlength,page=3]{uK31.pdf}}%
    \put(0.43733377,0.48055597){\makebox(0,0)[lt]{\lineheight{1.45000005}\smash{\begin{tabular}[t]{l}$u$\end{tabular}}}}%
    \put(0,0){\includegraphics[width=\unitlength,page=4]{uK31.pdf}}%
  \end{picture}%
\endgroup%

%% file: uK4.pdf_tex
\begingroup%
  \makeatletter%
  \providecommand\color[2][]{%
    \errmessage{(Inkscape) Color is used for the text in Inkscape, but the package 'color.sty' is not loaded}%
    \renewcommand\color[2][]{}%
  }%
  \providecommand\transparent[1]{%
    \errmessage{(Inkscape) Transparency is used (non-zero) for the text in Inkscape, but the package 'transparent.sty' is not loaded}%
    \renewcommand\transparent[1]{}%
  }%
  \providecommand\rotatebox[2]{#2}%
  \newcommand*\fsize{\dimexpr\f@size pt\relax}%
  \newcommand*\lineheight[1]{\fontsize{\fsize}{#1\fsize}\selectfont}%
  \ifx\svgwidth\undefined%
    \setlength{\unitlength}{116.63948329bp}%
    \ifx\svgscale\undefined%
      \relax%
    \else%
      \setlength{\unitlength}{\unitlength * \real{\svgscale}}%
    \fi%
  \else%
    \setlength{\unitlength}{\svgwidth}%
  \fi%
  \global\let\svgwidth\undefined%
  \global\let\svgscale\undefined%
  \makeatother%
  \begin{picture}(1,0.79475102)%
    \lineheight{1}%
    \setlength\tabcolsep{0pt}%
    \put(0,0){\includegraphics[width=\unitlength,page=1]{uK4.pdf}}%
    \put(0.15073892,0.00796451){\makebox(0,0)[lt]{\lineheight{1.45000005}\smash{\begin{tabular}[t]{l}$1$\end{tabular}}}}%
    \put(0.40546167,0.00906392){\makebox(0,0)[lt]{\lineheight{1.45000005}\smash{\begin{tabular}[t]{l}$2$\end{tabular}}}}%
    \put(0.78931173,0.01294974){\makebox(0,0)[lt]{\lineheight{1.45000005}\smash{\begin{tabular}[t]{l}$n$\end{tabular}}}}%
    \put(0,0){\includegraphics[width=\unitlength,page=2]{uK4.pdf}}%
    \put(0.58356232,0.01288012){\makebox(0,0)[lt]{\lineheight{1.45000005}\smash{\begin{tabular}[t]{l}$\cdots$\end{tabular}}}}%
    \put(0,0){\includegraphics[width=\unitlength,page=3]{uK4.pdf}}%
    \put(0.42286184,0.492916){\makebox(0,0)[lt]{\lineheight{1.45000005}\smash{\begin{tabular}[t]{l}$u$\end{tabular}}}}%
  \end{picture}%
\endgroup%

%% file: B20.pdf_tex
\begingroup%
  \makeatletter%
  \providecommand\color[2][]{%
    \errmessage{(Inkscape) Color is used for the text in Inkscape, but the package 'color.sty' is not loaded}%
    \renewcommand\color[2][]{}%
  }%
  \providecommand\transparent[1]{%
    \errmessage{(Inkscape) Transparency is used (non-zero) for the text in Inkscape, but the package 'transparent.sty' is not loaded}%
    \renewcommand\transparent[1]{}%
  }%
  \providecommand\rotatebox[2]{#2}%
  \newcommand*\fsize{\dimexpr\f@size pt\relax}%
  \newcommand*\lineheight[1]{\fontsize{\fsize}{#1\fsize}\selectfont}%
  \ifx\svgwidth\undefined%
    \setlength{\unitlength}{79.66166656bp}%
    \ifx\svgscale\undefined%
      \relax%
    \else%
      \setlength{\unitlength}{\unitlength * \real{\svgscale}}%
    \fi%
  \else%
    \setlength{\unitlength}{\svgwidth}%
  \fi%
  \global\let\svgwidth\undefined%
  \global\let\svgscale\undefined%
  \makeatother%
  \begin{picture}(1,0.0257271)%
    \lineheight{1}%
    \setlength\tabcolsep{0pt}%
    \put(0,0){\includegraphics[width=\unitlength,page=1]{B20.pdf}}%
    \put(-0.00190013,0.00594414){\makebox(0,0)[lt]{\lineheight{1.45000005}\smash{\begin{tabular}[t]{l}$w_1$\end{tabular}}}}%
    \put(0.35465186,0.00442546){\makebox(0,0)[lt]{\lineheight{1.45000005}\smash{\begin{tabular}[t]{l}$w_2$\end{tabular}}}}%
    \put(0.54873588,0.0054381){\makebox(0,0)[lt]{\lineheight{1.45000005}\smash{\begin{tabular}[t]{l}$w_3$\end{tabular}}}}%
    \put(0.91727706,0.0064071){\makebox(0,0)[lt]{\lineheight{1.45000005}\smash{\begin{tabular}[t]{l}$w_4$\end{tabular}}}}%
  \end{picture}%
\endgroup%

%% file: B101.pdf_tex
\begingroup%
  \makeatletter%
  \providecommand\color[2][]{%
    \errmessage{(Inkscape) Color is used for the text in Inkscape, but the package 'color.sty' is not loaded}%
    \renewcommand\color[2][]{}%
  }%
  \providecommand\transparent[1]{%
    \errmessage{(Inkscape) Transparency is used (non-zero) for the text in Inkscape, but the package 'transparent.sty' is not loaded}%
    \renewcommand\transparent[1]{}%
  }%
  \providecommand\rotatebox[2]{#2}%
  \newcommand*\fsize{\dimexpr\f@size pt\relax}%
  \newcommand*\lineheight[1]{\fontsize{\fsize}{#1\fsize}\selectfont}%
  \ifx\svgwidth\undefined%
    \setlength{\unitlength}{33.51865928bp}%
    \ifx\svgscale\undefined%
      \relax%
    \else%
      \setlength{\unitlength}{\unitlength * \real{\svgscale}}%
    \fi%
  \else%
    \setlength{\unitlength}{\svgwidth}%
  \fi%
  \global\let\svgwidth\undefined%
  \global\let\svgscale\undefined%
  \makeatother%
  \begin{picture}(1,0.05719829)%
    \lineheight{1}%
    \setlength\tabcolsep{0pt}%
    \put(0,0){\includegraphics[width=\unitlength,page=1]{B101.pdf}}%
    \put(-0.00451591,0.01128167){\makebox(0,0)[lt]{\lineheight{1.45000005}\smash{\begin{tabular}[t]{l}$w_1$\end{tabular}}}}%
    \put(0.80339764,0.0105177){\makebox(0,0)[lt]{\lineheight{1.45000005}\smash{\begin{tabular}[t]{l}$w_2$\end{tabular}}}}%
  \end{picture}%
\endgroup%

%% file: B102.pdf_tex
\begingroup%
  \makeatletter%
  \providecommand\color[2][]{%
    \errmessage{(Inkscape) Color is used for the text in Inkscape, but the package 'color.sty' is not loaded}%
    \renewcommand\color[2][]{}%
  }%
  \providecommand\transparent[1]{%
    \errmessage{(Inkscape) Transparency is used (non-zero) for the text in Inkscape, but the package 'transparent.sty' is not loaded}%
    \renewcommand\transparent[1]{}%
  }%
  \providecommand\rotatebox[2]{#2}%
  \newcommand*\fsize{\dimexpr\f@size pt\relax}%
  \newcommand*\lineheight[1]{\fontsize{\fsize}{#1\fsize}\selectfont}%
  \ifx\svgwidth\undefined%
    \setlength{\unitlength}{33.51865928bp}%
    \ifx\svgscale\undefined%
      \relax%
    \else%
      \setlength{\unitlength}{\unitlength * \real{\svgscale}}%
    \fi%
  \else%
    \setlength{\unitlength}{\svgwidth}%
  \fi%
  \global\let\svgwidth\undefined%
  \global\let\svgscale\undefined%
  \makeatother%
  \begin{picture}(1,0.05719829)%
    \lineheight{1}%
    \setlength\tabcolsep{0pt}%
    \put(0,0){\includegraphics[width=\unitlength,page=1]{B102.pdf}}%
    \put(-0.00451591,0.01128167){\makebox(0,0)[lt]{\lineheight{1.45000005}\smash{\begin{tabular}[t]{l}$w_3$\end{tabular}}}}%
    \put(0.80339764,0.0105177){\makebox(0,0)[lt]{\lineheight{1.45000005}\smash{\begin{tabular}[t]{l}$w_4$\end{tabular}}}}%
  \end{picture}%
\endgroup%

%% file: mu.pdf_tex
\begingroup%
  \makeatletter%
  \providecommand\color[2][]{%
    \errmessage{(Inkscape) Color is used for the text in Inkscape, but the package 'color.sty' is not loaded}%
    \renewcommand\color[2][]{}%
  }%
  \providecommand\transparent[1]{%
    \errmessage{(Inkscape) Transparency is used (non-zero) for the text in Inkscape, but the package 'transparent.sty' is not loaded}%
    \renewcommand\transparent[1]{}%
  }%
  \providecommand\rotatebox[2]{#2}%
  \newcommand*\fsize{\dimexpr\f@size pt\relax}%
  \newcommand*\lineheight[1]{\fontsize{\fsize}{#1\fsize}\selectfont}%
  \ifx\svgwidth\undefined%
    \setlength{\unitlength}{30.7499995bp}%
    \ifx\svgscale\undefined%
      \relax%
    \else%
      \setlength{\unitlength}{\unitlength * \real{\svgscale}}%
    \fi%
  \else%
    \setlength{\unitlength}{\svgwidth}%
  \fi%
  \global\let\svgwidth\undefined%
  \global\let\svgscale\undefined%
  \makeatother%
  \begin{picture}(1,1.30247283)%
    \lineheight{1}%
    \setlength\tabcolsep{0pt}%
    \put(0,0){\includegraphics[width=\unitlength,page=1]{mu.pdf}}%
  \end{picture}%
\endgroup%

%% file: etainD.pdf_tex
\begingroup%
  \makeatletter%
  \providecommand\color[2][]{%
    \errmessage{(Inkscape) Color is used for the text in Inkscape, but the package 'color.sty' is not loaded}%
    \renewcommand\color[2][]{}%
  }%
  \providecommand\transparent[1]{%
    \errmessage{(Inkscape) Transparency is used (non-zero) for the text in Inkscape, but the package 'transparent.sty' is not loaded}%
    \renewcommand\transparent[1]{}%
  }%
  \providecommand\rotatebox[2]{#2}%
  \newcommand*\fsize{\dimexpr\f@size pt\relax}%
  \newcommand*\lineheight[1]{\fontsize{\fsize}{#1\fsize}\selectfont}%
  \ifx\svgwidth\undefined%
    \setlength{\unitlength}{30.7499995bp}%
    \ifx\svgscale\undefined%
      \relax%
    \else%
      \setlength{\unitlength}{\unitlength * \real{\svgscale}}%
    \fi%
  \else%
    \setlength{\unitlength}{\svgwidth}%
  \fi%
  \global\let\svgwidth\undefined%
  \global\let\svgscale\undefined%
  \makeatother%
  \begin{picture}(1,1.01165838)%
    \lineheight{1}%
    \setlength\tabcolsep{0pt}%
    \put(0,0){\includegraphics[width=\unitlength,page=1]{etainD.pdf}}%
  \end{picture}%
\endgroup%

%% file: cinA.pdf_tex
\begingroup%
  \makeatletter%
  \providecommand\color[2][]{%
    \errmessage{(Inkscape) Color is used for the text in Inkscape, but the package 'color.sty' is not loaded}%
    \renewcommand\color[2][]{}%
  }%
  \providecommand\transparent[1]{%
    \errmessage{(Inkscape) Transparency is used (non-zero) for the text in Inkscape, but the package 'transparent.sty' is not loaded}%
    \renewcommand\transparent[1]{}%
  }%
  \providecommand\rotatebox[2]{#2}%
  \newcommand*\fsize{\dimexpr\f@size pt\relax}%
  \newcommand*\lineheight[1]{\fontsize{\fsize}{#1\fsize}\selectfont}%
  \ifx\svgwidth\undefined%
    \setlength{\unitlength}{30.75118016bp}%
    \ifx\svgscale\undefined%
      \relax%
    \else%
      \setlength{\unitlength}{\unitlength * \real{\svgscale}}%
    \fi%
  \else%
    \setlength{\unitlength}{\svgwidth}%
  \fi%
  \global\let\svgwidth\undefined%
  \global\let\svgscale\undefined%
  \makeatother%
  \begin{picture}(1,1.01163876)%
    \lineheight{1}%
    \setlength\tabcolsep{0pt}%
    \put(0,0){\includegraphics[width=\unitlength,page=1]{cinA.pdf}}%
  \end{picture}%
\endgroup%

%% file: symn.pdf_tex
\begingroup%
  \makeatletter%
  \providecommand\color[2][]{%
    \errmessage{(Inkscape) Color is used for the text in Inkscape, but the package 'color.sty' is not loaded}%
    \renewcommand\color[2][]{}%
  }%
  \providecommand\transparent[1]{%
    \errmessage{(Inkscape) Transparency is used (non-zero) for the text in Inkscape, but the package 'transparent.sty' is not loaded}%
    \renewcommand\transparent[1]{}%
  }%
  \providecommand\rotatebox[2]{#2}%
  \newcommand*\fsize{\dimexpr\f@size pt\relax}%
  \newcommand*\lineheight[1]{\fontsize{\fsize}{#1\fsize}\selectfont}%
  \ifx\svgwidth\undefined%
    \setlength{\unitlength}{30.75117737bp}%
    \ifx\svgscale\undefined%
      \relax%
    \else%
      \setlength{\unitlength}{\unitlength * \real{\svgscale}}%
    \fi%
  \else%
    \setlength{\unitlength}{\svgwidth}%
  \fi%
  \global\let\svgwidth\undefined%
  \global\let\svgscale\undefined%
  \makeatother%
  \begin{picture}(1,1.09751899)%
    \lineheight{1}%
    \setlength\tabcolsep{0pt}%
    \put(0,0){\includegraphics[width=\unitlength,page=1]{symn.pdf}}%
    \put(0.07929683,0.50062537){\makebox(0,0)[lt]{\lineheight{1.45000005}\smash{\begin{tabular}[t]{l}$sym_m$\end{tabular}}}}%
    \put(0,0){\includegraphics[width=\unitlength,page=2]{symn.pdf}}%
  \end{picture}%
\endgroup%

%% file: altn.pdf_tex
\begingroup%
  \makeatletter%
  \providecommand\color[2][]{%
    \errmessage{(Inkscape) Color is used for the text in Inkscape, but the package 'color.sty' is not loaded}%
    \renewcommand\color[2][]{}%
  }%
  \providecommand\transparent[1]{%
    \errmessage{(Inkscape) Transparency is used (non-zero) for the text in Inkscape, but the package 'transparent.sty' is not loaded}%
    \renewcommand\transparent[1]{}%
  }%
  \providecommand\rotatebox[2]{#2}%
  \newcommand*\fsize{\dimexpr\f@size pt\relax}%
  \newcommand*\lineheight[1]{\fontsize{\fsize}{#1\fsize}\selectfont}%
  \ifx\svgwidth\undefined%
    \setlength{\unitlength}{30.75117737bp}%
    \ifx\svgscale\undefined%
      \relax%
    \else%
      \setlength{\unitlength}{\unitlength * \real{\svgscale}}%
    \fi%
  \else%
    \setlength{\unitlength}{\svgwidth}%
  \fi%
  \global\let\svgwidth\undefined%
  \global\let\svgscale\undefined%
  \makeatother%
  \begin{picture}(1,1.09751899)%
    \lineheight{1}%
    \setlength\tabcolsep{0pt}%
    \put(0,0){\includegraphics[width=\unitlength,page=1]{altn.pdf}}%
    \put(0.22078711,0.48330002){\makebox(0,0)[lt]{\lineheight{1.45000005}\smash{\begin{tabular}[t]{l}$alt_m$\end{tabular}}}}%
    \put(0,0){\includegraphics[width=\unitlength,page=2]{altn.pdf}}%
  \end{picture}%
\endgroup%

%% file: cc1324.pdf_tex
\begingroup%
  \makeatletter%
  \providecommand\color[2][]{%
    \errmessage{(Inkscape) Color is used for the text in Inkscape, but the package 'color.sty' is not loaded}%
    \renewcommand\color[2][]{}%
  }%
  \providecommand\transparent[1]{%
    \errmessage{(Inkscape) Transparency is used (non-zero) for the text in Inkscape, but the package 'transparent.sty' is not loaded}%
    \renewcommand\transparent[1]{}%
  }%
  \providecommand\rotatebox[2]{#2}%
  \newcommand*\fsize{\dimexpr\f@size pt\relax}%
  \newcommand*\lineheight[1]{\fontsize{\fsize}{#1\fsize}\selectfont}%
  \ifx\svgwidth\undefined%
    \setlength{\unitlength}{19.46367335bp}%
    \ifx\svgscale\undefined%
      \relax%
    \else%
      \setlength{\unitlength}{\unitlength * \real{\svgscale}}%
    \fi%
  \else%
    \setlength{\unitlength}{\svgwidth}%
  \fi%
  \global\let\svgwidth\undefined%
  \global\let\svgscale\undefined%
  \makeatother%
  \begin{picture}(1,0.47293477)%
    \lineheight{1}%
    \setlength\tabcolsep{0pt}%
    \put(0,0){\includegraphics[width=\unitlength,page=1]{cc1324.pdf}}%
  \end{picture}%
\endgroup%

%% file: cc1423.pdf_tex
\begingroup%
  \makeatletter%
  \providecommand\color[2][]{%
    \errmessage{(Inkscape) Color is used for the text in Inkscape, but the package 'color.sty' is not loaded}%
    \renewcommand\color[2][]{}%
  }%
  \providecommand\transparent[1]{%
    \errmessage{(Inkscape) Transparency is used (non-zero) for the text in Inkscape, but the package 'transparent.sty' is not loaded}%
    \renewcommand\transparent[1]{}%
  }%
  \providecommand\rotatebox[2]{#2}%
  \newcommand*\fsize{\dimexpr\f@size pt\relax}%
  \newcommand*\lineheight[1]{\fontsize{\fsize}{#1\fsize}\selectfont}%
  \ifx\svgwidth\undefined%
    \setlength{\unitlength}{15.67079398bp}%
    \ifx\svgscale\undefined%
      \relax%
    \else%
      \setlength{\unitlength}{\unitlength * \real{\svgscale}}%
    \fi%
  \else%
    \setlength{\unitlength}{\svgwidth}%
  \fi%
  \global\let\svgwidth\undefined%
  \global\let\svgscale\undefined%
  \makeatother%
  \begin{picture}(1,0.59024833)%
    \lineheight{1}%
    \setlength\tabcolsep{0pt}%
    \put(0,0){\includegraphics[width=\unitlength,page=1]{cc1423.pdf}}%
  \end{picture}%
\endgroup%

%% file: A22gen132.pdf_tex
\begingroup%
  \makeatletter%
  \providecommand\color[2][]{%
    \errmessage{(Inkscape) Color is used for the text in Inkscape, but the package 'color.sty' is not loaded}%
    \renewcommand\color[2][]{}%
  }%
  \providecommand\transparent[1]{%
    \errmessage{(Inkscape) Transparency is used (non-zero) for the text in Inkscape, but the package 'transparent.sty' is not loaded}%
    \renewcommand\transparent[1]{}%
  }%
  \providecommand\rotatebox[2]{#2}%
  \newcommand*\fsize{\dimexpr\f@size pt\relax}%
  \newcommand*\lineheight[1]{\fontsize{\fsize}{#1\fsize}\selectfont}%
  \ifx\svgwidth\undefined%
    \setlength{\unitlength}{42.75117945bp}%
    \ifx\svgscale\undefined%
      \relax%
    \else%
      \setlength{\unitlength}{\unitlength * \real{\svgscale}}%
    \fi%
  \else%
    \setlength{\unitlength}{\svgwidth}%
  \fi%
  \global\let\svgwidth\undefined%
  \global\let\svgscale\undefined%
  \makeatother%
  \begin{picture}(1,0.83312995)%
    \lineheight{1}%
    \setlength\tabcolsep{0pt}%
    \put(0,0){\includegraphics[width=\unitlength,page=1]{A22gen132.pdf}}%
    \put(0.02565245,0.33620548){\makebox(0,0)[lt]{\lineheight{1.45000005}\smash{\begin{tabular}[t]{l}$alt_2$\end{tabular}}}}%
    \put(0.57684974,0.33620548){\makebox(0,0)[lt]{\lineheight{1.45000005}\smash{\begin{tabular}[t]{l}$alt_2$\end{tabular}}}}%
  \end{picture}%
\endgroup%

%% file: A22gen13.pdf_tex
\begingroup%
  \makeatletter%
  \providecommand\color[2][]{%
    \errmessage{(Inkscape) Color is used for the text in Inkscape, but the package 'color.sty' is not loaded}%
    \renewcommand\color[2][]{}%
  }%
  \providecommand\transparent[1]{%
    \errmessage{(Inkscape) Transparency is used (non-zero) for the text in Inkscape, but the package 'transparent.sty' is not loaded}%
    \renewcommand\transparent[1]{}%
  }%
  \providecommand\rotatebox[2]{#2}%
  \newcommand*\fsize{\dimexpr\f@size pt\relax}%
  \newcommand*\lineheight[1]{\fontsize{\fsize}{#1\fsize}\selectfont}%
  \ifx\svgwidth\undefined%
    \setlength{\unitlength}{42.75117945bp}%
    \ifx\svgscale\undefined%
      \relax%
    \else%
      \setlength{\unitlength}{\unitlength * \real{\svgscale}}%
    \fi%
  \else%
    \setlength{\unitlength}{\svgwidth}%
  \fi%
  \global\let\svgwidth\undefined%
  \global\let\svgscale\undefined%
  \makeatother%
  \begin{picture}(1,0.83312995)%
    \lineheight{1}%
    \setlength\tabcolsep{0pt}%
    \put(0,0){\includegraphics[width=\unitlength,page=1]{A22gen13.pdf}}%
    \put(0.02565245,0.33620548){\makebox(0,0)[lt]{\lineheight{1.45000005}\smash{\begin{tabular}[t]{l}$alt_2$\end{tabular}}}}%
    \put(0.57684974,0.33620548){\makebox(0,0)[lt]{\lineheight{1.45000005}\smash{\begin{tabular}[t]{l}$alt_2$\end{tabular}}}}%
  \end{picture}%
\endgroup%

%% file: symijkl.pdf_tex
\begingroup%
  \makeatletter%
  \providecommand\color[2][]{%
    \errmessage{(Inkscape) Color is used for the text in Inkscape, but the package 'color.sty' is not loaded}%
    \renewcommand\color[2][]{}%
  }%
  \providecommand\transparent[1]{%
    \errmessage{(Inkscape) Transparency is used (non-zero) for the text in Inkscape, but the package 'transparent.sty' is not loaded}%
    \renewcommand\transparent[1]{}%
  }%
  \providecommand\rotatebox[2]{#2}%
  \newcommand*\fsize{\dimexpr\f@size pt\relax}%
  \newcommand*\lineheight[1]{\fontsize{\fsize}{#1\fsize}\selectfont}%
  \ifx\svgwidth\undefined%
    \setlength{\unitlength}{30.75117737bp}%
    \ifx\svgscale\undefined%
      \relax%
    \else%
      \setlength{\unitlength}{\unitlength * \real{\svgscale}}%
    \fi%
  \else%
    \setlength{\unitlength}{\svgwidth}%
  \fi%
  \global\let\svgwidth\undefined%
  \global\let\svgscale\undefined%
  \makeatother%
  \begin{picture}(1,1.27773764)%
    \lineheight{1}%
    \setlength\tabcolsep{0pt}%
    \put(0,0){\includegraphics[width=\unitlength,page=1]{symijkl.pdf}}%
    \put(0.20732822,0.74636054){\makebox(0,0)[lt]{\lineheight{1.45000005}\smash{\begin{tabular}[t]{l}$sym_4$\end{tabular}}}}%
    \put(0.08702019,0.0141558){\makebox(0,0)[lt]{\lineheight{1.45000005}\smash{\begin{tabular}[t]{l}$i$\end{tabular}}}}%
    \put(0.32329939,0.01343318){\makebox(0,0)[lt]{\lineheight{1.45000005}\smash{\begin{tabular}[t]{l}$j$\end{tabular}}}}%
    \put(0.56525924,0.01239268){\makebox(0,0)[lt]{\lineheight{1.45000005}\smash{\begin{tabular}[t]{l}$k$\end{tabular}}}}%
    \put(0.81139436,0.0137094){\makebox(0,0)[lt]{\lineheight{1.45000005}\smash{\begin{tabular}[t]{l}$l$\end{tabular}}}}%
    \put(0,0){\includegraphics[width=\unitlength,page=2]{symijkl.pdf}}%
  \end{picture}%
\endgroup%

%% file: symvijkl.pdf_tex
\begingroup%
  \makeatletter%
  \providecommand\color[2][]{%
    \errmessage{(Inkscape) Color is used for the text in Inkscape, but the package 'color.sty' is not loaded}%
    \renewcommand\color[2][]{}%
  }%
  \providecommand\transparent[1]{%
    \errmessage{(Inkscape) Transparency is used (non-zero) for the text in Inkscape, but the package 'transparent.sty' is not loaded}%
    \renewcommand\transparent[1]{}%
  }%
  \providecommand\rotatebox[2]{#2}%
  \newcommand*\fsize{\dimexpr\f@size pt\relax}%
  \newcommand*\lineheight[1]{\fontsize{\fsize}{#1\fsize}\selectfont}%
  \ifx\svgwidth\undefined%
    \setlength{\unitlength}{31.61114041bp}%
    \ifx\svgscale\undefined%
      \relax%
    \else%
      \setlength{\unitlength}{\unitlength * \real{\svgscale}}%
    \fi%
  \else%
    \setlength{\unitlength}{\svgwidth}%
  \fi%
  \global\let\svgwidth\undefined%
  \global\let\svgscale\undefined%
  \makeatother%
  \begin{picture}(1,1.08544569)%
    \lineheight{1}%
    \setlength\tabcolsep{0pt}%
    \put(0,0){\includegraphics[width=\unitlength,page=1]{symvijkl.pdf}}%
    \put(0.20168797,0.59898024){\makebox(0,0)[lt]{\lineheight{1.45000005}\smash{\begin{tabular}[t]{l}$sym_4$\end{tabular}}}}%
    \put(0.00255826,0.01354177){\makebox(0,0)[lt]{\lineheight{1.45000005}\smash{\begin{tabular}[t]{l}$v_i$\end{tabular}}}}%
    \put(0.28188021,0.01306774){\makebox(0,0)[lt]{\lineheight{1.45000005}\smash{\begin{tabular}[t]{l}$v_j$\end{tabular}}}}%
    \put(0.526938,0.01681943){\makebox(0,0)[lt]{\lineheight{1.45000005}\smash{\begin{tabular}[t]{l}$v_k$\end{tabular}}}}%
    \put(0.82872918,0.01638516){\makebox(0,0)[lt]{\lineheight{1.45000005}\smash{\begin{tabular}[t]{l}$v_l$\end{tabular}}}}%
  \end{picture}%
\endgroup%

%% file: altijkl.pdf_tex
\begingroup%
  \makeatletter%
  \providecommand\color[2][]{%
    \errmessage{(Inkscape) Color is used for the text in Inkscape, but the package 'color.sty' is not loaded}%
    \renewcommand\color[2][]{}%
  }%
  \providecommand\transparent[1]{%
    \errmessage{(Inkscape) Transparency is used (non-zero) for the text in Inkscape, but the package 'transparent.sty' is not loaded}%
    \renewcommand\transparent[1]{}%
  }%
  \providecommand\rotatebox[2]{#2}%
  \newcommand*\fsize{\dimexpr\f@size pt\relax}%
  \newcommand*\lineheight[1]{\fontsize{\fsize}{#1\fsize}\selectfont}%
  \ifx\svgwidth\undefined%
    \setlength{\unitlength}{42.75117945bp}%
    \ifx\svgscale\undefined%
      \relax%
    \else%
      \setlength{\unitlength}{\unitlength * \real{\svgscale}}%
    \fi%
  \else%
    \setlength{\unitlength}{\svgwidth}%
  \fi%
  \global\let\svgwidth\undefined%
  \global\let\svgscale\undefined%
  \makeatother%
  \begin{picture}(1,1.01270299)%
    \lineheight{1}%
    \setlength\tabcolsep{0pt}%
    \put(0,0){\includegraphics[width=\unitlength,page=1]{altijkl.pdf}}%
    \put(0.02565245,0.51577852){\makebox(0,0)[lt]{\lineheight{1.45000005}\smash{\begin{tabular}[t]{l}$alt_2$\end{tabular}}}}%
    \put(0.57684974,0.51577852){\makebox(0,0)[lt]{\lineheight{1.45000005}\smash{\begin{tabular}[t]{l}$alt_2$\end{tabular}}}}%
    \put(0.05647505,0.01348056){\makebox(0,0)[lt]{\lineheight{1.45000005}\smash{\begin{tabular}[t]{l}$i$\end{tabular}}}}%
    \put(0.28921515,0.0103281){\makebox(0,0)[lt]{\lineheight{1.45000005}\smash{\begin{tabular}[t]{l}$j$\end{tabular}}}}%
    \put(0.63600086,0.01092262){\makebox(0,0)[lt]{\lineheight{1.45000005}\smash{\begin{tabular}[t]{l}$k$\end{tabular}}}}%
    \put(0.86863185,0.00824631){\makebox(0,0)[lt]{\lineheight{1.45000005}\smash{\begin{tabular}[t]{l}$l$\end{tabular}}}}%
    \put(0,0){\includegraphics[width=\unitlength,page=2]{altijkl.pdf}}%
  \end{picture}%
\endgroup%

%% file: altvijkl.pdf_tex
\begingroup%
  \makeatletter%
  \providecommand\color[2][]{%
    \errmessage{(Inkscape) Color is used for the text in Inkscape, but the package 'color.sty' is not loaded}%
    \renewcommand\color[2][]{}%
  }%
  \providecommand\transparent[1]{%
    \errmessage{(Inkscape) Transparency is used (non-zero) for the text in Inkscape, but the package 'transparent.sty' is not loaded}%
    \renewcommand\transparent[1]{}%
  }%
  \providecommand\rotatebox[2]{#2}%
  \newcommand*\fsize{\dimexpr\f@size pt\relax}%
  \newcommand*\lineheight[1]{\fontsize{\fsize}{#1\fsize}\selectfont}%
  \ifx\svgwidth\undefined%
    \setlength{\unitlength}{42.75117945bp}%
    \ifx\svgscale\undefined%
      \relax%
    \else%
      \setlength{\unitlength}{\unitlength * \real{\svgscale}}%
    \fi%
  \else%
    \setlength{\unitlength}{\svgwidth}%
  \fi%
  \global\let\svgwidth\undefined%
  \global\let\svgscale\undefined%
  \makeatother%
  \begin{picture}(1,0.90580655)%
    \lineheight{1}%
    \setlength\tabcolsep{0pt}%
    \put(0,0){\includegraphics[width=\unitlength,page=1]{altvijkl.pdf}}%
    \put(0.02565245,0.43537995){\makebox(0,0)[lt]{\lineheight{1.45000005}\smash{\begin{tabular}[t]{l}$alt_2$\end{tabular}}}}%
    \put(0.57684974,0.43537995){\makebox(0,0)[lt]{\lineheight{1.45000005}\smash{\begin{tabular}[t]{l}$alt_2$\end{tabular}}}}%
    \put(0.05925393,0.00950094){\makebox(0,0)[lt]{\lineheight{1.45000005}\smash{\begin{tabular}[t]{l}$v_i$\end{tabular}}}}%
    \put(0.27912692,0.00966257){\makebox(0,0)[lt]{\lineheight{1.45000005}\smash{\begin{tabular}[t]{l}$v_j$\end{tabular}}}}%
    \put(0.6318326,0.01250076){\makebox(0,0)[lt]{\lineheight{1.45000005}\smash{\begin{tabular}[t]{l}$v_k$\end{tabular}}}}%
    \put(0.87141073,0.00982445){\makebox(0,0)[lt]{\lineheight{1.45000005}\smash{\begin{tabular}[t]{l}$v_l$\end{tabular}}}}%
  \end{picture}%
\endgroup%

%% file: B2021.pdf_tex
\begingroup%
  \makeatletter%
  \providecommand\color[2][]{%
    \errmessage{(Inkscape) Color is used for the text in Inkscape, but the package 'color.sty' is not loaded}%
    \renewcommand\color[2][]{}%
  }%
  \providecommand\transparent[1]{%
    \errmessage{(Inkscape) Transparency is used (non-zero) for the text in Inkscape, but the package 'transparent.sty' is not loaded}%
    \renewcommand\transparent[1]{}%
  }%
  \providecommand\rotatebox[2]{#2}%
  \newcommand*\fsize{\dimexpr\f@size pt\relax}%
  \newcommand*\lineheight[1]{\fontsize{\fsize}{#1\fsize}\selectfont}%
  \ifx\svgwidth\undefined%
    \setlength{\unitlength}{64.05482107bp}%
    \ifx\svgscale\undefined%
      \relax%
    \else%
      \setlength{\unitlength}{\unitlength * \real{\svgscale}}%
    \fi%
  \else%
    \setlength{\unitlength}{\svgwidth}%
  \fi%
  \global\let\svgwidth\undefined%
  \global\let\svgscale\undefined%
  \makeatother%
  \begin{picture}(1,0.03151534)%
    \lineheight{1}%
    \setlength\tabcolsep{0pt}%
    \put(0,0){\includegraphics[width=\unitlength,page=1]{B2021.pdf}}%
    \put(-0.00236309,0.00652332){\makebox(0,0)[lt]{\lineheight{1.45000005}\smash{\begin{tabular}[t]{l}$v_1$\end{tabular}}}}%
    \put(0.36632651,0.00550371){\makebox(0,0)[lt]{\lineheight{1.45000005}\smash{\begin{tabular}[t]{l}$v_1$\end{tabular}}}}%
    \put(0.52397177,0.00741107){\makebox(0,0)[lt]{\lineheight{1.45000005}\smash{\begin{tabular}[t]{l}$v_2$\end{tabular}}}}%
    \put(0.90419578,0.00748805){\makebox(0,0)[lt]{\lineheight{1.45000005}\smash{\begin{tabular}[t]{l}$v_2$\end{tabular}}}}%
  \end{picture}%
\endgroup%

%% file: B2022.pdf_tex
\begingroup%
  \makeatletter%
  \providecommand\color[2][]{%
    \errmessage{(Inkscape) Color is used for the text in Inkscape, but the package 'color.sty' is not loaded}%
    \renewcommand\color[2][]{}%
  }%
  \providecommand\transparent[1]{%
    \errmessage{(Inkscape) Transparency is used (non-zero) for the text in Inkscape, but the package 'transparent.sty' is not loaded}%
    \renewcommand\transparent[1]{}%
  }%
  \providecommand\rotatebox[2]{#2}%
  \newcommand*\fsize{\dimexpr\f@size pt\relax}%
  \newcommand*\lineheight[1]{\fontsize{\fsize}{#1\fsize}\selectfont}%
  \ifx\svgwidth\undefined%
    \setlength{\unitlength}{64.05482107bp}%
    \ifx\svgscale\undefined%
      \relax%
    \else%
      \setlength{\unitlength}{\unitlength * \real{\svgscale}}%
    \fi%
  \else%
    \setlength{\unitlength}{\svgwidth}%
  \fi%
  \global\let\svgwidth\undefined%
  \global\let\svgscale\undefined%
  \makeatother%
  \begin{picture}(1,0.03151534)%
    \lineheight{1}%
    \setlength\tabcolsep{0pt}%
    \put(0,0){\includegraphics[width=\unitlength,page=1]{B2022.pdf}}%
    \put(-0.00236309,0.00652332){\makebox(0,0)[lt]{\lineheight{1.45000005}\smash{\begin{tabular}[t]{l}$v_1$\end{tabular}}}}%
    \put(0.36632651,0.00550371){\makebox(0,0)[lt]{\lineheight{1.45000005}\smash{\begin{tabular}[t]{l}$v_2$\end{tabular}}}}%
    \put(0.52397177,0.00741107){\makebox(0,0)[lt]{\lineheight{1.45000005}\smash{\begin{tabular}[t]{l}$v_1$\end{tabular}}}}%
    \put(0.90419578,0.00748805){\makebox(0,0)[lt]{\lineheight{1.45000005}\smash{\begin{tabular}[t]{l}$v_2$\end{tabular}}}}%
  \end{picture}%
\endgroup%

%% file: u_123.pdf_tex
\begingroup%
  \makeatletter%
  \providecommand\color[2][]{%
    \errmessage{(Inkscape) Color is used for the text in Inkscape, but the package 'color.sty' is not loaded}%
    \renewcommand\color[2][]{}%
  }%
  \providecommand\transparent[1]{%
    \errmessage{(Inkscape) Transparency is used (non-zero) for the text in Inkscape, but the package 'transparent.sty' is not loaded}%
    \renewcommand\transparent[1]{}%
  }%
  \providecommand\rotatebox[2]{#2}%
  \newcommand*\fsize{\dimexpr\f@size pt\relax}%
  \newcommand*\lineheight[1]{\fontsize{\fsize}{#1\fsize}\selectfont}%
  \ifx\svgwidth\undefined%
    \setlength{\unitlength}{37.2932625bp}%
    \ifx\svgscale\undefined%
      \relax%
    \else%
      \setlength{\unitlength}{\unitlength * \real{\svgscale}}%
    \fi%
  \else%
    \setlength{\unitlength}{\svgwidth}%
  \fi%
  \global\let\svgwidth\undefined%
  \global\let\svgscale\undefined%
  \makeatother%
  \begin{picture}(1,0.72197429)%
    \lineheight{1}%
    \setlength\tabcolsep{0pt}%
    \put(-0.00405884,0.00945316){\makebox(0,0)[lt]{\lineheight{1.45000005}\smash{\begin{tabular}[t]{l}$v_1$\end{tabular}}}}%
    \put(0.40665432,0.01638878){\makebox(0,0)[lt]{\lineheight{1.45000005}\smash{\begin{tabular}[t]{l}$v_2$\end{tabular}}}}%
    \put(0.83544689,0.01214915){\makebox(0,0)[lt]{\lineheight{1.45000005}\smash{\begin{tabular}[t]{l}$v_3$\end{tabular}}}}%
    \put(0,0){\includegraphics[width=\unitlength,page=1]{u_123.pdf}}%
  \end{picture}%
\endgroup%

%% file: B22.pdf_tex
\begingroup%
  \makeatletter%
  \providecommand\color[2][]{%
    \errmessage{(Inkscape) Color is used for the text in Inkscape, but the package 'color.sty' is not loaded}%
    \renewcommand\color[2][]{}%
  }%
  \providecommand\transparent[1]{%
    \errmessage{(Inkscape) Transparency is used (non-zero) for the text in Inkscape, but the package 'transparent.sty' is not loaded}%
    \renewcommand\transparent[1]{}%
  }%
  \providecommand\rotatebox[2]{#2}%
  \newcommand*\fsize{\dimexpr\f@size pt\relax}%
  \newcommand*\lineheight[1]{\fontsize{\fsize}{#1\fsize}\selectfont}%
  \ifx\svgwidth\undefined%
    \setlength{\unitlength}{40.78738117bp}%
    \ifx\svgscale\undefined%
      \relax%
    \else%
      \setlength{\unitlength}{\unitlength * \real{\svgscale}}%
    \fi%
  \else%
    \setlength{\unitlength}{\svgwidth}%
  \fi%
  \global\let\svgwidth\undefined%
  \global\let\svgscale\undefined%
  \makeatother%
  \begin{picture}(1,0.20229739)%
    \lineheight{1}%
    \setlength\tabcolsep{0pt}%
    \put(0,0){\includegraphics[width=\unitlength,page=1]{B22.pdf}}%
    \put(-0.00371113,0.05753813){\makebox(0,0)[lt]{\lineheight{1.45000005}\smash{\begin{tabular}[t]{l}$v_1$\end{tabular}}}}%
    \put(0.8495436,0.05199924){\makebox(0,0)[lt]{\lineheight{1.45000005}\smash{\begin{tabular}[t]{l}$v_1$\end{tabular}}}}%
    \put(0,0){\includegraphics[width=\unitlength,page=2]{B22.pdf}}%
  \end{picture}%
\endgroup%

%% file: a221.pdf_tex
\begingroup%
  \makeatletter%
  \providecommand\color[2][]{%
    \errmessage{(Inkscape) Color is used for the text in Inkscape, but the package 'color.sty' is not loaded}%
    \renewcommand\color[2][]{}%
  }%
  \providecommand\transparent[1]{%
    \errmessage{(Inkscape) Transparency is used (non-zero) for the text in Inkscape, but the package 'transparent.sty' is not loaded}%
    \renewcommand\transparent[1]{}%
  }%
  \providecommand\rotatebox[2]{#2}%
  \newcommand*\fsize{\dimexpr\f@size pt\relax}%
  \newcommand*\lineheight[1]{\fontsize{\fsize}{#1\fsize}\selectfont}%
  \ifx\svgwidth\undefined%
    \setlength{\unitlength}{21.76448206bp}%
    \ifx\svgscale\undefined%
      \relax%
    \else%
      \setlength{\unitlength}{\unitlength * \real{\svgscale}}%
    \fi%
  \else%
    \setlength{\unitlength}{\svgwidth}%
  \fi%
  \global\let\svgwidth\undefined%
  \global\let\svgscale\undefined%
  \makeatother%
  \begin{picture}(1,1.99154066)%
    \lineheight{1}%
    \setlength\tabcolsep{0pt}%
    \put(0,0){\includegraphics[width=\unitlength,page=1]{a221.pdf}}%
    \put(0.38566051,0.03239579){\makebox(0,0)[lt]{\lineheight{1.45000005}\smash{\begin{tabular}[t]{l}$1$\end{tabular}}}}%
    \put(0,0){\includegraphics[width=\unitlength,page=2]{a221.pdf}}%
  \end{picture}%
\endgroup%

%% file: A201122.pdf_tex
\begingroup%
  \makeatletter%
  \providecommand\color[2][]{%
    \errmessage{(Inkscape) Color is used for the text in Inkscape, but the package 'color.sty' is not loaded}%
    \renewcommand\color[2][]{}%
  }%
  \providecommand\transparent[1]{%
    \errmessage{(Inkscape) Transparency is used (non-zero) for the text in Inkscape, but the package 'transparent.sty' is not loaded}%
    \renewcommand\transparent[1]{}%
  }%
  \providecommand\rotatebox[2]{#2}%
  \newcommand*\fsize{\dimexpr\f@size pt\relax}%
  \newcommand*\lineheight[1]{\fontsize{\fsize}{#1\fsize}\selectfont}%
  \ifx\svgwidth\undefined%
    \setlength{\unitlength}{50.22484093bp}%
    \ifx\svgscale\undefined%
      \relax%
    \else%
      \setlength{\unitlength}{\unitlength * \real{\svgscale}}%
    \fi%
  \else%
    \setlength{\unitlength}{\svgwidth}%
  \fi%
  \global\let\svgwidth\undefined%
  \global\let\svgscale\undefined%
  \makeatother%
  \begin{picture}(1,0.72208727)%
    \lineheight{1}%
    \setlength\tabcolsep{0pt}%
    \put(0,0){\includegraphics[width=\unitlength,page=1]{A201122.pdf}}%
    \put(0.15898375,0.01403842){\makebox(0,0)[lt]{\lineheight{1.45000005}\smash{\begin{tabular}[t]{l}$1$\end{tabular}}}}%
    \put(0.76196547,0.01721687){\makebox(0,0)[lt]{\lineheight{1.45000005}\smash{\begin{tabular}[t]{l}$2$\end{tabular}}}}%
    \put(0,0){\includegraphics[width=\unitlength,page=2]{A201122.pdf}}%
  \end{picture}%
\endgroup%

%% file: A201212.pdf_tex
\begingroup%
  \makeatletter%
  \providecommand\color[2][]{%
    \errmessage{(Inkscape) Color is used for the text in Inkscape, but the package 'color.sty' is not loaded}%
    \renewcommand\color[2][]{}%
  }%
  \providecommand\transparent[1]{%
    \errmessage{(Inkscape) Transparency is used (non-zero) for the text in Inkscape, but the package 'transparent.sty' is not loaded}%
    \renewcommand\transparent[1]{}%
  }%
  \providecommand\rotatebox[2]{#2}%
  \newcommand*\fsize{\dimexpr\f@size pt\relax}%
  \newcommand*\lineheight[1]{\fontsize{\fsize}{#1\fsize}\selectfont}%
  \ifx\svgwidth\undefined%
    \setlength{\unitlength}{50.22484093bp}%
    \ifx\svgscale\undefined%
      \relax%
    \else%
      \setlength{\unitlength}{\unitlength * \real{\svgscale}}%
    \fi%
  \else%
    \setlength{\unitlength}{\svgwidth}%
  \fi%
  \global\let\svgwidth\undefined%
  \global\let\svgscale\undefined%
  \makeatother%
  \begin{picture}(1,0.80944125)%
    \lineheight{1}%
    \setlength\tabcolsep{0pt}%
    \put(0,0){\includegraphics[width=\unitlength,page=1]{A201212.pdf}}%
    \put(0.15606783,0.01983479){\makebox(0,0)[lt]{\lineheight{1.45000005}\smash{\begin{tabular}[t]{l}$1$\end{tabular}}}}%
    \put(0.76984786,0.01403842){\makebox(0,0)[lt]{\lineheight{1.45000005}\smash{\begin{tabular}[t]{l}$2$\end{tabular}}}}%
    \put(0,0){\includegraphics[width=\unitlength,page=2]{A201212.pdf}}%
  \end{picture}%
\endgroup%

%% file: A2012122.pdf_tex
\begingroup%
  \makeatletter%
  \providecommand\color[2][]{%
    \errmessage{(Inkscape) Color is used for the text in Inkscape, but the package 'color.sty' is not loaded}%
    \renewcommand\color[2][]{}%
  }%
  \providecommand\transparent[1]{%
    \errmessage{(Inkscape) Transparency is used (non-zero) for the text in Inkscape, but the package 'transparent.sty' is not loaded}%
    \renewcommand\transparent[1]{}%
  }%
  \providecommand\rotatebox[2]{#2}%
  \newcommand*\fsize{\dimexpr\f@size pt\relax}%
  \newcommand*\lineheight[1]{\fontsize{\fsize}{#1\fsize}\selectfont}%
  \ifx\svgwidth\undefined%
    \setlength{\unitlength}{50.22484093bp}%
    \ifx\svgscale\undefined%
      \relax%
    \else%
      \setlength{\unitlength}{\unitlength * \real{\svgscale}}%
    \fi%
  \else%
    \setlength{\unitlength}{\svgwidth}%
  \fi%
  \global\let\svgwidth\undefined%
  \global\let\svgscale\undefined%
  \makeatother%
  \begin{picture}(1,0.74897786)%
    \lineheight{1}%
    \setlength\tabcolsep{0pt}%
    \put(0,0){\includegraphics[width=\unitlength,page=1]{A2012122.pdf}}%
    \put(0.1676442,0.01425275){\makebox(0,0)[lt]{\lineheight{1.45000005}\smash{\begin{tabular}[t]{l}$1$\end{tabular}}}}%
    \put(0.76723317,0.01403842){\makebox(0,0)[lt]{\lineheight{1.45000005}\smash{\begin{tabular}[t]{l}$2$\end{tabular}}}}%
    \put(0,0){\includegraphics[width=\unitlength,page=2]{A2012122.pdf}}%
  \end{picture}%
\endgroup%

%% file: A2211.pdf_tex
\begingroup%
  \makeatletter%
  \providecommand\color[2][]{%
    \errmessage{(Inkscape) Color is used for the text in Inkscape, but the package 'color.sty' is not loaded}%
    \renewcommand\color[2][]{}%
  }%
  \providecommand\transparent[1]{%
    \errmessage{(Inkscape) Transparency is used (non-zero) for the text in Inkscape, but the package 'transparent.sty' is not loaded}%
    \renewcommand\transparent[1]{}%
  }%
  \providecommand\rotatebox[2]{#2}%
  \newcommand*\fsize{\dimexpr\f@size pt\relax}%
  \newcommand*\lineheight[1]{\fontsize{\fsize}{#1\fsize}\selectfont}%
  \ifx\svgwidth\undefined%
    \setlength{\unitlength}{46.47484282bp}%
    \ifx\svgscale\undefined%
      \relax%
    \else%
      \setlength{\unitlength}{\unitlength * \real{\svgscale}}%
    \fi%
  \else%
    \setlength{\unitlength}{\svgwidth}%
  \fi%
  \global\let\svgwidth\undefined%
  \global\let\svgscale\undefined%
  \makeatother%
  \begin{picture}(1,0.93535461)%
    \lineheight{1}%
    \setlength\tabcolsep{0pt}%
    \put(0,0){\includegraphics[width=\unitlength,page=1]{A2211.pdf}}%
    \put(0.18060742,0.01787386){\makebox(0,0)[lt]{\lineheight{1.45000005}\smash{\begin{tabular}[t]{l}$1$\end{tabular}}}}%
    \put(0.74895116,0.01517117){\makebox(0,0)[lt]{\lineheight{1.45000005}\smash{\begin{tabular}[t]{l}$2$\end{tabular}}}}%
    \put(0,0){\includegraphics[width=\unitlength,page=2]{A2211.pdf}}%
  \end{picture}%
\endgroup%

%% file: u22.pdf_tex
\begingroup%
  \makeatletter%
  \providecommand\color[2][]{%
    \errmessage{(Inkscape) Color is used for the text in Inkscape, but the package 'color.sty' is not loaded}%
    \renewcommand\color[2][]{}%
  }%
  \providecommand\transparent[1]{%
    \errmessage{(Inkscape) Transparency is used (non-zero) for the text in Inkscape, but the package 'transparent.sty' is not loaded}%
    \renewcommand\transparent[1]{}%
  }%
  \providecommand\rotatebox[2]{#2}%
  \newcommand*\fsize{\dimexpr\f@size pt\relax}%
  \newcommand*\lineheight[1]{\fontsize{\fsize}{#1\fsize}\selectfont}%
  \ifx\svgwidth\undefined%
    \setlength{\unitlength}{48.01447908bp}%
    \ifx\svgscale\undefined%
      \relax%
    \else%
      \setlength{\unitlength}{\unitlength * \real{\svgscale}}%
    \fi%
  \else%
    \setlength{\unitlength}{\svgwidth}%
  \fi%
  \global\let\svgwidth\undefined%
  \global\let\svgscale\undefined%
  \makeatother%
  \begin{picture}(1,0.90536151)%
    \lineheight{1}%
    \setlength\tabcolsep{0pt}%
    \put(0,0){\includegraphics[width=\unitlength,page=1]{u22.pdf}}%
    \put(0.72152604,0.01730072){\makebox(0,0)[lt]{\lineheight{1.45000005}\smash{\begin{tabular}[t]{l}$2$\end{tabular}}}}%
    \put(0.1782251,0.01468469){\makebox(0,0)[lt]{\lineheight{1.45000005}\smash{\begin{tabular}[t]{l}$1$\end{tabular}}}}%
    \put(0,0){\includegraphics[width=\unitlength,page=2]{u22.pdf}}%
  \end{picture}%
\endgroup%

%% file: rnk.pdf_tex
\begingroup%
  \makeatletter%
  \providecommand\color[2][]{%
    \errmessage{(Inkscape) Color is used for the text in Inkscape, but the package 'color.sty' is not loaded}%
    \renewcommand\color[2][]{}%
  }%
  \providecommand\transparent[1]{%
    \errmessage{(Inkscape) Transparency is used (non-zero) for the text in Inkscape, but the package 'transparent.sty' is not loaded}%
    \renewcommand\transparent[1]{}%
  }%
  \providecommand\rotatebox[2]{#2}%
  \newcommand*\fsize{\dimexpr\f@size pt\relax}%
  \newcommand*\lineheight[1]{\fontsize{\fsize}{#1\fsize}\selectfont}%
  \ifx\svgwidth\undefined%
    \setlength{\unitlength}{142.87417512bp}%
    \ifx\svgscale\undefined%
      \relax%
    \else%
      \setlength{\unitlength}{\unitlength * \real{\svgscale}}%
    \fi%
  \else%
    \setlength{\unitlength}{\svgwidth}%
  \fi%
  \global\let\svgwidth\undefined%
  \global\let\svgscale\undefined%
  \makeatother%
  \begin{picture}(1,0.29733817)%
    \lineheight{1}%
    \setlength\tabcolsep{0pt}%
    \put(0,0){\includegraphics[width=\unitlength,page=1]{rnk.pdf}}%
    \put(0.4862152,0.00294516){\color[rgb]{0,0,0}\makebox(0,0)[lt]{\lineheight{1.25}\smash{\begin{tabular}[t]{l}$k$\end{tabular}}}}%
    \put(0.07060026,0.00314572){\color[rgb]{0,0,0}\makebox(0,0)[lt]{\lineheight{1.25}\smash{\begin{tabular}[t]{l}$1$\end{tabular}}}}%
    \put(0.83650838,0.00246748){\color[rgb]{0,0,0}\makebox(0,0)[lt]{\lineheight{1.25}\smash{\begin{tabular}[t]{l}$2d+1$\end{tabular}}}}%
  \end{picture}%
\endgroup%